%% file: main.tex
\documentclass{amsart}

\input{preamble.tex}

\title{Absolutely Continuous Furstenberg Measures}
\author{Samuel Kittle}
\address{Department of Mathematics\\
University College London\\
25 Gordon Street (UCL Union Building)\\
London WC1H 0AY\\
United Kingdom}
\email{s.kittle@ucl.ac.uk}
\date{\today}

\begin{document}

\thanks{The author has received funding from the European Research Council (ERC) under the European Union’s Horizon 2020 research and innovation program (grant agreement No. 803711). The Author has also received funding from a London Mathematical Society ECF. The author has also received funding from the Heilbronn Institute for Mathematical Research.}
\begin{abstract}
In this paper we provide a sufficient condition for a Furstenberg measure generated by a finitely supported measure to be absolutely continuous. Using this we give completely explicit examples of absolutely continuous Furstenberg measures including examples which are generated by measures which are not symmetric.
\end{abstract}

\maketitle
\setcounter{tocdepth}{1}
\tableofcontents

\section{Introduction} \label{section:introduction}

In this paper we find a sufficient condition for a Furstenberg measure to be absolutely continuous. Using this we are able to give explicit examples of measures $\mu$ on $\Gp$ supported on finitely many points - including examples supported on only two points - such that the Furstenberg measure $\nu$ on $\B$ generated by $\mu$ is absolutely continuous. We are able to give much broader classes of examples than are given in earlier works such as \cite{BOURGAIN_2012}. In particular we do not require $\mu$ to be symmetric.

Given a measure $\mu$ on $\Gp$ we say that a measure $\nu$ on $\B$ is a Furstenberg measure generated by $\mu$ if $\nu$ is stationary under action by $\mu$. In other words we require
\begin{equation*}
\nu = \mu * \nu
\end{equation*}
where $*$ denotes convolution. It is a theorem of Furstenberg in \cite{FURSTENBERG_KIFER_1983} that if $\mu$ is strongly irreducible and the support of $\mu$ is not contained in a compact subgroup of $\Gp$ then there is a unique Furstenberg measure generated by $\mu$. Throughout this paper we will only be concerned with the case where $\mu$ is supported on finitely many points.

Furstenberg measures have many similarities with self-similar measures. A probability measure $\lambda$ on $\R^d$ is self-similar if there are similarities $S_1, S_2, \dots, S_n : \R^d \to \R^d$ and a probability vector $(p_1, p_2, \dots, p_n)$ such that
\begin{equation*}
\lambda = \sum_{i=1}^n p_i \lambda \circ S_i^{-1}.
\end{equation*}
Some important recent developments in the study of self-similar measures and their dimensions can be found in for example \cite{SOLOMYAK_1995}, \cite{HOCHMAN_2014}, \cite{SHMERKIN_2019}, \cite{VARJU_2019} or \cite{Kittle2024ASENS}.

Two fundamental questions about Furstenberg measures are what are their dimensions? And when are they absolutely continuous?

It is a classical result by Guivarc'h \cite{guivarc1990produits} that if $\mu$ is strongly irreducible and the support of $\mu$ is not contained in a compact subgroup of $\Gp$ and there is some $\varepsilon>0$ such that $\int \| g \|^{\varepsilon} \, d \mu(g) < \infty$ then there exist $C, \delta > 0$ such that if we let $\nu$ be the Furstenberg measure generated by $\mu$, let $x \in \B$ and let $r > 0$ then
\begin{equation*}
\nu(B(x, r))  \leq C r^{\delta}
\end{equation*}
where $B(x, r)$ is the open ball in $\B$ centre $x$ and radius $r$. This implies in particular that under these conditions $\nu$ has positive dimension.

In \cite{KAIMANOVICH_LE_PRINCE_2011} it was conjectured that if $\mu$ is supported on finitely many points then its Furstenberg measure $\nu$ is singular.  This conjecture was disproved by B\'ar\'any, Pollicott, and Simon in \cite{BARANY_POLLICOTT_SIMON_2012} which gave a probabilistic construction of measures $\mu$ on $\Gp$ supported on finitely many points with absolutely continuous Furstenberg measures. A variant of this conjecture that also requires $\mu$ to be supported on a discrete subgroup of $\Gp$ remains open.

In \cite{BOURGAIN_2012} Bourgain gives examples of measures $\mu$ on $\Gp$ supported on finitely many points such that the Furstenberg measure generated by $\mu$  is absolutely continuous.

In \cite{HOCHMAN_SOLOMYAK_2017}, building on the work of Hochman in \cite{HOCHMAN_2014}, Hochman and Solomyak show that providing $\mu$ satisfies some exponential separation condition then its Furstenberg measure $\nu$ satisfies
\begin{equation*}
\dim \nu = \min \left\{ \frac{h_{RW}}{2 \chi}, 1 \right\}
\end{equation*}
where $h_{RW}$ is the random walk entropy and $\chi$ is the Lypanov exponent. In particular, they show that if $\mu$ satisfies some exponential separation condition and
\begin{equation*}
\frac{h_{RW}}{\chi} \geq 2
\end{equation*}
then $\nu$ has dimension $1$. In this paper we will show that there is some $C$ which depends on, amongst other things, the rate of the exponential separation such that if
\begin{equation*}
\frac{h_{RW}}{\chi} \geq C
\end{equation*}
then $\nu$ is absolutely continuous. The result we end up with is similar to the result of Varj\'u in \cite[Theorem 1]{VARJU_2019} but applies to Furstenberg measures rather than Bernoulli convolutions. Our techniques are somewhat inspired by those of Hochman \cite{HOCHMAN_2014}, Hochman and Solomyak \cite{HOCHMAN_SOLOMYAK_2017}, and Varj\'u \cite{VARJU_2019} but we introduce several crucial new ingredients including, amongst other things, the concept of ``detail" from \cite{Kittle2024ASENS}.

\subsection{Main results}

We now state our result on the absolute continuity of Furstenberg measures. To do this we first need some definitions.

\begin{definition} \label{definition:strongly_irreducible}
Let $\mu$ be a probability measure on $\Gp$. We say that $\mu$ is strongly irreducible if there is no finite set $S \subset \B$ which is invariant when acted upon by the support of $\mu$.
\end{definition}

\begin{definition} \label{definition:Lyapunov_exponent}
Given a measure $\mu$ on $\Gp$ we define the \emph{Lyapunov exponent} of $\mu$ to be given by the almost sure limit
\begin{equation*}
\chi := \lim_{n \to \infty} \frac{1}{n} \log \lnb \gamma_1 \gamma_2 \dots \gamma_n \rnb
\end{equation*}
where $\gamma_1, \gamma_2, \dots$ are i.i.d.\ samples from $\mu$.
\end{definition}

It is a result of Furstenberg and Kesten \cite{FURSTENBERG_KESTEN_1960} and Furstenberg \cite{Furstenberg1971} that if $$\int \log \| g \| \, d \mu(g) < \infty,$$ $\mu$ is strongly irreducible and its support is not contained in a compact subgroup of $\Gp$ then this limit exists almost surely and is positive.

Note that $\mu$ being strongly irreducible and its support not being contained in a compact subgroup is equivalent to the support of $\mu$ generating a Zariski-dense semigroup. Therefore, using the notation of \cite{BenoistQuint2016}, we will refer to such measures as \emph{Zariski-dense} measures.

Throughout this paper given some $g \in \Gp$ we will write $\lnb g \rnb$ to mean the operator norm of $\hat{g}$ where $\hat{g} \in \Sl_2(\R)$ is some representative of $g$. Note that this does not depend on our choice of $\hat{g}$. We will also fix some left invariant Riemannian metric on $\Gp$ and let $d$ be its distance function. We then have the following definition.

\begin{definition} \label{definition:splitting_rate}
Let $\mu$ be a discrete measure on $\Gp$ supported on finitely many points. Let
\begin{equation*}
S_n := \bigcup_{i=1}^{n} \supp (\mu^{*i}).
\end{equation*}
Then we define the \emph{splitting rate} of $\mu$, which we will denote by $M_{\mu}$, by
\begin{equation*}
M_{\mu} := \exp \left( \limsup \limits_{x, y \in S_n, x \neq y} -\frac{1}{n} \log d(x, y) \right).
\end{equation*}
\end{definition}

Note that all left invariant Riemannian metrics are equivalent and therefore $M_{\mu}$ does not depend on our choice of Riemannian metric. We also need to define the following.

\begin{definition} \label{definition:phi}
We define the bijective function $\phi$ by
\begin{align*}
\phi: \B & \to \Bim\\
\left[ \begin{pmatrix}
\cos x \\ 
\sin x
\end{pmatrix} \right]
& \mapsto x.
\end{align*}
\end{definition}

We now define the following quantitative non-degeneracy condition.

\begin{definition} \label{definition:alpha_0_t_non_degenerate}
Given some probability measure $\mu$ on $\Gp$ generating a Furstenberg measure $\nu$ on $\B$ and given some $\alpha_0, t > 0$ we say that $\mu$ is \emph{$\alpha_0, t$-non-degenerate} if whenever $a \in \R$ we have
\begin{equation*}
\nu ( \phi^{-1} ([a, a+t] + \pi \Z)) \leq \alpha_0. 
\end{equation*} 
\end{definition}

This just says that each arc of length $t$ has $\nu$ measure at most $\alpha_0$. We now have everything needed to state the our new result on the absolute continuity of Furstenberg measures.

\begin{theorem} \label{theorem:main_furstenberg}
For all $R > 1$, $\alpha_0 \in (0, \frac{1}{3})$ and $t >0$ there is some $C > 0$ such that the following holds. Suppose that $\mu$ is a probability measure on $\Gp$ which is Zariski-dense, $\alpha_0, t$- non-degenerate, and is such that on the support of $\mu$ the operator norm is at most $R$. Suppose that $M_{\mu} < \infty$ and
\begin{equation}
\frac{h_{RW}}{ \chi}   > C \left( \max \left \{ 1 , \log \frac{\log M_{\mu}}{h_{RW}} \right \} \right)^{2} . \label{eq:main_theorem_condition}
\end{equation}
Then the Furstenberg measure $\nu$ on $\B$ generated by $\mu$ is absolutely continuous.
\end{theorem}

The constant $C$ can be computed by following the proof. 
\begin{remark}
The condition $M_{\mu} < \infty$ is closely related to the exponential separation condition in \cite{HOCHMAN_SOLOMYAK_2017}. Indeed in \cite{HOCHMAN_SOLOMYAK_2017} Hochman and Solomyak prove that if
\begin{equation*}
\limsup \limits_{x, y \in \supp (\mu^{*n}), x \neq y} -\frac{1}{n} \log d(x, y) < \infty
\end{equation*}
and $\frac{h_{RW}}{\chi} \geq 2$ then the Furstenberg measure has dimension $1$.
\end{remark}

We will now discuss how this result compares to previously existing results.

As we mentioned above, Bourgain \cite{BOURGAIN_2012} gave examples of absolutely continuous Furstenberg measures generated by measures on $\Gp$ supported on finitely many points. Bourgain was able to construct examples with density function in $C^r$ for every finite $r>0$. His approach was revisited by several authors including Benoist and Quint \cite{BenoistQuint2018}, Boutonnet, Ioana and Golsefidy \cite{BOUTONNET_IOANA_GOLSEFIDY_2017}, Lequen \cite{LEQUEN_2022}, and Kogler \cite{KOGLER_2022}. We quote the following result from \cite{KOGLER_2022}.

\begin{theorem} \label{theorem:kogler}
For every $c_1, c_2 > 0$ and $m \in \N$ there is some positive $\varepsilon_0 = \varepsilon_0(m, c_1, c_2)$ such that the following holds. Suppose that $\varepsilon \leq \varepsilon_0$ and let $\mu$ be a symmetric probability measure on $\Gp$ such that
\begin{equation}
\mu^{*n} \left( B_{\varepsilon^{c_1 n}}(H) \right) \leq \varepsilon^{c_2 n} \label{eq:weakly_diophantine_kogler}
\end{equation}
for all proper closed connected subgroups $H < \Gp$ and all sufficiently large $n$. Suppose further that
\begin{equation}
\supp \mu \subset B_{\varepsilon} ( \id). \label{eq:near_id_kogler}
\end{equation}
Then the Furstenberg measure generated by $\mu$ is absolutely continuous with $m$-times continuously differentiable density function.
\end{theorem}

Here $B_{\varepsilon}(\cdot)$ denotes $\varepsilon$-neighbourhood of a set with respect to our left invariant Riemannian metric.
 
The conditions of this theorem are not directly comparable to ours but they are related. Condition \eqref{eq:weakly_diophantine_kogler} can be verified for $H = \{ \id \}$ if $M_{\mu} \leq \varepsilon^{-c_1}$ and $\mu^{*n}(\id) \leq \varepsilon^{c_2 n}$ for all sufficiently large $n$. If that is the case then $h_{RW} \geq c_2 \log \varepsilon^{-1}$. When condition \eqref{eq:near_id_kogler} holds we must have $\chi \leq O(\varepsilon)$. Informally speaking the conditions \eqref{eq:weakly_diophantine_kogler} and \eqref{eq:near_id_kogler} correspond to $M_{\mu} \leq \varepsilon^{-c_1}$,  $h_{RW} \geq c_2 \log \varepsilon^{-1}$, and $\chi \leq O(\varepsilon)$. In comparison condition \eqref{eq:main_theorem_condition} in Theorem \ref{theorem:main_furstenberg} is satisfied if $M_{\mu} \leq \exp \left( \exp \left( c \varepsilon^{-1/2} \right) \right)$, $h_{RW} \geq c$, and $\chi \leq \varepsilon$ for some suitably small $c > 0$.

It is important to note however, that Theorem \ref{theorem:kogler} gives higher regularity for the Furstenberg measure than our result.

To demonstrate the applicability of our result we give several examples of measures satisfying the conditions of Theorem \ref{theorem:main_furstenberg}. We will prove that these examples satisfy the conditions of Theorem \ref{theorem:main_furstenberg} in Section \ref{section:examples}.

\begin{definition}[Height]
Let $\alpha_1$ be an algebraic number of degree $d$ with algebraic conjugates $\alpha_2, \alpha_3, \dots, \alpha_d$. Suppose that the minimal polynomial for $\alpha_1$ over $\Z[X]$ has positive leading coefficient $a_0$. Then we define the \emph{height} of $\alpha_1$ by
\begin{equation*}
\mathcal{H}(\alpha_1) := \left( a_0 \prod_{i=1}^{n} \max \{ 1, |\alpha_i| \} \right)^{1/d}.
\end{equation*}
\end{definition}

Note that the height of a rational number is the maximum of the absolute values of its numerator and denominator. Also note that the height of an algebraic number is the $d$th root of its Mahler measure.

\begin{corollary} \label{corollary:symmetric_examples}
For every $A > 0$ there is some $C > 0$ such that the following is true. Let $r > 0$ be sufficiently small (depending on $A$) and let $\mu$ be a finitely supported symmetric probability measure on $\Gp$. Suppose that all the entries of the matrices in the support of $\mu$ are algebraic and that the support of $\mu$ is not contained in any compact subgroup of $\Gp$. Let $M$ be the greatest of the heights of these entries and let $k$ be the degree of the number field generated by these entries.

Let $U$ be a random variable taking values in $\A$ such that $\lnb U \rnb \leq r$ almost surely, $\exp(U)$ has law $\mu$, and the smallest eigenvalue of the covariance matrix of $U$ is at least $A r^2$.

Suppose that for any virtually solvable group $H < \Gp$ we have $\mu(H) \leq 1/2$.

Suppose further that
\begin{equation*}
r \leq C \left( \log  k + \log  \log (M + 10)  \right)^{-2}.
\end{equation*}
Then the Furstenberg measure generated by $\mu$ is absolutely continuous.
\end{corollary}

In the above corollary we can replace the requirement that $\mu$ is symmetric with the requirement $\|\mathbb{E} [U]\| < c r^2$ for any $c>0$. We can also replace the requirement $\mu(H) \leq 1/2$ with $\mu(H) \leq 1 - \varepsilon$ for any $\varepsilon > 0$. If we do this then we must allow $C$ to also depend on $c$ and $\varepsilon$.

Unlike examples based on the methods of Bourgain we do not require the support of $\mu$ to be close to the identity. We may prove the following.

\begin{corollary} \label{corollary:large_example}
For all $r>0$ there exists some Zariski-dense finitely supported probability measure $\mu$ on $\Gp$ such that all the elements in the support of $\mu$ are conjugate to a diagonal matrix with largest entry at least $r$ under conjugation by a rotation and the Furstenberg measure generated by $\mu$ is absolutely continuous.
\end{corollary}

We also have the following family of examples supported on two elements.

\begin{corollary} \label{corollary:two_gens_ac}
For all sufficiently large $n \in \N$ the following is true.

Let $A \in \Gp$ be defined by
\begin{equation*}
A := \begin{pmatrix}
\frac{n^2-1}{n^2+1} & -\frac{2n}{n^2+1} \\
\frac{2n}{n^2+1} & \frac{n^2-1}{n^2+1}
\end{pmatrix}
\end{equation*}
and let $B \in \Gp$ be defined by
\begin{equation*}
B := \begin{pmatrix}
\frac{n^3 + 1}{n^3} & 0\\
0 & \frac{n^3}{n^3 + 1}
\end{pmatrix}.
\end{equation*}
Let $\mu = \frac{1}{2} \delta_A + \frac{1}{2} \delta_B$. Then $\mu$ is Zariski-dense and the Furstenberg measure generated by $\mu$ is absolutely continuous.
\end{corollary}

\subsection{Outline of the proof} \label{section:outline}

We will now give an overview of the proof of Theorem \ref{theorem:main_furstenberg}. We adapt the concept of detail from \cite{Kittle2024ASENS} to work with measures on $\B$ or equivalently $\Bim$ instead of measures on $\R$. The detail of a measure $\lambda$ around scale $r$, denoted by $s_r(\lambda)$, is a quantitative measure of how smooth a measure is at scale $r$. We will define this in Definition \ref{definition:detail}. We then need the following result

\begin{lemma} \label{lemma:suff_abs_cont}
Suppose that $\lambda$ is a probability measure on $\B$ and that there exists some constant $\beta>1$ such that for all sufficiently small $r>0$ we have
\begin{equation*}
s_r(\lambda) < \left( \log r^{-1} \right) ^{-\beta}.
\end{equation*}
Then $\lambda$ is absolutely continuous.
\end{lemma}

A similar result for measures on $\R$ is proven in \cite[Lemma 1.18]{Kittle2024ASENS}. The same proof works for measures on $\Bim$.

In Definition \ref{definition:order_k_detail} we introduce a new quantity for measuring how smooth a measure is at some scale $r>0$ which we will call order $k$ detail around scale $r$ and denote by $s_r^{(k)}(\cdot)$. The definition is chosen such that trivially we have
\begin{equation}
s_r^{(k)}(\lambda_1 * \lambda_2 * \dots * \lambda_k) \leq s_r(\lambda_1) s_r(\lambda_2) \dots s_r(\lambda_k). \label{eq:trivial_srk_bound}
\end{equation}

We can also bound detail in terms of order $k$ detail using the following lemma.

\begin{lemma} \label{lemma:ind_srk_to_sr}
Let $k$ be an integer greater than $1$ and suppose that $\lambda$ is a probability measure on $\Bim$. Suppose that $a, b > 0$ and $\alpha \in (0,1)$. Suppose that $a < b$ and that for all $r \in [a, b]$ we have
\begin{equation*}
s_r^{(k)}(\lambda) \leq \alpha.
\end{equation*}
Then we have
\begin{equation*}
s_{a\sqrt{k}}(\lambda) \leq \alpha k \left( \frac{2e}{\pi} \right)^{\frac{k-1}{2}} + k!\cdot ka^2  b^{-2} .
\end{equation*}
\end{lemma}

\begin{remark}
Combining Lemma \ref{lemma:ind_srk_to_sr} with \eqref{eq:trivial_srk_bound} we get a result that can be stated informally as follows. Let $\lambda_1, \lambda_2, \dots, \lambda_n$ be measures on $\R / \pi \Z$. Assume that we have some bound on $s_r(\lambda_i)$ for all integers $i \in [1,n]$ and all $r$ in a suitably large range of scales around some scale $r_0$. Then we can get a vastly improved bound for $s_{r_0}(\lambda_1 * \lambda_2 * \dots * \lambda_n)$.

This is essentially the same as \cite[Theorem 1.19]{Kittle2024ASENS}. However \cite[Theorem 1.19]{Kittle2024ASENS} is not sufficient for the proof of our result on Furstenberg measures. In what follows, we decompose the Furstenberg measure $\nu$ as the convex combination of measures that can be approximated by the convolutions of measures. This allows us to estimate $s_r^{(k)}(\nu)$ for arbitrary scales using \eqref{eq:trivial_srk_bound} amongst other things. Unlike the setting of for example \cite{Kittle2024ASENS}, we cannot estimate the detail of the convolution factors at a sufficiently large range of scales and so cannot apply \cite[Theorem 1.19]{Kittle2024ASENS}.

In fact, the decomposition we use to estimate $s_r^{(k)}(\nu)$ depends on the exact value of $r$. For this reason the notion of order $k$ detail is a key innovation of this paper that is necessary for the proof.
\end{remark}

We now need tools for bounding the detail of a measure at a given scale. One of them is the following. 
	
\begin{lemma} \label{lemma:small_rvs_to_detail}
For every $\alpha > 0$ there exists some $C > 0$ such that the following is true. Let $X_1, X_2, \dots, X_n$ be independent random variables taking values in $\Bim$ such that $|X_i| < s$ almost surely for some $s>0$. Let $\sigma > 0 $ be defined by $\sigma^2 = \sum_{i=1}^n \var X_i$. Let $ r \in (s, \sigma)$. Suppose that
\begin{equation*}
\frac{\sigma}{r}, \frac{r}{s} \geq C.
\end{equation*}
Then
\begin{equation*}
s_r(X_1 + X_2 + \dots + X_n) \leq \alpha.
\end{equation*}
\end{lemma}

Here and through out this paper when $x \in \R/ \pi \Z$ we use $|x|$ to denote $\min_{y \in x} |y|$. The idea of the proof of Theorem \ref{theorem:main_furstenberg} is to show that $\nu \circ \phi^{-1}$ can be expressed as a convex combination of measures each of which can be approximated by the law of the sum of many small independent random variables with some control over the variances of these variables. One difficulty with this is that the measures which $\nu \circ \phi^{-1}$ is a convex combination of are only approximately the laws of sums of small independent random variables of the required form. To deal with this we will need the following.

\begin{lemma} \label{lemma:detail_wasserstein}
There is some constant $C>0$ such that the following is true. Let $\lambda_1$ and $\lambda_2$ be probability measures on $\Bim$ and let $r>0$. Let $k \in \Z_{>0}$. Then
\begin{equation*}
\left| s_r^{(k)}(\lambda_1) - s_r^{(k)}(\lambda_2) \right| \leq C r^{-1} \mathcal{W}_1(\lambda_1, \lambda_2).
\end{equation*}
\end{lemma}

Here $\mathcal{W}_1(\cdot, \cdot)$ denotes Wasserstein distance.

Now we need to explain how we express $\nu \circ \phi^{-1}$ as a convex combination of measures each of which are close to the law  of a sum of small independent random variables. To do this we will need a chart for some neighbourhood of the identity in $\Gp$.

To do this we use the logarithm from $\Gp$ to its Lie algebra $\A$ defined in some open neighbourhood of the identity in $\Gp$. We also fix some basis of $\A$ and use this to identify $\A$ with $\R^3$ and fix some Euclidean product and corresponding norm on $\A$.

Now we consider the expression
\begin{equation*}
x = \gamma_1\gamma_2 \dots \gamma_T b
\end{equation*}
where $T$ is a stopping time, $\gamma_1, \gamma_2, \dots$ are random variables taking values in $\Gp$ which are i.i.d.\ samples from $\mu$, and $b$ is a sample from $\nu$ independent of the $\gamma_i$. Clearly $x$ is a sample from $\nu$. We then construct some $\sigma$-algebra $\mathcal{A}$ such that we can write
\begin{equation}
x = g_1 \exp(u_1) g_2  \exp(u_2) \dots  g_n \exp(u_n)  b \label{eq:x_decomp}
\end{equation}
where all of the $g_i$ are $\mathcal{A}$ -measurable random variables taking values in $\Gp$ and $b$ is an $\mathcal{A}$-measurable random variable taking values in $\B$. Furthermore the $u_i$ are random variables taking values in $\A$ in a small ball around the origin such that conditional on $\mathcal{A}$ we can find a lower bound on their variance. We then Taylor expand to show that $\phi(x)$ can be approximated in the required way after conditioning on $\mathcal{A}$. To do this construction we construct stopping times $0 = T_0 < T_1 < T_2 < \dots < T_n = T$ and construct our random variables such that $$g_i \exp(u_i) = \gamma_{T_{i-1}+1} \dots \gamma_{T_i}.$$

To explain this statement more precisely we first need to define the Cartan decomposition.
	
\begin{definition}[Cartan decomposition] \label{definition:singular_value_decomp}
We can write each element $g$ of $\Gp$ with $\lnb g \rnb > 1$ in the form
\begin{equation*}
R_{\theta_1} A_{\lambda} R_{-\theta_2}
\end{equation*}
where 
\begin{equation*}
R_x := 
\begin{pmatrix}
\cos x & -\sin x\\
\sin x & \cos x
\end{pmatrix}
\end{equation*}
is the rotation by $x$ and
\begin{equation*}
A_{\lambda} := 
\begin{pmatrix}
\lambda & 0\\
0 & \lambda^{-1}
\end{pmatrix}
\end{equation*}
in exactly one way with $\lambda > 1$ and $\theta_1, \theta_2 \in \R / \pi \Z$. We will let $b^{+}(g) = \phi^{-1}(\theta_1)$ and $b^{-}(g) = \phi^{-1}(\theta_2 + \frac{\pi}{2} )$. 
\end{definition}

\begin{remark}
Note that in this notation we have that if $\lnb g \rnb$ is large then providing $b \in \B$ is not too close to $b^{-}(g)$ we have that $g b$ is close to $b^{+}(g)$. We will make this more precise in Lemma \ref{lemma:new_shape_b_simple_singular_value}.
\end{remark}

We now let $d$ denote the metric on $\B$ induced by $\phi$. In other words if $x, y \in \B$ then $d(x, y) := |\phi(x)-\phi(y)|$. Whenever we write $d(\cdot, \cdot)$ it will be clear whether we are applying it to elements of $\Gp$ or elements of $\B$ and so clear if we are referring to the distance function of our left invariant Riemannian metric on $\Gp$ or to our metric on $\B$.

By carrying out some calculations about the Cartan decomposition and applying Taylor's theorem we can prove the following.

\begin{proposition} \label{proposition:intro_decomp_detail}
For every $t > 0$ there exist $C, \delta > 0$ such that the following is true. Let $n \in \Z_{>0}$ and let $u^{(1)}, u^{(2)}, \dots, u^{(n)} \in \A$. Let $g_1, \dots, g_n \in \Gp$ and let $b \in \B$. Let $r>0$. Suppose that for each integer $i \in [1,n]$ we have
\begin{equation*}
\lnb g_i \rnb \geq C
\end{equation*}
and
\begin{equation*}
\lnb u^{(i)} \rnb \leq \lnb g_1 g_2 \dots g_i \rnb^2 r.
\end{equation*}
Suppose that for each integer $i \in [1, n-1]$ we have
\begin{equation*}
d(b^{+}(g_i), b^{-}(g_{i+1})) > t 
\end{equation*}
and also that
\begin{equation*}
d(b, b^{-}(g_{n})) > t.
\end{equation*}
Suppose further that
\begin{equation*}
\lnb g_1 g_2 \dots g_n \rnb^2 r < \delta.
\end{equation*}
Let $x$ be defined by
\begin{equation}
x = g_1 \exp(u^{(1)}) \dots g_n \exp(u^{(n)}) b. \label{eq:expression_for_x}
\end{equation}
For each integer $i \in [1, n]$ let $\zeta_i \in \As$ be the derivative defined by
\begin{equation}
\zeta_i = D_u(\phi(g_1 g_2 \dots g_i \exp(u) g_{i+1} g_{i+2} \dots g_n b ))|_{u=0} \label{eq:derivative}
\end{equation}
and let $S$ be defined by
\begin{equation*}
S = \phi(g_1g_2 \dots g_{n}b) + \sum_{i=1}^{n} \zeta_i(u^{(i)}).
\end{equation*}
Then we have
\begin{equation*}
d \left(\phi(x), S \right) \leq C^n \lnb g_1 g_2 \dots g_n \rnb^2 r^2.
\end{equation*}
\end{proposition}

Informally this proposition states that under some conditions, when $x$ is  of the form \eqref{eq:expression_for_x} then $\phi(x)$ is close to its first order Taylor expansion in the $u^{(i)}$.

In \eqref{eq:derivative} $D_u$ denotes the derivative of the map with respect to $u$.

We will later use this along with some results about the first derivatives of the exponential at $0$, Lemma \ref{lemma:small_rvs_to_detail}, and \eqref{eq:trivial_srk_bound} to get a bound on the order $k$ detail of the expression $x$. We can then get an upper bound on the order $k$ detail of some sample $x$ from $\nu$ conditional on some $\sigma$-algebra $\mathcal{A}$. Due to the convexity of $s_r^{(k)} (\cdot)$ we can then find an upper bound for $s_r^{(k)} (\nu)$ by taking the expectation of this bound.

We will now outline some of the tools we will use to decompose $x$ in the way described in \eqref{eq:x_decomp}. To do this we introduce the following stopping times.

\begin{definition} \label{definition:tau_t_v}
    Suppose that $\gamma = (\gamma_1, \gamma_2, \dots)$ is a sequence of random variables taking values in $\Gp$. Then given some $P>0$ and some $v \in \B$ we define the stopping time $\tau_{P, v}(\gamma)$ by
    \begin{equation*}
        \tau_{P, v}(\gamma) := \inf \{ n : \lnb (\gamma_1 \gamma_2 \dots \gamma_n)^T \hat{v} \rnb \geq P \lnb \hat{v} \rnb \}
    \end{equation*}
    where $\hat{v} \in \R^2 \backslash \{0\}$ is a representative of $v$ and $T$ denotes transpose. Where $\gamma$ is obvious from context we will write $\tau_{P, v}$ to mean $\tau_{P, v}(\gamma)$.
\end{definition}

Note that this definition does not depend on our choice of $\hat{v}$. We now let $\gamma_1, \gamma_2, \dots$ be i.i.d.\ samples from $\mu$. We will show that we can find some $\sigma$-algebra $\hat{\mathcal{A}}$, some $\hat{\mathcal{A}}$-measurable random variable $a$ taking values in $\Gp$ and some random variable $u$ taking values in a small ball around the origin in $\A$ such that we may write $\gamma_1 \gamma_2 \dots \gamma_{\tau_{P, v}} = a \exp(u)$ and such that conditional on $\hat{\mathcal{A}}$ we know that $u$ has at least some variance.
	
First we need to define some analogue of variance for random values taking values in $\Gp$. For this we will make use of  $\log$. Specifically given some fixed $g_0 \in \Gp$ and some random variable $g$ taking values in $\Gp$ such that $g_0^{-1}g$ is always in the domain of $\log$ we will define $\Tr \vart_{g_0}[g]$ to be the trace of the covariance matrix of $\log(g_0^{-1} g)$. This clearly depends on our choice of Euclidean structure on $\A$. The proof will work with any choice of structure though the choice will affect the value of the constant $C$ we find in Theorem \ref{theorem:main_furstenberg}.

We now define the quantity $v(g;r)$ as follows.
	
\begin{definition} \label{definition:v_g_r}
Let $g$ be a random variable taking values in $\Gp$ and let $r> 0$. We then define $v(g;r)$ to be the supremum of all $v \geq 0$ such that we can find some $\sigma$-algebra $\mathcal{A}$ and some $\mathcal{A}$- measurable random variable $a$ taking values in $\Gp$ such that $|\log(a^{-1} g)| \leq r$ almost surely and
\begin{equation*}
\mathbb{E} \left[\Tr \vart_{a} \left[ g | \mathcal{A} \right] \right] \geq v r^2.
\end{equation*}
\end{definition}

\begin{proposition} \label{proposition:lots_of_v}
There is some absolute constant $c > 0$ such that the following is true. Let $\mu$ be a finitely supported Zariski-dense probability measure on $\Gp$ and let $\hat{\nu}$ be some probability measure on $\B$. Suppose that $M_{\mu} < \infty$ and that $h_{RW} / \chi$ is sufficiently large. Let $M > M_{\mu}$ be chosen large enough that $\log M \geq h_{RW}$. Suppose that $P$ is sufficiently large (depending on $\mu$ and $M$) and let $\hat{m} = \floor{\frac{\log M}{100 \chi}}$.

Let $\gamma_1, \gamma_2, \dots$ be i.i.d.\ samples from $\mu$ and let $\tau_{P, v}$ be as in Definition \ref{definition:tau_t_v}. Then there exist some $s_1, s_2, \dots, s_{\hat{m}} > 0$ such that for each $i \in [1, \hat{m}] \cap \Z$
\begin{equation*}
s_i \in \left( t^{ - \frac{\log M}{\chi}}, t^{ -\frac{h_{RW}}{10 \chi}} \right)
\end{equation*}
and for each $i \in [\hat{m}-1]$
\begin{equation*}
s_{i+1} \geq P^{3} s_i
\end{equation*}
and such that
\begin{equation*}
\sum_{i=1}^{\hat{m}} \int_{\B} v(\gamma_1 \gamma_2 \dots \gamma_{\tau_{P, w}}; s_i) \, \hat{\nu}(dw)
 \geq c \left( \frac{h_{RW}}{\chi} \right)  \left( \max \left\{ 1,  \log \frac{ \log M}{h_{RW}} \right\} \right)^{-1} .
\end{equation*}
\end{proposition}

The measure $\hat{\nu}$ for which we apply Proposition \ref{proposition:lots_of_v} comes from the following result in renewal theory.

\begin{theorem} \label{theorem:renewal_theorem}
Let $\mu$ be a Zariski-dense compactly supported probability measure on $\Gp$. Then there is some probability measure $\hat{\nu}$ on $\B$ such that the following is true. Let $\gamma_1, \gamma_2, \dots$ be i.i.d. samples from $\mu$. Then for all $v \in \B$ the law of $(\gamma_1 \gamma_2 \dots \gamma_{\tau_{P, v}})^T v$ converges weakly to $\hat{\nu}$ as $P \to \infty$. Furthermore this convergence is uniform in $v$.
\end{theorem}

We will also need the following corollary.

\begin{corollary} \label{corollary:renewal_theorem}
	Let $\mu$ be a Zariski-dense compactly supported probability measure on $\Gp$. Let $\hat{\nu}$ be as in Theorem \ref{theorem:renewal_theorem}. Let $\gamma_1, \gamma_2, \dots$ be i.i.d. samples from $\mu$. Let $a \in \Gp$, $P>0$ and define $\tau_{P, a}$ by
	\begin{equation*}
	\tau_{P, a} := \inf \{ n : \lnb a \gamma_1 \gamma_2 \dots \gamma_n \rnb \geq P \lnb a \rnb \}.
	\end{equation*}
	Then $b^{-}(a \gamma_1 \gamma_2 \dots \gamma_{\tau_{P, a}})^{\perp}$ converges weakly to $\hat{\nu}$ as $P \to \infty$. Furthermore this convergence is uniform in $a$. 
\end{corollary}

In \cite[Theorem 1]{KESTEN_1974} it is proven that Theorem \ref{theorem:renewal_theorem} holds without the condition that it is uniform in $v$ in a much more general setting providing some conditions are satisfied. In \cite[Section 4]{GUIVARCH_LEPAGE_2016} it is shown that the conditions of \cite[Theorem 1]{KESTEN_1974} are satisfied in the setting of Theorem \ref{theorem:renewal_theorem}. In Section \ref{section:appendix}, we will prove Theorem \ref{theorem:renewal_theorem} by deducing uniform convergence from (not necessarily uniform) convergence and deduce Corollary \ref{corollary:renewal_theorem} from it. A formula for $\hat{\nu}$ is given in \cite[Theorem 1]{KESTEN_1974} though this will not be needed for the purposes of this paper.

In Section \ref{section:sum_of_variances} we show how to construct the decomposition \eqref{eq:x_decomp} of a sample $x$ from $\nu$. The details are very technical so we only discuss in this outline how given a sufficiently small scale $\tilde{r}$ one can construct a stopping time $\tau$, and a $\sigma$-algebra $\mathcal{A}$ such that
\begin{equation*}
\gamma_1 \gamma_2 \dots \gamma_{\tau} = g \exp(u)
\end{equation*}
for some $\mathcal{A}$-measurable random variable $g$ taking values in $\Gp$ and some random $u$ taking values in  $\A$ such that $\lnb u \rnb \leq \lnb g \rnb^2 \tilde{r}$ almost surely and after conditioning on $\mathcal{A}$ we have a good lower bound for $\frac{\var (u)}{\lnb g \rnb^4 \tilde{r}^2}$.

We fix a small $s$ and some $P$ that is much smaller that $s^{-1}$. Let $s_{i_0}$ be one of the scales we get when we apply Proposition \ref{proposition:lots_of_v} with the measure from Theorem \ref{theorem:renewal_theorem} in the role of $\hat{\nu}$.

Fix an arbitrary $b \in \B$. Let $Q = (s / s_{i_0})^{1/2}/ P$ and let the stopping time $S$ be defined by $$S = \inf \{ n : \lnb (\gamma_1 \dots \gamma_n)^T b \rnb \geq Q \lnb b \rnb \}.$$

By Theorem \ref{theorem:renewal_theorem}, there is a random variable $w$ taking values in $\B$ such that $w^{\perp}$ has law $\hat{\nu}$ and
\begin{equation*}
d(b^{-}(\gamma_1 \gamma_2 \dots \gamma_{S}), w)
\end{equation*}
is small with high probability.

Now let $$T = \inf \{ n : \lnb (\gamma_{S+1} \gamma_{S+2} \dots \gamma_n)^T w^{\perp} \rnb \geq P \lnb w^{\perp} \rnb \}.$$ Note that by Proposition \ref{proposition:lots_of_v} there is some $\sigma$-algebra $\tilde{\mathcal{A}}$ such that
\begin{equation*}
\gamma_{S + 1} \gamma_{S+ 2} \dots \gamma_{T} = a \exp(u)
\end{equation*}
where $a$ is an $\tilde{\mathcal{A}}$-measurable random element of $\Gp$ and $u$ is a random element of $\A$ with $\lnb u \rnb \leq s_{i_0}$ and a good lower bound on $\frac{\Tr \var (u)}{s_{i_0}^2}$.

Now we define $g = \gamma_1 \dots \gamma_{S} a$. Using the definition of $w$ it is possible to show that $\lnb g \rnb$ is approximately $Q \cdot P = (s / s_{i_0})^{1/2}$.

Note that the scale $s_{i_0}$ depends on the measure $\hat{\nu}$ so the convergence in Theorem \ref{theorem:renewal_theorem} is important. On the other hand it does not matter what this limit measure is.

The construction in Section \ref{section:sum_of_variances} is significantly more elaborate. In particular, we will make use of all the scales $s_1, \dots, s_{\hat{m}}$ provided by Proposition \ref{proposition:lots_of_v}. Moreover, we will need to apply it for a carefully chosen sequence of parameters in the role of $P$. To aid with this in Section \ref{section:sum_of_variances} we construct a family of ways of writing a stopped random walk in $\Gp$ in such a way that we may apply Proposition \ref{proposition:intro_decomp_detail} which is closed under concatenation. 

Finally we discuss some ingredients of the proof of Proposition \ref{proposition:lots_of_v}. We define the entropy of an absolutely continuous random variable taking values in $\Gp$ to be the differential entropy with respect to a certain normalisation of the Haar measure and denote this by $H(\cdot)$. We define this more precisely in Section \ref{section:entropy}. We will then prove the following theorem.

\begin{theorem} \label{theorem:prob_exist}
Let $g, s_1$ and $s_2$ be independent random variables taking values in $\Gp$ such that $s_1$ and $s_2$ are absolutely continuous and  have finite entropy. Define $k$ by
\begin{equation*}
k := H(gs_1) - H(s_1) - H(g s_2) + H(s_2)
\end{equation*}
and let $c := \frac{3}{2} \log \frac{2}{3} \pi e \Tr \vart_{\id}[s_1] - H(s_1)$. Suppose that $k>0$. Suppose further that $s_1$ and $s_2$ are supported on the ball of radius $\varepsilon$ centred at the identity for some sufficiently small $\varepsilon > 0$. Suppose also that $\Tr \vart_{\id} [s_1] \geq A \varepsilon^2$ for some positive constant $A$. Then
\begin{equation*}
\mathbb{E} \left[ \Tr \vart_{g s_2} \left[ g| g s_2 \right] \right] \geq \frac{2}{3} (k - c -C \varepsilon ) \Tr \vart_{\id} [s_1]
\end{equation*}
where $C$ is some positive constant depending only on $A$.
\end{theorem}

We apply this theorem when $s_1$ and $s_2$ are smoothing functions at appropriate scales with $s_2$ corresponding to a larger scale than $s_1$. The value $k$ can be thought of as the new information that can be gained by discretising at the scale corresponding  to $s_1$ after discretising at the scale corresponding to $s_2$. When we apply this theorem we bound $k$ in the following way. We let $g = \gamma_1 \gamma_2 \dots \gamma_{\tau}$ where the $\gamma_i$ are i.i.d.\ samples from $\mu$ and $\tau$ is some stopping time. We let $s_1, s_2, \dots, s_n$ be a sequence of smoothing random variables corresponding to various scales with $s_i$ corresponding to a larger scale than $s_j$ whenever $i > j$. For $i = 1, \dots, n-1$ we let $k_i$ be defined by $$k_i = H(g s_i) - H(s_i) - H(g s_{i+1}) + H(s_{i+1})$$ and note that we have the following telescoping sum
\begin{align*}
\sum_{i=1}^{n-1} k_i &= \sum_{i=1}^{n-1} H(g s_i) - H(s_i) - H(g s_{i+1}) + H( s_{i+1}) \\
&= H(g s_1) - H(s_1) - H(g s_n) + H(s_n).
\end{align*}
Since when we apply this theorem $s_n$ will correspond to a scale much larger than $s_1$ we are able to bound $ H(g s_1) - H(s_1) - H(g s_n) + H(s_n)$ for our careful choice of smoothing functions in terms of $h_{RW}$, $M_{\mu}$ and $\chi$.

The value $c$ in the above theorem measures how close $s_1$ is to being a spherical normal distribution. For random variables taking values in $\R^d$ it is well known that the random variable with the greatest differential entropy out of all random variables with a given covariance matrix is a multivariate normal distribution. From this it is easy to deduce that if  $X$ is a continuous random variable taking values in $\R^d$ then $H(X) \leq \frac{d}{2} \log \frac{2}{d} \pi e \Tr \var X$ with equality if and only if $X$ is a spherical normal distribution. A similar thing is true for random variables taking values in $\Gp$. In particular $c$ is small when $s_1$ is close to being the image of a spherical normal distribution on $\A$ under $\exp$.

For the conclusion of Theorem \ref{theorem:prob_exist} to be useful in proving Proposition \ref{proposition:lots_of_v} we need $g$ to almost surely be contained in some ball of radius $O \left( \sqrt{\Tr \vart_{\id} [s_1]} \right)$ centred on $g s_2$. For this reason we require $s_2$ to be compactly supported. To make our telescoping sum useful we need $s_1$ and $s_2$ to be members of the same family of random variables. For this reason we take $s_1$ and $s_2$ to be compactly supported approximations of the image of the spherical normal distribution on $\A$ under $\exp$. To do this we will find bounds on the differential entropy of various objects smoothed with these compactly supported approximations to the normal distribution at different scales.

We then combine Theorems \ref{theorem:prob_exist} and a bound for the entropy of the stopped random walk along with some calculations about the entropy and variance of the smoothing functions to prove Proposition \ref{proposition:lots_of_v}.

\subsection{Notation}
We will use Landau's $O(\cdot)$ notation. Given some positive quantity $X$ we write $O(X)$ to mean some quantity whose absolute values is bounded above by $CX$ some constant $C$. If $C$ is allowed to depend on some other parameters then these will be denoted by subscripts. Similarly we write $o(X)$ to mean some quantity whose absolute value is bounded above by $cX$, where $c$ is some positive value which tends to $0$ as $X \to \infty$. Again if $c$ is allowed to depend on some other parameters then these will be denoted by subscripts. We also let $\Theta(X)$ be some quantity which is bounded below by $CX$ where $C$ is some positive absolute constant. If $C$ is allowed to depend on some other parameters then these will be denoted by subscripts.

We write $X \lesssim Y$ to mean that there is some constant $C>0$ such that $X \leq CY$. Similarly we write $X \gtrsim Y$ to mean that there is some constant $C>0$ such that $X \geq CY$ and $X \cong Y$ to mean $X \lesssim Y$ and $X \gtrsim Y$. If these constants are allowed to depend on some other parameters then these are denoted in subscripts.

\subsection{Organisation of the Paper}

Here we give some brief remarks on the organisation of the paper. In Section \ref{section:prerequisites} we state some results on random walks on $\Gp$, entropy and probability which will be used though-out the paper. In Section \ref{section:detail} we recall some results on detail from \cite{Kittle2024ASENS} and introduce order $k$ detail. In Section \ref{section:taylor} we carry out some calculations on derivatives of various products in $\Gp$ and prove Proposition \ref{proposition:intro_decomp_detail}. In Section \ref{section:disintegration_argument} we prove some basic results about entropy, regular conditional probability and variance on $\Gp$ and use them to prove Theorem \ref{theorem:prob_exist}. In Section \ref{section:entropy_gap} we use Theorem \ref{theorem:prob_exist} and some calculations with entropy to prove Proposition \ref{proposition:lots_of_v}. In Section \ref{section:sum_of_variances} we develop some tools for putting together the variance found in Proposition \ref{proposition:lots_of_v} at different scales. In Section \ref{section:proof_of_main_theorem} we use these tools to prove Theorem \ref{theorem:main_furstenberg}. In Section \ref{section:examples} we give examples of Furstenberg measures satisfying the conditions of Theorem \ref{theorem:main_furstenberg}. Finally in Section \ref{section:appendix} we prove Theorem \ref{theorem:renewal_theorem}.

\section{Prerequisites} \label{section:prerequisites}
In this subsection we give some prerequisites for the paper.
\subsection{Random walks on $\Gp$}
Here we give some well known results about random walks on $\Gp$. These results may be found in \cite{BOUGEROL_LACROIX_1985} or follow easily from results found therein.

\begin{lemma} \label{lemma:first_large_deviations_estimates}
    Suppose that $\mu$ is a compactly supported Zariski-dense probability measure on $\Gp$ and let $\chi$ be its Lyaponuv exponent. Let $\gamma_1, \gamma_2, \dots$ be i.i.d.\ samples from $\mu$. Then for every $\varepsilon > 0$ there is some $\delta > 0$ such that the following holds.

    For all sufficiently large $n$ we have
    \begin{equation}
        \mathbb{P}\left[ \left| n \chi - \log \lnb \gamma_1 \gamma_2 \dots \gamma_n \rnb\right| > \varepsilon n \right] < \exp( - \delta n). \label{eq:large_deviations_estimate}
    \end{equation}

    Furthermore for all $v \in \R^2 \backslash \{0\}$ for all sufficiently large $n$ we have
    \begin{equation}
        \mathbb{P}\left[ \left| n \chi + \log \lnb v \rnb - \log \lnb (\gamma_1 \gamma_2 \dots \gamma_n)^T v \rnb\right| > \varepsilon n \right] < \exp( - \delta n) \label{eq:large_deviations_v}
    \end{equation}

    Furthermore if $P>0$ is sufficiently large and we define
    \begin{equation*}
        \tau_P := \inf \{n: \lnb \gamma_1 \gamma_2 \dots \gamma_n \rnb \geq P \} 
    \end{equation*}
    then
    \begin{equation}
        \mathbb{P}\left[ \left| \tau_P  - \log P / \chi \right| > \varepsilon \log P \right] < \exp( - \delta \log P) \label{eq:large_deviations_stopping_time}.
    \end{equation}

    Furthermore for all $v \in \B$ for all sufficiently large $P>0$ if we take $\tau_{P, v}$ as in Definition \ref{definition:tau_t_v} then
    \begin{equation}
        \mathbb{P}\left[ \left| \tau_{P, v}  - \log P / \chi \right| > \varepsilon \log P \right] < \exp( - \delta \log P) \label{eq:large_deviations_stopping_time_v}.
    \end{equation}
\end{lemma}

\begin{proof}
    Equation \eqref{eq:large_deviations_estimate} follows from \cite[Theorem V.6.2]{BOUGEROL_LACROIX_1985}. Equation \eqref{eq:large_deviations_v} is a special case of \cite[Theorem V.6.1]{BOUGEROL_LACROIX_1985}. 

    We now deduce \eqref{eq:large_deviations_stopping_time} from \eqref{eq:large_deviations_estimate}. If $ \tau_P> \log P / \chi + \varepsilon \log P $ then we must have
    \begin{equation*}
        \lnb \gamma_1 \gamma_2 \dots \gamma_{\lfloor \log P / \chi + \varepsilon \log P \rfloor } \rnb \leq P.
    \end{equation*}
    By \eqref{eq:large_deviations_estimate}, providing $P$ is sufficiently large, this has probability at most $\exp(- \delta \log P)$.

    Choose $R > 0$ such that $\lnb \gamma_i \rnb \leq R$ almost surely (this is possible as $\mu$ is compactly supported). If $\tau < \log P / \chi - \varepsilon \log P$ then there must be some integer $k \in \left[\log P / \log R, \log P / \chi - \varepsilon \log P \right]$ such that
    \begin{equation*}
        \log \lnb \gamma_1 \gamma_2 \dots \gamma_k \rnb \geq \log P > k (\chi + \varepsilon \chi).
    \end{equation*}
    The result now follows from \eqref{eq:large_deviations_estimate} and summing a geometric series.

    Finally \eqref{eq:large_deviations_stopping_time_v} follows from \eqref{eq:large_deviations_v} by essentially the same argument.
\end{proof}

We will need the following positive dimensionality result.
\begin{theorem} \label{theorem:positive_dimension}
    Suppose that $\mu$ is a Zariski-dense probability measure on $\Gp$ and let $\nu$ be its Furstenberg measure. Suppose that there exists some $\varepsilon > 0$ such that $$\int \| g \|^{\varepsilon} \, d \mu(g) < \infty.$$ Then there exist $C, \delta >0$ such that for any $x \in \B$ and any $r>0$ we have
    \begin{equation*}
        \nu(B(x, r)) < C r^{\delta}.
    \end{equation*}
\end{theorem}

\begin{proof}
    This is \cite[Corollary VI.4.2]{BOUGEROL_LACROIX_1985}.
\end{proof}

We also need the following facts about the speed of convergence to the Furstenberg measure.
\begin{lemma} \label{lemma:convergence_to_fm}
    Suppose that $\mu$ is a compactly supported Zariski-dense probability measure on $\Gp$ and let $\gamma_1, \gamma_2, \dots$ be i.i.d.\ samples from $\mu$. Then $b^{+}(\gamma_1 \gamma_2 \dots \gamma_n)$ converges almost surely and furthermore there exists some constant $\varepsilon > 0$ such that for all sufficiently large $n$
    \begin{equation}
        \mathbb{P}[d(b^{+}(\gamma_1 \gamma_2 \dots \gamma_n), \lim \limits_{n \to \infty} b^{+}(\gamma_1 \gamma_2 \dots \gamma_n) ) > \exp(- \varepsilon n)] < \exp(-\varepsilon n). \label{eq:exponential_convergence_to_fm}
    \end{equation}
    Furthermore for all sufficiently large $N$ we have
    \begin{equation}
        \mathbb{P}[\exists n \geq N: d(b^{+}(\gamma_1 \gamma_2 \dots \gamma_n), \lim \limits_{m \to \infty} b^{+}(\gamma_1 \gamma_2 \dots \gamma_m) ) > \exp(- \varepsilon n)] < \exp(-\varepsilon N) \label{eq:c_exponential_convergence_to_fm}
    \end{equation}
    and for all $v \in \B$ we have
    \begin{equation}
        \mathbb{P}[\exists m \geq N: d(v, b^{+}(\gamma_1 \dots \gamma_m)) < \exp(- \varepsilon m)] < \exp(- \delta N) \label{eq:products_not_too_close}.
    \end{equation}
\end{lemma}

\begin{proof}
    The convergence of $b^{+}(\gamma_1 \gamma_2 \dots \gamma_n)$ and \eqref{eq:exponential_convergence_to_fm} follow from for example \cite[Proposition V.2.3]{BOUGEROL_LACROIX_1985}. Equation \eqref{eq:c_exponential_convergence_to_fm} follows from \eqref{eq:exponential_convergence_to_fm} and summing a geometric series. Finally \eqref{eq:products_not_too_close} follows easily from \eqref{eq:c_exponential_convergence_to_fm} and Theorem \ref{theorem:positive_dimension}.
\end{proof}

We finish this subsection with the following corollary.
\begin{corollary}\label{corollary:size_after_stopping_time}
    Suppose that $\mu$ is a compactly supported Zariski-dense probability measure on $\Gp$. Let $\gamma_1, \gamma_2, \dots$ be i.i.d.\ samples from $\mu$ and let $\varepsilon > 0$. Then there exists  $delta > 0$ such that for all sufficiently large $P$ and all $v \in \B$ we have
    \begin{equation*}
        \mathbb{P}[\left| \log \lnb \gamma_1 \gamma_2 \dots \gamma_{\tau_{P, v}} \rnb - \log P \right| > \varepsilon \log P] < \exp(- \delta \log P).
    \end{equation*}
\end{corollary}

\begin{proof}
    By definition we trivially have $\lnb \gamma_1 \gamma_2 \dots \gamma_{\tau_{P, v}} \rnb \geq P$. Let $R$ be chosen such that $\lnb \cdot \rnb \leq R$ on the support of $\mu$. Clearly $\tau_{P, v} \geq \log P / \log R$ and
    \begin{align*}
        PR & \geq \lnb (\gamma_1 \gamma_2 \dots \gamma_{\tau_{P, v}})^T v \rnb \\
        & = \lnb \gamma_1 \gamma_2 \dots \gamma_{\tau_{P, v}} \rnb \sin d(b^{+}(\gamma_1 \gamma_2 \dots \gamma_{\tau_{P, v}}), v).
    \end{align*}
    In particular if $\log \lnb \gamma_1 \gamma_2 \dots \gamma_{\tau_{P, v}} \rnb \geq (1+ \varepsilon) \log P$ then
    \begin{equation*}
        d(b^{+}(\gamma_1 \gamma_2 \dots \gamma_{\tau_{P, v}}), v) \leq 10 R \exp(- \varepsilon \log P).
    \end{equation*}
    The result now follows by \eqref{eq:products_not_too_close}.
\end{proof}

\subsection{Entropy}
In this subsection we will describe some of the properties of entropy used in this paper. We will describe entropy for both absolutely continuous and discrete measures on $\R^d$ and $\Gp$.

\begin{definition}[KL-divergence]
Let $\lambda_1$ be a probability measure on a measurable space $(E, \xi)$ and let $\lambda_2$ be a measure on $(E, \xi)$. Then we define the \emph{KL-divergence} of $\lambda_1$ given $\lambda_2$ by
\begin{equation*}
\mathcal{KL}(\lambda_1, \lambda_2) := \int_E \log \frac{d \lambda_1}{d \lambda_2} \, d \lambda_1.
\end{equation*}
\end{definition}
	
\begin{definition}[Entropy]
    Given a probability measure $\lambda_1$ on a measurable space $(E, \xi)$ and a measure $\lambda_2$ on the same space we define the entropy of $\lambda_1$ with respect to $\lambda_2$ by $$D(\lambda_1 || \lambda_2) := -\mathcal{KL}(\lambda_1, \lambda_2).$$
\end{definition}

\begin{definition}
    Given a discrete probability measure $\lambda$ on some measurable set $(E, \xi)$ we define the entropy of $\lambda$ to be the entropy with respect to the counting measure and we denote this by $H(\lambda)$. In other words if $\lambda = \sum_{i} p_i \delta_{x_i}$ then $$H(\lambda) := - \sum_{i} p_i \log p_i.$$
    We define the entropy of a random variable to be the entropy of its law.
\end{definition}

\begin{definition}
    Given an absolutely continuous probability measure $\lambda$ on $\R^d$ we define the entropy of $\lambda$ to be the entropy of $\lambda$ with respect to the Lebesgue measure and denote this by $H(\lambda)$. We define the entropy of a random variable to be the entropy of its law.
\end{definition}

We use $H$ to denote entropy in both cases. It will be clear from context whether $H$ is being applied to a discrete measure (or random variable) or an absolutely continuous measure (or random variable) so this will not cause confusion.

We now wish to define entropy for an absolutely continuous probability measure on $\Gp$. To do this we introduce the following normalisation of the Haar measure.

\begin{definition} \label{definition:tilde_m}
Let $\tilde{m}$ denote the Haar measure on $\Gp$ normalized such that $$\frac{d \tilde{m}}{d m \circ \log} (\id) = 1$$ where $m$ denotes the Lebesgue measure on $\A$ under our identification of $\A$ with $\R^3$.
\end{definition}

\begin{definition}
Let $\lambda$ be an absolutely continuous measure on $\Gp$. We then define the entropy of $\lambda$ to be its entropy with respect to $\tilde{m}$ and denote this by $H(\lambda)$.

Similarly if $g$ is a random variable taking values in $\Gp$ then we let $H(g)$ denote the entropy of its law.
\end{definition}

We have the following simple result.

\begin{lemma} \label{lemma:ent_increase}
Suppose that $g_1$ and $g_2$ are independent random variables taking values in some group $\mathbf{G}$ with $\sigma$-algebra $\xi$. Let $\lambda$ be a left invariant measure on $(\mathbf{G}, \xi)$. Then
\begin{equation*}
D(\mathcal{L}(g_1 g_2)||\lambda) \geq D(\mathcal{L}(g_2)||\lambda)
\end{equation*}
\end{lemma}
Here and throughout this paper given a random variable $X$ we will use $\mathcal{L}(X)$ to denote the law of $X$
\begin{proof}
This is well known. A proof in the special case where $\mathbf{G} = (\R, +)$ is given in \cite[Lemma 1.15]{Johnson_2004}. The same proof works in the more general setting described above.
\end{proof}

We also define entropy for non-probability measures.

\begin{definition} \label{definition:non_prob_entropy}
Suppose that $\lambda$ is a finite measure discrete measure on some set $S$. Then we define
\begin{equation*}
H(\lambda) := \lnb \lambda \rnb_1 H(\lambda / \lnb \lambda \rnb_1)
\end{equation*}
where $H(\lambda / \lnb \lambda \rnb_1)$ denotes either the Shannon entropy of $\lambda / \| \lambda \|_1$. Similarly if $\lambda$ is a finite absolutely continuous measure on $\R^d$ or $\Gp$ we define $H(\lambda) := \lnb \lambda \rnb_1 H(\lambda / \lnb \lambda \rnb_1)$ where $H(\lambda / \lnb \lambda \rnb_1)$ denotes the differential entropy of $\lambda / \lnb \lambda \rnb_1$ with respect to the Lebesgue measure on $\R^d$ or $\tilde{m}$ respectively.
\end{definition}

We say that a finite discrete measure with masses $p_1, p_2, \dots$ has finite entropy if $$\sum_{i=1}^{\infty} p_i |\log p_i| < \infty.$$ Similarly we say that a finite absolutely continuous measure on $\R^d$ or $\Gp$ with density function $f$ with respect to the Lebesgue measure or our normalised version of the Haar measure has finite entropy if $$\int f |\log f| < \infty.$$

We now have the following simple lemmas.

\begin{lemma}[Entropy is concave] \label{lemma:entropy_is_concave}
Let $\lambda_1, \lambda_2, \dots$ be finite measures with finite entropy either all on $\R^d$ or all on $\Gp$ which are either all absolutely continuous or all discrete. Suppose that $\sum_{i=1}^{\infty} \lnb \lambda_i \rnb_1 < \infty$ and both $H \left( \sum_{i=N}^{\infty} \lambda_i \right)$ and $\sum_{i=N}^{\infty} H \left(  \lambda_i \right)$ tend to $0$ as $N \to \infty$. Then
\begin{equation*}
H(\sum_{i=1}^{\infty} \lambda_i) \geq \sum_{i=1}^{\infty } H( \lambda_i).
\end{equation*}
\end{lemma}

\begin{proof}
This is proven for measures on $\R^d$ in \cite[Lemma 4.6]{Kittle2024ASENS}. The same proof also works in this setting.
\end{proof}

\begin{lemma}[Entropy is almost convex] \label{lemma:entropy_is_almost_convex}
Let $\lambda_1, \lambda_2, \dots$ be probability measures either all on $\R^d$ or all on $\Gp$ which are either all absolutely continuous or all discrete. Suppose that all of the probability measures have finite entropy. Let $\mathbf{p} = (p_1, p_2, \dots)$ be a probability vector. Then
\begin{equation*}
H(\sum_{i=1}^{\infty}p_i \lambda_i) \leq \sum_{i=1}^{\infty}p_i H( \lambda_i) + H(\mathbf{p}).  \label{eq:inf_h_entropy_b}
\end{equation*}
In particular if $p_i = 0$ for all $i > k$ for some $k\in \N$ then
\begin{equation*}
H\left(\sum_{i=1}^k p_i \lambda_i\right) \leq \sum_{i=1}^k p_i H( \lambda_i) + \log k . \label{eq:k_h_entropy}
\end{equation*}
\end{lemma}

\begin{proof}
This is proven in \cite[Lemma 4.7]{Kittle2024ASENS} for measures on $\R^d$. The same proof works in this setting.
\end{proof}

\begin{lemma} \label{lemma:entropy_adds}
Let $d$ be the distance function of a left invariant metric and let $r > 0$. Suppose that $g$ is a discrete random variable taking values in $\Gp$ and that there are $x_1, x_2, \dots, x_n \in \Gp$ and a probability vector $\mathbf{p} = (p_1, p_2, \dots, p_n)$ such that
\begin{equation*}
\mathbb{P} \left[ g = x_i \right] = p_i.
\end{equation*}
Suppose further that for every $i \neq j$ we have $d(x_i, x_j) > 2r$. Let $h$ be an absolutely continuous random variable taking values in $\Gp$. Suppose that $d(\id, h) \leq r$ almost surely. Suppose further that $h$ has finite entropy. Then
\begin{equation*}
H(gh) = H(g) + H(h)
\end{equation*}
\end{lemma}

\begin{proof}
This is proven for random variables taken values in $\R^d$ in \cite[Lemma 4.8]{Kittle2024ASENS}. The same proof works in this context.
\end{proof}

We will also adopt the following convention for defining the entropy on a product space. Let $(E_1, \xi_1)$ and $(E_2, \xi_2)$ be measurable spaces endowed with reference measures $m_1$ and $m_2$ such that if $\lambda$ is a measure on $(E_i, \xi_i)$ then we define the entropy of $\lambda$ by $H(\lambda) := D(\lambda || m_i)$. Then we take $m_1 \times m_2$ to be the corresponding reference measure for $E_1 \times E_2$. That is given some measure $\lambda$ on $E_1 \times E_2$ we take the entropy of $\lambda$ to be defined by $H(\lambda) = D(\lambda || m_1 \times m_2)$. With this we can give the following definition.
	
\begin{definition}[Conditional Entropy]
Let $X_1$ and $X_2$ be two random variables with finite entropy. Then we define the \emph{entropy of $X_1$ given $X_2$} by
\begin{equation*}
H(X_1 | X_2) = H(X_1, X_2) - H(X_2).
\end{equation*}
\end{definition}

\subsection{Probability}
In this subsection we will list some standard results from probability which we will use in this paper.

\begin{definition}[Filtration]
	We say that a sequence of $\sigma$-algebras $\mathcal{F} = (\mathcal{F}_1, \mathcal{F}_2, \dots)$ is a \emph{filtration} if $\mathcal{F}_1 \subset \mathcal{F}_2 \subset \dots$. Furthermore if we are also given a sequence of random variables $\gamma = (\gamma_1, \gamma_2 \dots)$ then we say that $\mathcal{F}$ is a filtration for $\gamma$ if in addition $\gamma_i$ is $\mathcal{F}_i$-measurable.
\end{definition}

\begin{definition}[Stopping time]
Given a filtration $\mathcal{F} = (\mathcal{F}_1, \mathcal{F}_2, \dots)$ we say that a random variable $T$ taking values in $\Z_{>0}$ is a stopping time for $\mathcal{F}$ if for every $n \in \Z_{>0}$ the event $T=n$ is $\mathcal{F}_n$ measurable. Given a sequence of random variables $\gamma = (\gamma_1, \gamma_2, \dots)$ we say that $T$ is a stopping time for $\gamma$ if it is a stopping time for the filtration $\sigma(\gamma_1), \sigma(\gamma_1, \gamma_2), \sigma(\gamma_1, \gamma_2, \gamma_3), \dots$.
\end{definition}

Stopping times and filtrations are important objects in probability. A fundamental property is that if $\mathcal{F}$ is a filtration for a sequence of i.i.d.\ random variables  $\gamma$ with $\gamma_{i+1}$ independent of $\mathcal{F}_i$ for all $i$ and $T$ is a stopping time for $\mathcal{F}$ then $(\gamma_{T+1}, \gamma_{T+2}, \gamma_{T+3}, \dots)$ has the same law as $(\gamma_1, \gamma_2, \dots)$ and is independent of $\mathcal{F}_T$. This is known as the strong Markov property. For a more thorough introduction to stopping times and filtrations see for example \cite[Chapter 17]{KLENKE_2014}.

\begin{lemma} \label{lemma:stopping_time_invariant}
    Let $\mathbf{G}$ be a group acting on some set $\mathbf{B}$. Let $\mu$ be a probability measure on $\mathbf{G}$ and suppose that $\nu$ is some probability measure on $\mathbf{B}$ which is invariant under $\mu$ - that is $\nu = \mu * \nu$.
    
    Let $\gamma_1, \gamma_2, \dots$ be i.i.d.\ random variables with law $\mu$ and let $\mathcal{F}_i$ be a filtration for the $\gamma_i$ such that $\gamma_{i+1}$ is independent from $\mathcal{F}_{i+1}$. Let $\tau$ be a stopping time for the filtration $\mathcal{F}_i$. Let $b$ be an independent sample from $\nu$. Then
    \begin{equation*}
        \gamma_1 \gamma_2 \dots \gamma_{\tau} b
    \end{equation*}
    has law $\nu$.
\end{lemma}

\begin{proof}
    First we will deal with the case where there is some $N \in \Z_{>0}$ such that $\tau\leq N$ almost surely. By the strong Markov property we know that
    \begin{equation*}
    \gamma_{\tau+1} \gamma_{\tau+2} \dots \gamma_N b
    \end{equation*}
    has law $\nu$ and is independent from $\gamma_1, \gamma_2, \dots, \gamma_{\tau}$. In particular this means that $\gamma_1 \gamma_2 \dots \gamma_{\tau} b$ has the same law as $\gamma_1 \gamma_2 \dots \gamma_N b$ and so $\gamma_1 \gamma_2 \dots \gamma_{\tau} b$ has law $\nu$. The general case follows by considering the stopping time $\tau' = \min \{\tau, N \}$ and taking the limit as $N\to \infty$.
\end{proof}

\begin{lemma} \label{lemma:generates_independent_filtration}
    Let $(\mathbb{P}, \Omega, \xi)$ be a probability space. Suppose that $\gamma_1, \gamma_2, \dots$ are i.i.d.\ random variables on this probability space taking values in some measurable set $\mathbf{X}$ with filtration $\mathcal{A}_i$ and suppose that $\gamma_{i+1}$ is independent of $\mathcal{A}_i$. Let $S$ be a stopping time for $(\mathcal{A}_i)_{i=1}^{\infty}$ and let $\hat{\mathcal{A}} \subset \xi$ be a $\sigma$-algebra which is conditionally independent of $\gamma_{S+1}, \gamma_{S+2}, \dots$ given $\gamma_1, \gamma_2, \dots, \gamma_S$. For $i = 1, 2, \dots$ define $\mathcal{F}_i$ by
    \begin{equation*}
        \mathcal{F}_i = \{ F \in \xi: F \cap \{i < S \} \in \mathcal{A}_i, F \cap \{i \geq S\} \in \sigma(\mathcal{A}_i, \hat{\mathcal{A}}) \}.
    \end{equation*}
    Then $\mathcal{F}_i$ is a filtration for the $\gamma_i$ and $\gamma_{i+1}$ is independent of $\mathcal{F}_i$.
\end{lemma}
First note that this lemma is in some sense trivial. Essentially it says that if we have a sequence of independent random variables which we sequentially draw and after some stopping time we gain some extra information which is conditionally independent of everything after that stopping time given what we have seen so far then at each step in this process the value of the next random variable will be independent of all the information we have so far. We now give a formal proof.
\begin{proof}
    It is trivial that $\mathcal{F}_i$ is a filtration for the $\gamma_i$. This means that we only need to show that $\gamma_{i+1}$ is independent of $\mathcal{F}_i$. Let $D \subset \mathbf{X}$ be measurable and let $F \in \mathcal{F}_i$. Then we have $F \cap \{ i < S \} \in \mathcal{A}_i$ and so since $\gamma_{i+1}$ is independent of $\mathcal{A}_i$ we have
    \begin{equation}
        \mathbb{P}[F \cap \{ i < S \} \cap \{\gamma_{i+1} \in D\}] = \mathbb{P}[F \cap \{ i < S \}] \mathbb{P}[\{\gamma_{i+1} \in D\}]. \label{eq:some_independence}
    \end{equation}

    We also know that for each integer $k \leq i$ we have $F \cap \{S = k\} \in \sigma(\mathcal{A}_i, \hat{\mathcal{A}}_i)$. This means that for each $k\leq i$ we can write
    \begin{equation*}
        F \cap \{S = k\} = \bigsqcup_{j=1}^{\infty} A_j \cap B_j
    \end{equation*}
    with $A_j \in \mathcal{A}_i$, $A_j \subset \{S = k\}$ and $B_j \in \hat{\mathcal{A}}$. Here $\bigsqcup$ denotes a disjoint union.

    Since $\hat{\mathcal{A}}$ is conditionally independent of $\gamma_{S+1}, \gamma_{S+2}, \dots$ given $\gamma_1, \gamma_2, \dots, \gamma_S$ and $A_j \in \sigma(\gamma_1, \gamma_2, \dots, \gamma_S)$ we have
    \begin{align*}
        \mathbb{P}[A_j \cap B_j \cap \{\gamma_{i+1} \in D\}] & = \mathbb{P}[A_j \cap B_j \cap \{\gamma_{S+i+1-k} \in D\}]\\
        & = \mathbb{P}[B_j \cap \{\gamma_{S+i+1-k} \in D\} | A_j] \mathbb{P}[A_j]\\
        &= \mathbb{P}[B_j | A_j] \mathbb{P}[\{\gamma_{S+i+1-k} \in D\} | A_j] \mathbb{P}[A_j]\\
        &= \mathbb{P}[A_j \cap B_j] \mathbb{P}[\{\gamma_{S+i+1-k} \in D\}] \\
        &= \mathbb{P}[A_j \cap B_j] \mathbb{P}[\{\gamma_{i+1} \in D\}].
    \end{align*}
    Summing this result over $j$ gives
    \begin{equation*}
        \mathbb{P}[F \cap \{S = k\} \cap \{\gamma_{i+1} \in D\}] = \mathbb{P}[F \cap \{S = k\}] \mathbb{P}[ \{\gamma_{i+1} \in D\}].
    \end{equation*}
    Summing over $k$ and adding \eqref{eq:some_independence} completes the proof.
\end{proof}

\subsubsection{Regular conditional probability} \label{section:r_c_p}

In order to understand our decomposition \eqref{eq:x_decomp} after conditioning on $\mathcal{A}$ and in order to prove Theorem \ref{theorem:prob_exist} we need to introduce the concept of regular condition probability.

For a more comprehensive text on regular conditional distributions see for example \cite[Chapter 8]{KLENKE_2014}. Some readers may be more familiar with the use of conditional measures as described in for example \cite[Chapter 5]{EINSIEDLER_WARD_2010}. These two concepts are equivalent.
	
\begin{definition}[Markov Kernel]
Let $(\Omega_1, \mathcal{A}_1)$ and $(\Omega_2, \mathcal{A}_2)$ be measurable spaces. We say that a function $\kappa: \Omega_1 \times \mathcal{A}_2 : \to [0, 1]$ is a \emph{Markov Kernel} on $(\Omega_1, \mathcal{A}_1)$ and $(\Omega_2, \mathcal{A}_2)$ if:
\begin{itemize}
\item For any $A_2 \in \mathcal{A}_2$ the function $\omega_1 \mapsto \kappa(\omega_1, A_2)$ is $\mathcal{A}_1$ - measurable
\item For any $\omega_1 \in \Omega_1$ the function $A_2 \mapsto \kappa(\omega_1, A_2)$ is a probability measure.
\end{itemize}
\end{definition}
	
\begin{definition}
Let $(\Omega, \mathcal{F}, \mathbb{P})$ be a probability space, let $(E, \xi)$ be a measurable space, and let $Y : (\Omega, \mathcal{F}) \to (E, \xi)$ be a random variable. Let $\mathcal{A} \subset \mathcal{F}$ be a $\sigma$-algebra. Then we say that a Markov kernel
\begin{equation*}
\kappa_{Y, \mathcal{A}}: \Omega \times \xi \to [0,1]
\end{equation*}
on $(\Omega, \mathcal{A})$ and $(E, \xi)$ is a \emph{regular conditional distribution} for $Y$ given $\mathcal{A}$ if
\begin{equation*}
\kappa_{Y, \mathcal{A}}(\omega, B) = \mathbb{P}[Y \in B | \mathcal{A}]
\end{equation*}
for all $B \in \xi$ and almost all $\omega \in \Omega$. 

In other words we require
\begin{equation*}
\mathbb{P} \left[ A \cap \left\{ Y \in B\right\}\right] = \mathbb{E} \left[ \mathbb{I}_A \kappa_{Y, \mathcal{A}}(\cdot, B) \right] \text{ for all } A \in \mathcal{A}, B \in \xi.
\end{equation*}
\end{definition}

In the case where $Y$ is as above and $X$ is another random variable taking values in some measurable space $(E', \xi')$ then  we let the regular conditional distribution of $Y$ given $X$ refer to the regular conditional distribution of $Y$ given $\sigma(X)$. For this definition to be useful we need the following theorem.
\begin{theorem}
Let $(\Omega, \mathcal{F}, \mathbb{P})$ be a probability space, let $(E, \xi)$ be a standard Borel space, and let $Y : (\Omega, \mathcal{F}) \to (E, \xi)$ be a random variable. Then given any $\sigma$-algebra $\mathcal{A} \subset \mathcal{F}$ there exists a regular conditional distribution for $Y$ given $\mathcal{A}$.
\end{theorem}
\begin{proof}
This is \cite[Theorem 8.37]{KLENKE_2014}.
\end{proof}

\begin{definition} \label{definition:r_c_p}
Given some random variable $Y$ and some $\sigma$- algebra $\mathcal{A} \subset \mathcal{F}$ (or random variable $X$) we will write $(Y|\mathcal{A})$ (or $(Y|X)$) to mean the regular conditional distribution of $Y$ given $\mathcal{A}$ (or given $X$).

We also let $[Y|\mathcal{A}]$ (or $[Y|X]$) denote random variables defined on a different probability space to $Y$ which have law $(Y|\mathcal{A})$ (or $(Y|X)$).
\end{definition}

One can easily check that if the regular conditional distribution exists then it is unique up to equality almost everywhere.

Next we will need the following simple facts about regular condition distributions.

\begin{definition}
Let $(\Omega, \mathcal{F}, \mathbb{P})$ be a probability space and let $\mathcal{A} \subset \mathcal{F}$ be a $\sigma$-algebra. We say that two $\sigma$- algebras $\mathcal{G}_1, \mathcal{G}_2 \subset \mathcal{F}$ are conditionally independent given $\mathcal{A}$ if for any $U \in \mathcal{G}_1$ and $V \in \mathcal{G}_2$ we have
\begin{equation*}
\mathbb{P}[U \cap V|\mathcal{A}] = \mathbb{P}[ U |\mathcal{A}]  \mathbb{P}[ V|\mathcal{A}] 
\end{equation*}
almost surely. Similarly we say that two random variables or a random variable and a $\sigma$-algebra are conditionally independent given $\mathcal{A}$ if the $\sigma$-algebras generated by them are conditionally independent given $\mathcal{A}$.
\end{definition}

Now we have these three lemmas.

\begin{lemma} \label{lemma:r_c_d_conditional_independence}
Let $(\Omega, \mathcal{F}, \mathbb{P})$ be a probability space and let $\mathcal{A} \subset \mathcal{F}$ be a $\sigma$-algebra. Let $g$ and $x$ be random variables on $(\Omega, \mathcal{F}, \mathbb{P})$ with $g$ taking values in $\Gp$ and with $x$ taking values in $X$ where $X$ is either $\Gp$ or $\B$. Suppose that $g$ and $x$ are conditionally independent given $\mathcal{A}$. Then
\begin{equation*}
(gx |\mathcal{A}) = (g |\mathcal{A}) * (x |\mathcal{A})
\end{equation*}
almost surely.
\end{lemma}

\begin{proof}
This follows by essentially the same proof as the proof that the law of $gx$ is the convolution of the laws of $g$ and of $x$ and is left to the reader.
\end{proof}

\begin{lemma} \label{lemma:r_c_d_subsigma_conditional_ind}
Let $(\Omega, \mathcal{F}, \mathbb{P})$ be a probability space and let $\mathcal{A} \subset \mathcal{F}$ be a $\sigma$-algebra. Let $g$ be a random variable taking values in some measurable space $(X, \xi)$. Let $\mathcal{G}$ be a $\sigma$-algebra such that
\begin{equation*}
\mathcal{A} \subset \mathcal{G} \subset \mathcal{F}
\end{equation*}
and $g$ is independent of $\mathcal{G}$ conditional on $\mathcal{A}$. Then
\begin{equation*}
(g |\mathcal{G}) = (g |\mathcal{A})
\end{equation*}
\end{lemma}

\begin{proof}
This is immediate from the definitions of the objects involved.
\end{proof}

\begin{lemma} \label{lemma:r_c_d_measurable_means_delta}
Let $(\Omega, \mathcal{F}, \mathbb{P})$ be a probability space and let $\mathcal{A} \subset \mathcal{F}$ be a $\sigma$-algebra. Let $g$ be a random variable taking values in some measurable space $(X, \xi)$. Suppose that $g$ is $\mathcal{A}$-measurable. Then
\begin{equation*}
(g |\mathcal{A}) = \delta_g
\end{equation*}
almost surely.
\end{lemma}

\begin{proof}
This is immediate from the definitions of the objects involved.
\end{proof}

\begin{lemma} \label{lemma:r_c_d_complicated_formula}
Let $(\Omega, \mathcal{F}, \mathbb{P})$ be a probability space and let $\mathcal{A} \subset \mathcal{F}$ be a $\sigma$-algebra. Let $g$ be a random variable taking values in some measurable space $(X, \xi)$. Let $\mathcal{G}$ be a $\sigma$-algebra such that $\mathcal{A} \subset \mathcal{G} \subset \mathcal{F}$ and $g$ is $\mathcal{G}$ measurable. Let $A \in \mathcal{A}$ and construct the $\sigma$-algebra $\hat{\mathcal{A}}$ by
\begin{equation*}
\hat{\mathcal{A}} = \sigma (\mathcal{A}, \{ G \in \mathcal{G} : G \subset A \} ).
\end{equation*}
Then for almost all $\omega \in \Omega$ we have
\begin{equation*}
(g |\hat{\mathcal{A}})(\omega, \cdot) =
\begin{cases}
\delta_g & \text{if } \omega \in A\\
(g |\mathcal{A})(\omega, \cdot) & \text{otherwise.}
\end{cases}
\end{equation*}
\end{lemma}

\begin{proof}
Let
\begin{equation*}
Q(\omega, \cdot) :=
\begin{cases}
\delta_g & \text{if } \omega \in A\\
(g |\mathcal{A})(\omega, \cdot) & \text{otherwise.}
\end{cases}
\end{equation*}
We will show that $Q$ satisfies the conditions of being a regular conditional distribution for $g$ given $\hat{\mathcal{A}}$. Clearly $Q$ is a Markov kernel. Now let $D \in \hat{\mathcal{A}}$ and let $B \in \xi$. We simply need to show that
\begin{equation}
\mathbb{P} [D \cap \{g \in B\} ] = \mathbb{E} [\mathbb{I}_D Q(\cdot, B)]. \label{eq:r_c_d_condition_thing}
\end{equation}
First suppose that $D \subset A$. In this case the right-hand side of \eqref{eq:r_c_d_condition_thing} becomes $\mathbb{E} [\mathbb{I}_D \mathbb{I}_{g \in B}]$ which is trivially equal to the left-hand side.

Now suppose that $D \subset A^C$. This means that $D \in \mathcal{A}$. In this case by the definition of $(g |\mathcal{A})(\omega, \cdot)$ we know that \eqref{eq:r_c_d_condition_thing} is satisfied.

The general case follows by summing.

\end{proof}
We also need some results about the entropy of regular condition distributions.

\begin{definition}
Given some random variable $Y$ and a $\sigma$-algebra $\mathcal{A} \subset \mathcal{F}$ we define $H((Y|\mathcal{A}))$ to be the random variable
\begin{equation*}
H((Y|\mathcal{A})) : \omega \mapsto H((Y|\mathcal{A})(\omega, \cdot))
\end{equation*}
where $(Y|\mathcal{A})(\omega, \cdot)$ is the regular conditional distribution for $Y$ given $\mathcal{A}$. Similarly given some random variable $X$ we let $H((Y|X)) := H((Y|\sigma(X)))$.
\end{definition}

\begin{lemma} \label{lemma:relative_ent_expectation}
Let $X_1$ and $X_2$ be two random variables with finite entropy and finite joint entropy. Then
\begin{equation*}
H(X_1 | X_2) = \mathbb{E}[H((X_1 | X_2))].
\end{equation*}
\end{lemma}
	
\begin{proof}
This is just the chain rule for conditional distributions. It follows from a simple computation and a proof may be found in \cite[Proposition 3]{VIGNEAUX_2021}.
\end{proof}

\begin{lemma} \label{lemma:invariance_entropy}
Let $g$ be a random variable taking values in $\Gp$, let $\mathcal{A}$ be a $\sigma$-algebra, and let $a$ be a $\mathcal{A}$-measurable random variable taking values in $\Gp$. Then
\begin{equation*}
H((ag | \mathcal{A})) = H((g | \mathcal{A}))
\end{equation*}
almost surely. In particular if $h \in \Gp$ is fixed then
\begin{equation*}
H(hg ) = H(g).
\end{equation*}
\end{lemma}
\begin{proof}
For the first part note that $[ag | \mathcal{A}] = a[g | \mathcal{A}]$ almost surely. Also note that by the left invariance of the Haar measure
\begin{equation*}
H(a[g | \mathcal{A}]) = H([g | \mathcal{A}]).
\end{equation*}
The last part follows trivially by the first part.
\end{proof}
\section{Order $k$ Detail} \label{section:detail}

In this section we discuss the basic properties of detail around a scale. We will recall basic properties of detail from \cite{Kittle2024ASENS} and introduce order $k$ detail and prove some properties of it.

Detail is a quantitative measure of the smoothness of a measure at a given scale. The detail of a measure at some scale $r>0$ is close to $1$ if, for example, the measure is supported on a number of disjoint intervals of length much smaller than $r$, which are separated by a distance much greater than $r$. The detail of a measure is small if, for example, the measure is uniform on an interval of length significantly greater than $r$.

We now explain how we extend the concept of detail to measures taking values in $\B$ or equivalently $\Bim$. For this we need the following.

\begin{definition}
Given some $y>0$ let $\tilde{\eta}_y$ be the density of the pushforward of the normal distribution with mean $0$ and variance $y$ onto $\Bim$. In other words given $x \in \Bim$ let
\begin{equation*}
\tilde{\eta}_y(x) := \sum_{u \in x} \eta_y(u). \qedhere
\end{equation*}
\end{definition}
We will also use the following notation.
\begin{definition}
Given some $y>0$ let $\tilde{\eta}_y'$ be defined by
\begin{equation*}
\tilde{\eta}_y' := \frac{\partial}{\partial y} \tilde{\eta}_y. \qedhere
\end{equation*}
\end{definition}

We now define the following.
\begin{definition} \label{definition:detail}
Given a probability measure $\lambda$ on $\Bim$ and some $r>0$ we define the \emph{detail of $\lambda$ around scale $r$} by
\begin{equation*}
s_r(\lambda) := r^2 \sqrt{\frac{\pi e}{2}} \lnb \lambda *  \tilde{\eta}_{r^2}' \rnb_1.
\end{equation*}
\end{definition}
Similarly we define the detail of a probability measure on $\B$ to be the detail of the pushforward measure under $\phi$ and we define the detail of a random variable to be the detail of its law. The factor $ r^2 \sqrt{\frac{\pi e}{2}}$ exists to ensure that $s_r(\lambda) \in [0, 1]$. The smaller the value of detail around a scale the smoother the measure is at that scale.

\begin{remark}
We motivate our definition of detail as follows. Earlier work on stationary measures, including \cite{VARJU_BERLIAND_2020}, \cite{HOCHMAN_2014}, \cite{HOCHMAN_SOLOMYAK_2017} and \cite{VARJU_2019} studied quantities like $$H(\mu * F_{r_1}) - H(\mu * F_{r_2})$$
where $F_r$ is a smoothing function associated to scale $r$ (for example the law of the normal distribution with standard deviation $r$ or the law of a uniform random variable on $[0, r]$). Motivated by this and the work of Shmerkin \cite{SHMERKIN_2019}, it is natural to study quantities like
\begin{equation*}
    \| \mu * F_{r_1} \|_p - \| \mu * F_{r_2} \|_p.
\end{equation*}
However it turns out to be more useful to study
\begin{equation*}
    \| \mu * F_{r_1} - \mu * F_{r_2} \|_p
\end{equation*}
at least when $p = 1$. Detail is an infinitesimal version of this quantity
with Gaussian smoothing.

The Gaussian is chosen because the heat equation plays an important role in the proof of Lemma \ref{lemma:detail_order_k_detail_bound_many} and \cite[Lemma 2.5]{Kittle2024ASENS}. The property that the convolution of a Gaussian with a Gaussian is another Gaussian also plays a key role.
\end{remark}

In Section \ref{section:order_k_detail} we introduce a new quantity which we refer to as order $k$ detail. In Section \ref{section:order_k_detail_to_detail} we use this to bound detail. In Section \ref{section:wasserstein_distance_bound} we prove Lemma \ref{lemma:detail_wasserstein}. Finally in Section \ref{section:small_random_variables_bound} we prove Lemma \ref{lemma:small_rvs_to_detail}.

\subsection{Order $k$ detail} \label{section:order_k_detail}

We can now define the order $k$ detail around a scale.
\begin{definition}[Order $k$ detail around a scale] \label{definition:order_k_detail}
Given a probability measure $\lambda$ on $\Bim$ and some $k \in \N$ we define the \emph{order $k$ detail of $\lambda$ around scale $r$}, which we will denote by $s_r^{(k)}(\lambda)$, by
\begin{equation*}
s_r^{(k)}(\lambda) := r^{2k} \left( \frac{\pi e}{2} \right) ^{k / 2} \lnb \lambda * \left. \frac{\partial^k}{\partial y^k} \tilde{\eta}_y \right|_{y=kr^2} \rnb_1. 
\end{equation*}
\end{definition}

We also define the order $k$ detail of a measure on $\B$ to be the order $k$ detail of the pushforward measure under $\phi$ and define the order $k$ detail of a random variable to be the order $k$ detail of its law. It is worth noting that $s_r^{(1)}( \cdot) = s_r(\cdot)$. We will now prove some basic properties of order $k$ detail.

\begin{lemma} \label{lemma:detail_order_k_detail_bound_many}
Let $\lambda_1, \lambda_2, \dots, \lambda_k$ be probability measures on $\Bim$. Then we have
\begin{equation*}
s_r^{(k)}(\lambda_1 * \lambda_2 * \dots * \lambda_k) \leq s_r(\lambda_1) s_r(\lambda_2) \dots s_r(\lambda_k).
\end{equation*}
\end{lemma}

This is \eqref{eq:trivial_srk_bound} from Section \ref{section:outline}.

\begin{proof}
From the heat equation we know that
\begin{equation*}
    \frac{\partial}{\partial y} \eta_y(x) = \frac{1}{2} \frac{\partial^2}{\partial x^2} \eta_y(x).
\end{equation*}
Therefore by standard properties of convolution we have
\begin{align*}
\left. \frac{\partial^k}{\partial y^k} \tilde{\eta}_y \right|_{y=kr^2} & = 2^{-k} \frac{\partial^{2k}}{\partial x^{2k}} \tilde{\eta}_{kr^2}\\
&= \underbrace{\left(\frac{1}{2} \frac{\partial^2}{\partial x^2} \tilde{\eta}_{r^2} \right) * \left(\frac{1}{2} \frac{\partial^2}{\partial x^2} \tilde{\eta}_{r^2} \right) * \dots * \left(\frac{1}{2} \frac{\partial^2}{\partial x^2} \tilde{\eta}_{r^2} \right)}_{k \text{ times}}\\
&= \underbrace{\tilde{\eta}_{r^2}' * \tilde{\eta}_{r^2}' * \dots * \tilde{\eta}_{r^2}'}_{k \text{ times}}
\end{align*}
and therefore
\begin{equation*}
\lambda_1 * \lambda_2 * \dots * \lambda_k * \left. \frac{\partial^k}{\partial y^k} \tilde{\eta}_y \right|_{y=kr^2} = \lambda_1 * \tilde{\eta}_{r^2}' * \lambda_2 * \tilde{\eta}_{r^2}' *  \dots * \lambda_k * \tilde{\eta}_{r^2}'.
\end{equation*}
This means
\begin{equation*}
\lnb \lambda_1 * \lambda_2 * \dots * \lambda_k * \left. \frac{\partial^k}{\partial y^k} \tilde{\eta}_y \right|_{y=kr^2} \rnb_1 \leq \lnb \lambda_1 * \tilde{\eta}_{r^2}' \rnb_1 \cdot \lnb \lambda_2 * \tilde{\eta}_{r^2}' \rnb_1 \cdot \dots \cdot \lnb \lambda_k * \tilde{\eta}_{r^2}' \rnb_1.
\end{equation*}
The result follows.
\end{proof}
The following corollary will be useful.
\begin{corollary} \label{corollary:trivial_upper_bound_on_detail}
Suppose that $\lambda$ is a probability measure on $\Bim$. Then
\begin{equation*}
s_r^{(k)}(\lambda) \leq 1.
\end{equation*}
\end{corollary}
\begin{proof}
This is immediate by letting all but one of the measures in Lemma \ref{lemma:detail_order_k_detail_bound_many} be a delta function.
\end{proof}

There is no reason to assume that the bound in Corollary \ref{corollary:trivial_upper_bound_on_detail} is optimal for any $k \geq 2$. Indeed it is fairly simple to show that it is not. However the trivial upper bound of $1$ will still prove useful.

We also need the following corollary of Lemma \ref{lemma:small_rvs_to_detail} (which will be proven in Section \ref{section:small_random_variables_bound}) and Lemma \ref{lemma:detail_order_k_detail_bound_many}.

\begin{corollary}\label{corollary:small_rvs_to_order_k_detail}
For every $\alpha > 0$ there exists some $C > 0$ such that the following is true. Let $X_1, X_2, \dots, X_n$ be independent random variables taking values in $\Bim$ such that $|X_i| < s$ almost surely for some $s>0$. Let $\sigma > 0 $ be defined by $\sigma^2 = \sum_{i=1}^n \var X_i$. Let $ r \in (s, \sigma)$. Let $k \in \Z_{>0}$ and suppose that
\begin{equation*}
\frac{r}{s} \geq C
\end{equation*}
and
\begin{equation*}
    \frac{\sigma^2}{r^2} \geq Ck. 
\end{equation*}
Then
\begin{equation*}
s_r^{(k)}(X_1 + X_2 + \dots + X_n) \leq \alpha^k.
\end{equation*}
\end{corollary}

\begin{proof}
    Let $C_1$ be the $C$ from Lemma \ref{lemma:small_rvs_to_detail} with this value of $\alpha$. Suppose that $$\frac{r}{s} \geq \max \{C_1, 1\}$$ and $$\frac{\sigma^2}{r^2} \geq (C_1^2 + 1) k.$$
    Partition $[1, n]\cap \Z$ into $k$ sets $J_1, J_2, \dots, J_k$ such that for each $i = 1, 2, \dots, k$ we have $$\sum_{j \in J_i} \var X_j \geq C_1^2 r^2.$$ This is possible by a greedy algorithm. Note that by Lemma \ref{lemma:small_rvs_to_detail} this means
    $$s_r\left( \sum_{j \in J_i} X_j \right) < \alpha.$$ Noting that $$X_1 + X_2 + \dots + X_n = \sum_{i=1}^{k} \left( \sum_{j \in J_i} X_j \right)$$ and applying Lemma \ref{lemma:detail_order_k_detail_bound_many} gives the required result.
\end{proof}

\subsection{Bounding detail using order k detail} \label{section:order_k_detail_to_detail}

The purpose of this subsection is to prove Lemma \ref{lemma:ind_srk_to_sr}. For this we first need the following result.
 
\begin{lemma} \label{lemma:srk_to_srk_m_1}
Let $k$ be an integer greater than $1$ and suppose that $\lambda$ is a probability measure on $\Bim$. Suppose that $a, b, c > 0$ and $\alpha \in (0,1)$. Suppose that $a < b$ and that for all $r \in [a, b]$ we have
\begin{equation}
s_r^{(k)}(\lambda) \leq \alpha + c r^{2k}. \label{eq:requirement_srk_ind}
\end{equation}
Then for all $r \in \left[ a\sqrt{\frac{k}{k-1}}, b \sqrt{\frac{k}{k-1}}\right]$ we have
\begin{equation*}
s_r^{(k-1)}(\lambda) \leq  \frac{k}{k-1}\sqrt{\frac{2e}{\pi}} \alpha +  \left( b^{-2k+2} + kb^2c \right) r^{2(k-1)}.
\end{equation*}
\end{lemma}

\begin{proof}
Recall that
\begin{equation*}
s_r^{(k)}(\lambda)  = r^{2k} \left( \frac{\pi e}{2} \right) ^{\frac{k}{2}} \lnb \lambda * \left. \frac{\partial ^k}{\partial y^k} \tilde{\eta}_y \right|_{y = kr^2} \rnb_1.
\end{equation*}
This means by \eqref{eq:requirement_srk_ind} that when $y = k r^2$ we have
\begin{align*}
\lnb \lambda *  \frac{\partial ^k}{\partial y^k} \tilde{\eta}_y \rnb_1 & \leq \alpha r^{-2 k} \left( \frac{\pi e}{2} \right)^{ - \frac{k}{2}} + c\left( \frac{\pi e}{2} \right)^{ - \frac{k}{2}} \\
& =\alpha y^{-k} k^{k} \left( \frac{\pi e}{2} \right)^{ - \frac{k}{2}} + c\left( \frac{\pi e}{2} \right)^{ - \frac{k}{2}} 
\end{align*}
for all $y \in [ka^2, kb^2]$. This means that for $y \in [ka^2, kb^2]$ we have
\begin{align}
\MoveEqLeft \lnb \lambda * \frac{\partial ^{k-1}}{\partial y^{k-1}} \tilde{\eta}_y \rnb_1 \nonumber\\
& \leq \lnb \lambda * \left. \frac{\partial ^{k-1}}{\partial u^{k-1}} \tilde{\eta}_u \right|_{u=kb^2} \rnb_1 + \int_y^{kb^2} \lnb \lambda * \frac{\partial ^{k}}{\partial u^{k}} \tilde{\eta}_u \rnb_1 \, du\nonumber\\
& \leq \lnb \left. \frac{\partial^{k-1}}{\partial u^{k-1}} \tilde{\eta}_u \right|_{u=kb^2} \rnb_1 + \int_y^{kb^2} \alpha u^{-k} k^{k} \left( \frac{\pi e}{2} \right)^{ - \frac{k}{2}} + c\left( \frac{\pi e}{2} \right)^{ - \frac{k}{2}}  \, du\nonumber\\
& \leq \left( \frac{kb^2}{k-1} \right)^{-k+1} \left( \frac{\pi e}{2}\right)^{-(k-1)/2} +  \alpha \frac{y^{-k+1}}{k-1} k^{k} \left( \frac{\pi e}{2} \right)^{-\frac{k}{2}} + kb^2 c\left( \frac{\pi e}{2} \right)^{ - \frac{k}{2}} \label{eq:temp_ind_srk}
\end{align}
where in \eqref{eq:temp_ind_srk} we bound $\lnb \left. \frac{\partial^{k-1}}{\partial u^{k-1}} \tilde{\eta}_u \right|_{u=kb^2} \rnb_1$ using the fact that order $k-1$ detail is at most one, we bound $\int_y^{kb^2} \alpha u^{-k} k^{k} \left( \frac{\pi e}{2} \right)^{ - \frac{k}{2}}  \, du$ by $\int_y^{\infty} \alpha u^{-k} k^{k} \left( \frac{\pi e}{2} \right)^{ - \frac{k}{2}}  \, du$ and bound $\int_y^{kb^2}  c\left( \frac{\pi e}{2} \right)^{ - \frac{k}{2}}  \, du$ by $\int_0^{kb^2} c\left( \frac{\pi e}{2} \right)^{ - \frac{k}{2}}  \, du$. Noting that
\begin{equation*}
\left( \frac{k}{k-1} \right)^{-k+1}  < 1
\end{equation*}
and
\begin{equation*}
\left( \frac{\pi e}{2} \right)^{ - \frac{1}{2}} < 1
\end{equation*}
we get
\begin{equation*}
\lnb \lambda * \frac{\partial ^{k-1}}{\partial y^{k-1}} \eta_y \rnb_1 \leq  \alpha \frac{y^{-k+1}}{k-1} k^{k} \left( \frac{\pi e}{2} \right)^{-\frac{k}{2}} + \left( b^{-2k+2} + kb^2 c \right) \left( \frac{\pi e}{2} \right)^{ - \frac{k-1}{2}}.
\end{equation*}
Substituting in the definition of order $k$ detail gives
\begin{align*}
s_r^{(k-1)}(\lambda) & = r^{2(k-1)} \left( \frac{\pi e}{2} \right) ^{\frac{k-1}{2}}  \lnb \left. \lambda * \frac{\partial ^{k-1}}{\partial y^{k-1}} \tilde{\eta}_y \right|_{y=(k-1)r^2} \rnb_1\\
& \leq r ^{2(k-1)} \left( \frac{\pi e}{2} \right)^{-\frac{1}{2}} \alpha \frac{((k-1)r^2)^{-k+1}}{k-1} k^k + r ^{2(k-1)} \left( \frac{\pi e}{2} \right)^{-\frac{1}{2}} \left( b^{-2k+2} + kb^2 c \right)
\end{align*}
and so we have
\begin{equation*}
s_r^{(k-1)}(\lambda) \leq \alpha \sqrt{\frac{2}{\pi e}} \left( 1 + \frac{1}{k-1}\right)^{k}+ ( b^{-k+1} + kcb) r^{2 (k-1)}
\end{equation*}
for all $r \in \left[ a\sqrt{\frac{k}{k-1}}, b\sqrt{\frac{k}{k-1}} \right]$. Noting that  $\left( 1 + \frac{1}{k-1}\right)^{k} \leq \frac{k}{k-1}e$ gives the required result.
\end{proof}

We apply this inductively to prove Lemma \ref{lemma:ind_srk_to_sr}.

\begin{proof}[Proof of Lemma \ref{lemma:ind_srk_to_sr}]
Using Lemma \ref{lemma:srk_to_srk_m_1} we will prove by induction for $j = k, k-1, \dots, 1$ that for all $r\in \left[ a\sqrt{\frac{k}{j}} , b \sqrt{\frac{k}{j}} \right]$ we have
\begin{equation*}
 s_r^{(j)}(\lambda) \leq \alpha \frac{k}{j} \left( \frac{2e}{\pi} \right) ^ {\frac{k-j}{2}} + \frac{k!}{j!} b^{-2j} r^{2j}.
\end{equation*}

The case $j = k$ follows by the conditions of the lemma. Suppose that for all  $r\in \left[ a\sqrt{\frac{k}{j+1}} , b \sqrt{\frac{k}{j+1}} \right]$  we have
\begin{equation*}
s_r^{(j+1)}(\lambda) \leq \alpha \frac{k}{j+1} \left( \frac{2e}{\pi} \right) ^ {\frac{k-j-1}{2}} + \frac{k!}{(j+1)!} b^{-2j-2} r^{2(j+1)}.
\end{equation*}
Then by Lemma \ref{lemma:srk_to_srk_m_1} for all $r>0$ such that $r\in \left[ a\sqrt{\frac{k}{j}} , b \sqrt{\frac{k}{j}} \right]$  we have
\begin{align*}
s_r^{(j)}(\lambda) & \leq \alpha \frac{k}{j} \left( \frac{2e}{\pi} \right) ^ {\frac{k-j}{2}} + \left( b^{-2j} + jb^2 \left(\frac{k!}{(j+1)!}  b^{-2j-2} \right) \right) r^{2j}\\
& \leq \alpha \frac{k}{j} \left( \frac{2e}{\pi} \right) ^ {\frac{k-j}{2}} +  \left( \frac{k!}{(j+1)!} b^{-2j} + jb^2 \left(\frac{k!}{(j+1)!}  b^{-2j-2} \right) \right) r^{2j}\\
& = \alpha \frac{k}{j} \left( \frac{2e}{\pi} \right) ^ {\frac{k-j}{2}} + (j+1)\frac{k!}{(j+1)!} b^{-2j} r^{2j}\\
& = \alpha \frac{k}{j} \left( \frac{2e}{\pi} \right) ^ {\frac{k-j}{2}} + \frac{k!}{j!} b^{-2j} r^{2j}
\end{align*}
as required. Lemma \ref{lemma:ind_srk_to_sr} follows easily from the $j=1$ case.
\end{proof}

\subsection{Wasserstein distance bound} \label{section:wasserstein_distance_bound}

In this subsection we will bound the difference in order $k$ detail between two measures in terms of the Wasserstein distance between those two measures. Specifically we will prove Lemma \ref{lemma:detail_wasserstein}. First we need to define Wasserstein distance.

\begin{definition}[Coupling]
Given two probability measures $\lambda_1$ and $\lambda_2$ on a set $X$ we say that a \emph{coupling} between $\lambda_1$ and $\lambda_2$ is a measure $\gamma$ on $X \times X$ such that $\gamma(\cdot \times X) = \lambda_1(\cdot)$ and $\gamma(X \times \cdot) = \lambda_2(\cdot)$.
\end{definition}

\begin{definition}[Wasserstein distance] \label{definition:wasserstein_distance}
Given two probability measures $\lambda_1$ and $\lambda_2$ on $\Bim$ the Wasserstein distance between $\lambda_1$ and $\lambda_2$, which we will denote by $\mathcal{W}_1(\lambda_1, \lambda_2)$, is given by
\begin{equation*}
\mathcal{W}_1(\lambda_1, \lambda_2) := \inf_{\gamma \in \Gamma} \int_{(\Bim) ^2} |x-y| \, \gamma(dx, dy)
\end{equation*}
where $\Gamma$ is the set of couplings between $\lambda_1$ and $\lambda_2$.
\end{definition}

We can now prove Lemma \ref{lemma:detail_wasserstein}.

\begin{proof}[Proof of Lemma \ref{lemma:detail_wasserstein}]
Let $X$ and $Y$ be random variables with laws $\lambda_1$ and $\lambda_2$ respectively. Then we have

\begin{align*}
(\lambda_1 - \lambda_2) * \left. \frac{\partial^k}{\partial y^k} \tilde{\eta}_y\right| _{y = k r^2} (v) &= \mathbb{E} \left[  \left. \frac{\partial^k }{\partial y^k} \tilde{\eta}_y\right|_{y = k r^2} ( v- X) - \left. \frac{\partial^k}{\partial y^k} \tilde{\eta}_y\right|_{y = k r^2} ( v-Y ) \right].
\end{align*}
In particular
\begin{align*}
\left| (\lambda_1 - \lambda_2) * \left. \frac{\partial^k}{\partial y^k} \tilde{\eta}_y\right| _{y = k r^2} (v) \right| & \leq \mathbb{E} \left| \left. \frac{\partial^k }{\partial y^k} \tilde{\eta}_y\right|_{y = k r^2} ( v- X) - \left. \frac{\partial^k}{\partial y^k} \tilde{\eta}_y\right|_{y = k r^2} ( v-Y ) \right|.
\end{align*}
We note that
\begin{align*}
\left| \left. \frac{\partial^k }{\partial y^k} \tilde{\eta}_y\right|_{y = k r^2} ( v- X) - \left. \frac{\partial^k}{\partial y^k} \tilde{\eta}_y\right|_{y = k r^2} ( v-Y ) \right| & \leq
\int_{X}^{Y} \left| \left. \frac{\partial^{k+1} }{\partial x \partial y^k} \tilde{\eta}_y\right|_{y = k r^2} (v-u) \right| \, |du|
\end{align*}
where
\begin{equation*}
\int_x^y \cdot \, |du| 
\end{equation*}
is understood to be the integral along the shortest path between $x$ and $y$. This means that
\begin{align*}
\lnb (\lambda_1 - \lambda_2) * \left. \frac{\partial^k}{\partial y^k} \tilde{\eta}_y\right| _{y = k r^2} \rnb_1 & \leq \int_{\Bim} \mathbb{E} \left[ \int_{X}^{Y} \left| \left. \frac{\partial^{k+1} }{\partial x \partial y^k} \tilde{\eta}_y\right|_{y = k r^2} (v-u) \right| \, |du| \right] \, dv\\
& = \mathbb{E} \left[ \int_{X}^{Y} \int_{\Bim} \left| \left. \frac{\partial^{k+1} }{\partial x \partial y^k} \tilde{\eta}_y\right|_{y = k r^2} (v-u) \right| \, dv \, |du| \right]\\
& = \mathbb{E} \left[ \int_{X}^{Y} \lnb \left. \frac{\partial^{k+1} }{\partial x \partial y^k} \tilde{\eta}_y\right|_{y = k r^2} \rnb_1  \, |du| \right]\\
& = \lnb \left. \frac{\partial^{k+1} }{\partial x \partial y^k} \tilde{\eta}_y\right|_{y = k r^2} \rnb_1 \mathbb{E} |X -Y|.
\end{align*}

We now bound $\lnb \left. \frac{\partial^{k+1}}{\partial x \partial y^k} \tilde{\eta}_y\right|_{y = k r^2}\rnb_1 $. To do this note that
\begin{equation*}
\lnb \left. \frac{\partial^{k+1}}{\partial x \partial y^k} \tilde{\eta}_y\right|_{y = k r^2}\rnb_1 \leq \lnb \left. \frac{\partial^{k+1}}{\partial x \partial y^k} \eta_y\right|_{y = k r^2}\rnb_1.
\end{equation*}
By using the relation $\eta'_y = \frac{\partial^2}{\partial x^2} \eta_y$ in the same way as in the proof of Lemma \ref{lemma:detail_order_k_detail_bound_many} we get
\begin{equation*}
\left. \frac{\partial^{k+1}}{\partial x \partial y^k} \eta_y\right|_{y = k r^2} = \left. \frac{\partial}{\partial x } \eta_y \right|_{y = r^2} * \underbrace{\left. \frac{\partial}{\partial y} \eta_y \right|_{y = r^2}  * \left. \frac{\partial}{\partial y} \eta_y \right|_{y = r^2}  * \dots * \left. \frac{\partial}{\partial y} \eta_y \right|_{y = r^2} }_{k \text{ times}}
\end{equation*}
and so
\begin{equation*}
\lnb \left. \frac{\partial^{k+1}}{\partial x \partial y^k} \eta_y\right|_{y = k r^2}\rnb_1  \leq \lnb \frac{\partial}{\partial x } \eta_{r^2} \rnb_1 \cdot \lnb  \eta_{r^2}' \rnb_1^{k}.
\end{equation*}
		
Note that trivially there is some constant $C>0$ such that
\begin{equation*}
\lnb \frac{\partial}{\partial x } \eta_{r^2} \rnb_1 = C r^{-1}.
\end{equation*}
From the fact that detail is bounded above by $1$ we have
\begin{equation*}
\lnb \left. \frac{\partial}{\partial y} \eta_y \right|_{y = r^2} \rnb_1 =  r^{-2} \sqrt{\frac{2}{\pi e}}
\end{equation*}
meaning
\begin{equation*}
\lnb \left. \frac{\partial^{k+1}}{\partial x \partial y^k} \eta_y\right|_{y = k r^2}\rnb_1 \leq C r^{-2k-1}  \left( \frac{\pi e}{2} \right) ^{- \frac{k}{2}}.
\end{equation*}
Therefore
\begin{equation*}
r^{2k} \left( \frac{ \pi e}{2} \right) ^{\frac{k}{2}} \lnb \left. \frac{\partial^{k+1}}{\partial x \partial y^k} \eta_y\right|_{y = k r^2} \rnb_1 \leq C r^{-1}.
\end{equation*}
Choosing a coupling for $X$ and $Y$ which minimizes $\mathbb{E}|X-Y|$ gives the required result.
\end{proof}

\subsection{Small random variables bound} \label{section:small_random_variables_bound}

In this subsection we prove Lemma \ref{lemma:small_rvs_to_detail}. Recall that this gives a bound for the detail of the sum of many independent random variables each of which are contained in a small interval containing $0$ and have at least some variance. To prove this we will need the following quantitative version of the central limit theorem.

\begin{theorem} \label{theorem:srv_normal_wass}
Let $X_1, X_2, \dots, X_n$ be independent random variables taking values in $\R$ with mean $0$ and for each $i \in [1, n]$ let $\mathbb{E}[X_i^2] = \omega_i^2$ and $\mathbb{E}[|X_i|^3] = \gamma_i^3 < \infty$. Let $\omega^2 = \sum_{i=1}^n \omega_i^2$ and let $S = X_1 + \dots + X_n$. Then
	
\begin{equation*}
\mathcal{W}_1(S, \eta_{\omega^2}) \lesssim \frac{\sum_{i=1}^n \gamma_i^3}{\sum_{i=1}^n \omega_i^2}.
\end{equation*}

\end{theorem}
\begin{proof}
    Applying \cite[Theorem 1]{Erickson_1973} with $p=1$ and $\tau_k = \tau_k' = \infty$ for $k=1 \dots n$ and using the classical result that the Wasserstein distance between two real values random variables is equal to the $L^1$ distance between their cumulative distribution functions we get
    \begin{equation*}
        \mathcal{W}_1\left(\frac{S}{\omega}, \eta_{1}\right) \lesssim \frac{\sum_{i=1}^n \gamma_i^3}{\omega^3}.
    \end{equation*}
    The result follows.
\end{proof}

We are now ready to prove Lemma \ref{lemma:small_rvs_to_detail}.

\begin{proof}[Proof of Lemma \ref{lemma:small_rvs_to_detail}]
We will prove this in the case where the $X_i$ take values in $\R$. The case where they take values in $\Bim$ follows trivially from this case.

For $i=1, \dots, n$ let $X'_i = X_i - \mathbb{E}[X_i]$ and let $S' = \sum_{i=1}^n X_i'$. Note that $s_r(S) = s_r(S')$. Let $\mathbb{E}[|X_i'|^2] = \omega_i^2$ and $\mathbb{E}[|X_i'|^3] = \gamma_i^3$. Note that $\var X_i = \omega_i^2$ and so $\omega^2 = \sum_{i=1}^n \omega_i^2$. Note that almost surely $|X_i'| \leq 2 s $. This means that $\gamma_i^3 \leq 2 s \omega_i^2$. Therefore by Theorem \ref{theorem:srv_normal_wass} we have
\begin{equation*}
\mathcal{W}_1 \left( S', \eta_{\omega^2} \right) \leq O(s).
\end{equation*}
We also compute
\begin{align*}
s_r(\eta_{\omega^2}) &= \frac{ \lnb \eta_{r^2 + \omega^2}' \rnb_1}{\lnb \eta_{r^2} ' \rnb_1} \\
& = \frac{r^2}{r^2 + \omega^2}
\end{align*}
and so noting that $s_r(\cdot) = s_r^{(1)}(\cdot)$ we have by Lemma \ref{lemma:detail_wasserstein} that
\begin{align*}
s_r(S) &= s_r(S')\\
& \leq O\left(\frac{s}{r} \right) + \frac{r^2}{r^2 + \omega^2}.
\end{align*}
This gives the required result.
\end{proof}

\section{Computations for the Taylor Expansion} \label{section:taylor}

In this Section we will prove Proposition \ref{proposition:intro_decomp_detail}. We also do some computations on the derivatives $\zeta_i \in \As$ from Proposition \ref{proposition:intro_decomp_detail} which will later enable us to give bounds on the order $k$ detail of random variables produced by allowing the $u^{(i)}$ in the proposition to be appropriately chosen independent random variables. First we will give more detail on our notation.

Given normed vector spaces $V$ and $W$, some vector $v \in V$, and a function $f:V \to W$ which is differentiable at $v$ we write $D_vf(v)$ for the linear map $V \to W$ which is the derivative of $f$ at $v$. Similarly if $f$ is $n$ times differentiable at $v$ we write $D^n_v f(v)$ for the $n$-multi-linear map $V^n \to W$ which is the $n$th derivative of $f$ at $v$.

Now given some normed vector space $V$, some vector $v \in V$, and a function $f:V \to \Bim$ which is $n$ times differentiable at $v$ we can find some open set $U \subset V$ containing $v$ such that there exists some function $\tilde{f} : U \to \R$ which is $n$ times differentiable at $v$ and such that for all $u \in U$ we have
\begin{equation*}
f(u) = \tilde{f}(u) + \pi \Z.
\end{equation*}
In this case we take $D^n_v f (v)$ to be $D^n_v \tilde{f}(v)$. Clearly this does not depend on our choice of $U$ or $\tilde{f}$. Similarly given a sufficiently regular function $f : \Bim \to V$ we take $D_vf(v)$ to be $D_v \tilde{f}(v)$ where $\tilde{f} : \R \to V$ is defined by
\begin{equation*}
\tilde{f}(x) = f(x + \pi \Z).
\end{equation*}

As well as proving Proposition \ref{proposition:intro_decomp_detail} we also derive some bounds on the size of various first derivatives.
\begin{definition} \label{definition:e_vec}
Given some $b \in \B$ we let $\varrho_b \in \As$ be defined by
\begin{equation*}
\varrho_b = D_u \phi(\exp(u) b)|_{u=0}
\end{equation*}
\end{definition}

\begin{proposition} \label{proposition:large_first_derivatives}
For all $t >0$ there is some $\delta > 0$ such that the following is true. Let $v \in \A$ be a unit vector. Then there exist some $a_1, a_2 \in \R$ such that if
\begin{equation*}
b \in \B \backslash \phi^{-1}((a_1, a_1 + t) \cup (a_2, a_2 + t))
\end{equation*}
then
\begin{equation*}
| \varrho_b(v) | \geq \delta.
\end{equation*}
Furthermore we may construct $a_1$ and $a_2$ in such a way that they are measurable functions of $v$.
\end{proposition}

Motivated by this we have the following definition.

\begin{definition} \label{definition:u_t}
Let $t$, $v$, $a_1$, and $a_2$ be as in Proposition \ref{proposition:large_first_derivatives} and let $\varepsilon > 0$. Then we define $U_t(v)$ and $U_{t, \varepsilon}(v)$ by
\begin{equation*}
U_t(v) := \B \backslash \phi^{-1} ((a_1, a_1 + t) \cup (a_2, a_2 + t))
\end{equation*}
and
\begin{equation*}
U_{t, \varepsilon}(v) := \B \backslash \phi^{-1}((a_1- \varepsilon, a_1 + t+ \varepsilon) \cup (a_2 - \varepsilon, a_2 + t + \varepsilon)).
\end{equation*}
\end{definition}

We also have the following.

\begin{definition} \label{definition:principal_component}
Let $X$ be a random variable taking values in some Euclidean vector space $V$. We say that $u \in V$ is a \emph{first principal component} of $X$ if it is an eigenvector of its covariance matrix with maximal eigenvalue. 
\end{definition}

\begin{definition} \label{definition:principal_component_u_t}
Given a random variable $X$ taking values in $\A$, $t >0$, and $\varepsilon > 0$ we let
\begin{equation*}
U_t(X) = \cup_{v \in P} U_t(v)
\end{equation*}
and
\begin{equation*}
U_{t, \varepsilon}(X) = \cup_{v \in P} U_{t, \varepsilon}(v)
\end{equation*}
where $P$ is the set of first principal components of $X$. Similarly if $\mu$ is a probability measure which is the law of a random variable $X$ then we define $U_{t}(\lambda) := U_t(X)$ and $U_{t,\varepsilon}(\lambda) := U_{t, \varepsilon}(X)$.
\end{definition}

From this we may deduce the following.

\begin{proposition} \label{proposition:principal_component_u_t}
For all $t > 0$ there is some $\delta > 0$ such that the following is true. Suppose that $V$ is a random variable taking values in $\A$ and that $b \in \B$. Suppose that
\begin{equation*}
b \in U_t(V).
\end{equation*}
Then
\begin{equation*}
\var \rho_b(V) \geq \delta \var V.
\end{equation*}
\end{proposition}
We will prove Propositions \ref{proposition:large_first_derivatives} and \ref{proposition:principal_component_u_t} in Section \ref{section:first_derivs}.

\subsection{Cartan decomposition} \label{section:singular_value_decomp}
The purpose of this subsection is to prove the following proposition and a simple corollary of it.
\begin{proposition} \label{proposition:singular_value_shape}
Given any $t >0$ and $\varepsilon > 0$ there exist some constants $C, \delta > 0$ such that the following is true. Suppose that $n \in \Z_{>0}$, $g_1, g_2, \dots, g_n \in \Gp$, for $i = 1, \dots, n$ we have
\begin{equation*}
\lnb g_i \rnb \geq C
\end{equation*}
and for $i = 1, \dots, n-1$
\begin{equation*}
d(b^{-}(g_i), b^{+}(g_{i+1})) > t.
\end{equation*}
Suppose also that there are $u_1, u_2, \dots, u_{n-1} \in \A$ such that for $i=1, \dots, n-1$ we have
\begin{equation*}
\lnb u_i \rnb < \delta.
\end{equation*}
Then if we let $g' = g_1 \exp(u_1) g_2 \exp(u_2) \dots g_n$ we have
\begin{equation}
\lnb g' \rnb \geq C^{-(n-1)} \lnb g_1 \rnb \cdot  \lnb g_2 \rnb \cdot \dots \cdot \lnb g_n \rnb \label{eq:big_prodsize_bound}
\end{equation}
and
\begin{equation}
d(b^{+}(g'), b^{+}(g_1)) < \varepsilon \label{eq:big_vartheta_bound}
\end{equation}
and
\begin{equation}
d(b^{-}(g'), b^{-}(g_n))  < \varepsilon \label{eq:big_theta_bound}.
\end{equation}
\end{proposition}

\begin{corollary} \label{corollary:singular_value_shape}
Given any $t > 0$ and $\varepsilon > 0$ there exist some constants $C, \delta > 0$ such that the following is true. Suppose that $n \in \Z_{>0}$, $g_1, \dots, g_n \in \Gp$ and $u_1, u_2, \dots, u_{n} \in \A$ satisfy the conditions of Proposition \ref{proposition:singular_value_shape}. Suppose further that $b \in \B$ is such that 
\begin{equation*}
    d(b^{-}(g_n), b) > t.
\end{equation*}
Then if we let $b' = g_1 \exp(u_1) g_2 \exp(u_2) \dots  g_n \exp(u_n) b$ we have
\begin{equation*}
d(b', b^{+}(g_1)) < \varepsilon.
\end{equation*}
\end{corollary}

We will prove Proposition \ref{proposition:singular_value_shape} by induction and then deduce Corollary \ref{corollary:singular_value_shape} from it. First we need the following lemmas.

\begin{lemma} \label{lemma:new_shape_b_simple_singular_value}
Let  $g \in \Gp$, and $b \in \B$. Then

\begin{equation*}
d(b^{+}(g), gb) \lesssim \lnb g \rnb^{-2} d(b^{-}(g), b)^{-1}
\end{equation*}
and for any representative $\hat{b} \in \R^2 \backslash \{0\}$ of $b$ we have
\begin{equation*}
\lnb g\hat{b} \rnb \gtrsim \lnb g \rnb \cdot \lnb \hat{b} \rnb d(b^{-}(g), b).
\end{equation*}
\end{lemma}

\begin{proof}
 The first part follows from \cite[Lemma A.6]{bochi2019anosov}. The second part follows from equation (A.11) in \cite[Lemma A.3]{bochi2019anosov}.
\end{proof}

We also have the following simple corollary.
\begin{corollary} \label{corollary:new_shape_b_simple_singular_value}
For every $\varepsilon > 0$ there exists some $C>0$ such that the following is true. Let $g \in \Gp$ and $b \in \B$. Suppose that
\begin{equation*}
\lnb g \rnb \geq C
\end{equation*}
and
\begin{equation*}
d(b^{-}(g), b) \geq \varepsilon.
\end{equation*}
Then
\begin{equation*}
d(b^{+}(g), gb) \leq \varepsilon
\end{equation*}
and for any representative $\hat{b} \in \R^2 \backslash \{0\}$ of $b$
\begin{equation*}
\lnb g\hat{b} \rnb \geq C^{-1} \lnb g \rnb \cdot \lnb \hat{b} \rnb.
\end{equation*}
\end{corollary}

This corollary is trivial and left as an exercise to the reader.

\begin{lemma} \label{lemma:simple_sl2_prod_size_bound}
Let $g_1, g_2 \in \Gp$. Then
\begin{equation}
\lnb g_1 \rnb \cdot \lnb g_2 \rnb \sin d(b^{-}(g_1)  ,b^{+}(g_2)) \leq \lnb g_1 g_2 \rnb \leq \lnb g_1 \rnb \cdot \lnb g_2 \rnb. \label{eq:simple_prod_size}
\end{equation}
Furthermore, for every $A>1$ and $t > 0$ there exists some $C>0$ with $$C \leq O((A-1)^{-1} t^{-1})$$ such that if $\lnb g_1 \rnb, \lnb g_2 \rnb \geq C$ and $d(b^{-}(g_1)  ,b^{+}(g_2))  \geq t$ then
\begin{equation}
 \lnb g_1 g_2 \rnb \leq A \lnb g_1 \rnb \cdot \lnb g_2 \rnb \sin d(b^{-}(g_1)  ,b^{+}(g_2)). \label{eq:more_precise_prod_size}
\end{equation}
\end{lemma}

\begin{proof}
The right-hand side of \eqref{eq:simple_prod_size} is a well known result about the operator norm. For the left-hand side without loss of generality suppose that
\begin{equation*}
g_1 = \begin{pmatrix}
\lambda_1 & 0 \\ 0 & \lambda_1^{-1}
\end{pmatrix}
\end{equation*}
and
\begin{equation*}
g_2 = \begin{pmatrix}
\cos x & -\sin x \\ \sin x & \cos x
\end{pmatrix}
\begin{pmatrix}
\lambda_2 & 0 \\ 0 & \lambda_2^{-1}
\end{pmatrix} = \begin{pmatrix}
\lambda_2 \cos x & - \lambda_2^{-1} \sin x \\ \lambda_2  \sin x & \lambda_2^{-1} \cos x
\end{pmatrix}.
\end{equation*}
Note that
\begin{align*}
g_1 g_2 \begin{pmatrix}
1 \\ 0
\end{pmatrix} & = \begin{pmatrix}
\lambda_1 \lambda_2 \cos x \\ \lambda_1^{-1} \lambda_2 \sin x
\end{pmatrix}.
\end{align*}
This means $\lnb g_1 g_2 \rnb \geq \lambda_1 \lambda_2 \cos x = \lnb g_1 \rnb \cdot \lnb g_2 \rnb \sin |\phi(b^{-}(g_1) ) - \phi(b^{+}(g_2))|$ which proves \eqref{eq:simple_prod_size}.

For \eqref{eq:more_precise_prod_size} note that
\begin{align*}
g_1 g_2 = \begin{pmatrix}
\lambda_1 \lambda_2 \cos x & -\lambda_1 \lambda_2^{-1} \sin x \\
\lambda_1^{-1} \lambda_2 \sin x & \lambda_1 \lambda_2^{-1} \cos x
\end{pmatrix}.
\end{align*}
This means that 
\begin{align*}
\lnb g_1 g_2 \rnb \leq \lnb g_1 g_2 \rnb_2 \leq \left(1 + 3C^{-2} \left( \cos x \right)^{-1} \right) \lambda_1 \lambda_2 \cos x.
\end{align*}
This gives the required result.
\end{proof}

\begin{lemma} \label{lemma:mid_angles}
Given any $\varepsilon > 0$ and any $t >0$ there is some constant $C>0$ such that the following holds. Let $g_1, g_2 \in \Gp$ be such that $\lnb g_1\rnb, \lnb g_2 \rnb \geq C$ and $d(b^{-}(g_1), b^{+}(g_2)) \geq t$. Then
\begin{equation}
d( b^{+}(g_1), b^{+}(g_1 g_2) )< \varepsilon \label{eq:vartheta_bound}
\end{equation}
and
\begin{equation}
d( b^{-}(g_2), b^{-}(g_1 g_2) ) < \varepsilon. \label{eq:theta_bound}
\end{equation}
Furthermore we have $C \leq O\left( \left( \min \{\varepsilon, t \} \right)^{-1} \right)$.
\end{lemma}

\begin{proof}
This follows from \cite[Lemma A.9]{bochi2019anosov}.
\end{proof}

\begin{lemma} \label{lemma:small_left_move}
Given any $\varepsilon > 0$ there exist $C, \delta > 0$ such that the following is true. Suppose that $g \in \Gp$, $b \in \B$, and $u \in \A$. Suppose further that $\lnb g \rnb \geq C$ and $\lnb u \rnb < \delta$. Then we have
\begin{equation}
C^{-1} \lnb g \rnb \leq \lnb \exp(u) g \rnb \leq C \lnb g \rnb, \label{eq:small_left_modulus_change}
\end{equation}
\begin{equation}
d(b, \exp(u) b ) < \varepsilon, \label{eq:vartheta_close_b}
\end{equation}
and
\begin{equation}
d(b^{+}(g) , b^{+}(\exp(u) g) )  < \varepsilon. \label{eq:vartheta_close_g}
\end{equation}
\end{lemma}

\begin{proof}
First note that \eqref{eq:small_left_modulus_change} and \eqref{eq:vartheta_close_b} both follow from the fact that $\exp(\cdot)$ is smooth and $\B$ is compact. \eqref{eq:vartheta_close_g} follows from \eqref{eq:small_left_modulus_change}, \eqref{eq:vartheta_close_b} and applying Lemma \ref{lemma:new_shape_b_simple_singular_value} with some element of $\B$ which is not close to $b^{-}(g)$ or $b^{-}(\exp(u) g)$ in the role of $b$.
\end{proof}

This is enough to prove Proposition \ref{proposition:singular_value_shape} and Corollary \ref{corollary:singular_value_shape}.

\begin{proof}[Proof of Proposition \ref{proposition:singular_value_shape}]
Without loss of generality assume that $\varepsilon < t$. Let $C_1$ be as in Corollary \ref{corollary:new_shape_b_simple_singular_value} with $\frac{1}{10} \varepsilon$ in the role of $\varepsilon$. Let $C_2$ and $\delta_2$ be $C$ and $\delta$ from Lemma \ref{lemma:small_left_move} with $\frac{1}{10}\varepsilon$ in the role of $\varepsilon$. 

We now take $C = \max \{ C_1C_2, \left( \sin \frac{1}{10} t \right)^{-1} \}$ and $\delta = \delta_2$.

First we will deal with \eqref{eq:big_vartheta_bound}. Choose $b$ such that
\begin{equation*}
d(b, b^{-}(g_n)) > \frac{1}{10} \varepsilon
\end{equation*}
and
\begin{equation*}
d(b, b^{-}(g')) > \frac{1}{10} \varepsilon.
\end{equation*}

Note that by Corollary \ref{corollary:new_shape_b_simple_singular_value} we know that
\begin{equation*}
d(g_n b, b^{+}(g_n)) < \frac{1}{10} \varepsilon.
\end{equation*}
By Lemma \ref{lemma:small_left_move} we know that 
\begin{equation*}
d(\exp(u_{n-1} ) g_n b, g_n b) < \frac{1}{10} \varepsilon
\end{equation*}
and so
\begin{equation*}
d(\exp(u_{n-1} ) g_n b, b^{-}(g_{n-1})) > \frac{1}{10} \varepsilon.
\end{equation*}
Repeating this process we are able to show that
\begin{equation*}
d(g' b, b^{+}(g_1)) < \frac{1}{10} \varepsilon.
\end{equation*}
We also know that
\begin{equation*}
d(g'b, b^{+}(g')) < \frac{1}{10} \varepsilon.
\end{equation*}
Hence
\begin{equation*}
d(b^{+}(g'), b^{+}(g_1)) < \varepsilon.
\end{equation*}
To prove \eqref{eq:big_theta_bound} simply take the transpose of everything.

Now to prove \eqref{eq:big_prodsize_bound}. Let $b$ be chosen as before and let $u \in b$ be a unit vector. Note that by Corollary \ref{corollary:new_shape_b_simple_singular_value}
\begin{equation*}
\lnb g_n u \rnb \geq C_1^{-1} \lnb g_n \rnb \cdot \lnb u \rnb
\end{equation*}
and by Lemma \ref{lemma:small_left_move} we know that
\begin{equation*}
\lnb \exp(u_{n-1}) g_n u \rnb \geq C_1^{-1} C_2^{-1} \lnb g_n \rnb \cdot \lnb u \rnb.
\end{equation*}
Repeating this gives the required result.

\end{proof}
We also prove Corollary \ref{corollary:singular_value_shape}.

\begin{proof}[Proof of Corollary \ref{corollary:singular_value_shape}]
This follows from applying Proposition \ref{proposition:singular_value_shape} to $$g_1 \exp(u_1) g_2 \exp(u_2) \dots g_{n-1} \exp(u_{n-1}) g_n$$ before applying Lemma \ref{lemma:small_left_move} to $\exp(u_n) b$ and then applying Lemma \ref{lemma:new_shape_b_simple_singular_value}.
\end{proof}

\subsection{Proof of Proposition \ref{proposition:intro_decomp_detail}}

In this subsection we will prove Proposition \ref{proposition:intro_decomp_detail}. To do this we will need to find an upper bound on the size of various second derivatives and apply Taylor's theorem. We will use the following version of Taylor's theorem.

\begin{theorem} \label{theorem:taylor}
Let $f : \R ^n \to \Bim$ be twice differentiable and let $R_1, R_2, \dots , R_n > 0$. Let $U = [-R_1, R_1] \times [-R_2, R_2] \times \dots \times [-R_n, R_n] $. For integers $i, j \in [1, n]$ let $K_{i, j} = \sup \limits_{U} \left| \frac{\partial ^2 f}{\partial x_i \partial x_j} \right|$ and let $\mathbf{x} \in U$. Then we have
\begin{equation*}
\left| f(\mathbf{x}) -f(0) - \sum_{i = 1}^n x_i \left. \frac{\partial f}{\partial x_i} \right|_{\mathbf{x} = 0} \right| \leq \frac{1}{2} \sum_{i, j = 1}^n x_i K_{i, j} x_j.
\end{equation*}
\end{theorem}

In order to prove Proposition \ref{proposition:intro_decomp_detail} we need the following proposition.

\begin{proposition} \label{proposition:second_derivative_bounds}
Let $t > 0$. Then there exist some constants $C, \delta >0$ such that the following holds. Suppose that $n \in \Z_{>0}$, $g_1, g_2 \dots, g_n \in \Gp$, $b \in \B$ and let $$u^{(1)}, u^{(2)}, \dots, u^{(n)} \in \A$$ be such that $\lnb u^{(i)} \rnb \leq \delta$. Suppose that for each integer $i \in [1, n]$ we have 
\begin{equation*}
\lnb g_i \rnb \geq C
\end{equation*}
and for integers $i \in [1, n-1]$ we have
\begin{equation*}
d(b^{-}(g_i), b^{+}(g_{i+1})) > t
\end{equation*}
and
\begin{equation*}
d(b^{-}(g_n) , b) > t.
\end{equation*}
Let $x$ be defined by
\begin{equation*}
x = g_1 \exp(u^{(1)}) g_2 \exp(u^{(2)}) \dots g_n \exp(u^{(n)}) b.
\end{equation*}
Then for any $i, j \in \{1, 2, 3\}$ and any integers $k, \ell \in [1, n]$ with $k \leq \ell$ we have
\begin{equation*}
\left| \frac{\partial^2}{\partial u^{(k)}_i \partial u_j^{(\ell)}} \phi( x ) \right| < C^n \lnb g_1 g_2 \dots g_{\ell} \rnb^{-2}.
\end{equation*}
\end{proposition}

We will prove this later in this subsection.

Note that given some $u \in \A$ and some $i \in \{1, 2, 3\}$ by $u_i$ we mean the $i$th component of $u$ with respect to our choice of basis for $\A$ which we will fix throughout this paper. To prove this we need to understand the size of the second derivatives. For this we will need the following lemmas.

\begin{lemma} \label{lemma:derivs} 
Let $t > 0$, let $x \in \Bim$, and let $g \in \Gp$. Suppose that
\begin{equation}
d(b^{-}(g), \phi^{-1}(x)) > t. \label{eq:not_orth_derivs}
\end{equation}
Let $y = \phi ( g \phi^{-1}(x))$. Then
\begin{equation*}
\lnb g \rnb ^{-2} \leq \frac{\partial y}{\partial x} \leq O_t \left( \lnb g \rnb ^{-2} \right)
\end{equation*}
and
\begin{equation*}
\left| \frac{\partial^2 y}{\partial x^2} \right| \leq O_t \left( \lnb g \rnb ^{-2} \right).
\end{equation*}
\end{lemma}

\begin{proof}
Let $g = R_{\phi} A_{\lambda} R_{-\theta}$. First note that
\begin{equation}
y = \tan^{-1} \left( \lambda^{-2} \tan ( x - \theta) \right) + \phi. \label{eq:phi_expression}
\end{equation}
Recall that if $v = \tan^{-1} u$ then $\frac{d v}{d u} = \frac{1}{u^2+1}$. This means that by the chain rule we have
\begin{align*}
\frac{\partial y}{\partial x} & = \left( \frac{1}{\lambda^{-4} \tan^2 ( x - \theta) + 1} \right) \cdot \lambda^{-2} \cdot \left( \frac{1}{\cos^2 (x - \theta)} \right) \\
& = \frac{1}{\lambda^{2} \cos^2 (x-\theta) + \lambda^{-2} \sin^2 (x-\theta)}. 
\end{align*}
Differentiating this again gives
\begin{equation*}
\frac{\partial^2 y}{\partial x^2} = \frac{2  (\lambda^2 + \lambda^{-2}) \cos (x-\theta) \sin (x-\theta)}{\left(\lambda^2 \cos^2 (x-\theta) + \lambda^{-2} \sin^2 (x-\theta) \right)^2}.
\end{equation*}
Noting that \eqref{eq:not_orth_derivs} forces $\cos (x-\theta) \geq \sin t$ gives the required result.
\end{proof}

We also need to bound the second derivatives of various expressions involving $\exp$.

\begin{lemma} \label{lemma:g2der_cis}
There exists some constant $C>0$ such that the following is true. Let $b \in \B$ and define $w$ by
\begin{align*}
w : \A & \to \Bim \\
u & \mapsto \phi\left( \exp(u) b \right).
\end{align*}
Then whenever $\lnb u \rnb \leq 1$ we have
\begin{equation*}
\lnb D_u w \rnb \leq C
\end{equation*}
and
\begin{equation*}
\lnb D^2_u w \rnb \leq C.
\end{equation*}
\end{lemma}

\begin{proof}
This follows immediately from the fact that $\lnb D_u w \rnb$ and $\lnb D^2_u w \rnb$ are continuous in $b$ and $u$ and compactness.
\end{proof}

We will also need the following bound. Unfortunately this lemma doesn't follow easily from a compactness argument and needs to be done explicitly.
\begin{lemma} \label{lemma:g2der_trans}
For every $t >0$ there exist some constants $C, \delta>0$ such that the following holds. Let $g \in \Gp$, let $b \in \B$ and  let $w$ be defined by
\begin{align*}
w : \A \times \A & \to \Bim\\
(x, y) & \mapsto \phi \left( \exp (x ) g \exp(y) b \right).
\end{align*}
Suppose that 
\begin{equation*}
d(b^{-}(g), b) > t
\end{equation*}
and that $\lnb x \rnb, \lnb y \rnb \leq \delta$. Then
\begin{equation*}
\left| \frac{\partial ^2 w(x, y)}{\partial x_i \partial y_j} \right| \leq C \lnb g \rnb^{-2}.
\end{equation*}

\end{lemma}
	
\begin{proof}
Let $\hat{v} = \phi(\exp (y) b)$. First note that by compactness we have
\begin{equation*}
\left| \frac{\partial \hat{v}}{\partial y_j} \right| \leq O(1).
\end{equation*}
Now let $\tilde{v} := \phi( g \exp(y) b)$.  By Lemma \ref{lemma:small_left_move} and Lemma \ref{lemma:derivs} we have
\begin{equation*}
\left| \frac{\partial \tilde{v}}{\partial \hat{v}} \right| \leq O_t \left( C \lnb g \rnb^{-2} \right).
\end{equation*}
Also note that by compactness
\begin{equation*}
\left| \frac{\partial ^2 w}{\partial \tilde{v} \partial x_i} \right| \leq O(1).
\end{equation*}
Hence
\begin{equation*}
\left| \frac{\partial^2 w}{\partial x_i \partial y_j} \right| = \left| \frac{\partial ^2 w}{\partial \tilde{v} \partial x_i} \right| \cdot \left| \frac{\partial \tilde{v}}{\partial \hat{v}} \right| \cdot \left| \frac{\partial \hat{v}}{\partial y_j} \right|  \leq O_t \left( \lnb g \rnb^{-2} \right).
\end{equation*}
We are now done by Lemma \ref{lemma:small_left_move}.
\end{proof}

This is enough to prove Proposition \ref{proposition:second_derivative_bounds}.

\begin{proof}[Proof of Proposition \ref{proposition:second_derivative_bounds}]
First we will deal with the case where $\ell=k$. Let 
\begin{equation*}
a_1 = g_1 \exp(u^{(1)}) g_2 \exp(u^{(2)}) \dots  g_{k-1} \exp(u^{(k-1)}) g_k
\end{equation*}
and
\begin{equation*}
a_2 = g_{k+1} \exp(u^{(k+1)}) g_{k+2} \exp(u^{(k+2)}) \dots  g_{n} \exp(u^{(n)}) b
\end{equation*}
and let $a_3 = \phi(\exp(u^{(k)}) a_2)$. We have
\begin{equation*}
\frac{\partial x}{\partial u^{(k)}_i} = \frac{\partial x}{\partial a_3} \frac{\partial a_3}{\partial u^{(k)}_i}
\end{equation*}
and so
\begin{equation*}
\frac{\partial^2 x}{\partial u^{(k)}_i \partial u^{(k)}_j} = \frac{\partial^2 x}{\partial a_3^2} \frac{\partial a_3}{\partial u^{(k)}_i} \frac{\partial a_3}{\partial u^{(k)}_j} + \frac{\partial x}{\partial a_3} \frac{\partial^2 a_3}{\partial u^{(k)}_i \partial u^{(k)}_j}. 
\end{equation*}

By Proposition \ref{proposition:singular_value_shape} we know that providing $C$ is sufficiently large and $\delta$ is sufficiently small that
\begin{equation*}
d(b^{-}(a_1), a_2) > \frac{1}{2}t
\end{equation*}
By Lemmas \ref{lemma:derivs} and \ref{lemma:g2der_cis} this means that
\begin{equation*}
\left| \frac{\partial^2 x}{\partial u^{(k)}_i \partial u^{(k)}_j} \right| \leq O_t \left( \lnb a_1 \rnb^{-2} \right).
\end{equation*}
In particular by Proposition \ref{proposition:singular_value_shape} there is some constant $C$ depending only on $t$ such that
\begin{equation*}
\left| \frac{\partial^2 x}{\partial u^{(k)}_i \partial u^{(k)}_j} \right| \leq C^n \lnb g_1g_2 \dots g_k \rnb^{-2}
\end{equation*}
as required.

Now we will deal with the case where $\ell>k$. Let
\begin{equation*}
a_1 = g_1 \exp(u^{(1)}) g_2 \exp(u^{(2)}) \dots  g_{k-1} \exp(u^{(k-1)}) g_k
\end{equation*}
and
\begin{equation*}
a_2 = g_{k+1} \exp(u^{(k+1)}) g_{k+2} \exp(u^{(k+2)}) \dots  g_{\ell-1} \exp(u^{(\ell-1)}) g_{\ell}
\end{equation*}
and
\begin{equation*}
a_3 = g_{\ell+1} \exp(u^{(\ell+1)}) g_{\ell+2} \exp(u^{(\ell+2)}) \dots  g_{n} \exp(u^{(n)}) b.
\end{equation*}
Let $a_4 = \phi(\exp(u^{(k)})a_2\exp(u^{(\ell)})a_3)$. Again we have
\begin{equation*}
\frac{\partial^2 x}{\partial u^{(k)}_i \partial u^{(k)}_j} = \frac{\partial^2 x}{\partial a_4^2} \frac{\partial a_4}{\partial u^{(k)}_i} \frac{\partial a_4}{\partial u^{(k)}_j} + \frac{\partial x}{\partial a_4} \frac{\partial^2 a_4}{\partial u^{(k)}_i \partial u^{(k)}_j}. 
\end{equation*}
In a similar way to the case $\ell=k$ but using Lemma \ref{lemma:g2der_trans} instead of Lemma \ref{lemma:g2der_cis} we get
\begin{equation*}
\left| \frac{\partial^2 x}{\partial u^{(k)}_i \partial u^{(\ell)}_j} \right| < C^n \lnb g_1g_2 \dots g_{\ell} \rnb^{-2}
\end{equation*}
as required.
\end{proof}

From this we can now prove Proposition \ref{proposition:intro_decomp_detail}.

\begin{proof}[Proof of Proposition \ref{proposition:intro_decomp_detail}]
By Theorem \ref{theorem:taylor} and Proposition \ref{proposition:second_derivative_bounds} we know that
\begin{align*}
\MoveEqLeft \left| \phi(x) - \phi(g_1g_2 \dots g_{n+1}) - \sum_{i=1}^{n} \zeta_i(u^{(i)}) \right| \\ 
& \leq n^2 C^n \max \left\{ \lnb g_1g_2 \dots g_{i} \rnb^{2}: i \in [1, n]\right\} r^2.
\end{align*}
This is because each of the $n^2$ terms in the error term in the Taylor expansion can be bounded above by an expression of the form $C^ i \lnb g_1g_2 \dots g_{i} \rnb^{2} r^2$. The result follows by replacing $C$ with a slightly larger constant and noting that by Proposition \ref{proposition:singular_value_shape}
\begin{equation*}
\max \left\{ \lnb g_1g_2 \dots g_{i} \rnb^{2}: i  \in [1, n]\right\} = \lnb g_1g_2 \dots g_{n} \rnb^{2}. \qedhere
\end{equation*}
\end{proof}

\subsection{Bounds on first derivatives} \label{section:first_derivs}
The purpose of this subsection is to prove Propositions \ref{proposition:large_first_derivatives} and \ref{proposition:principal_component_u_t}. This bounds the size of various first derivatives. First we need the following lemma.

\begin{lemma} \label{lemma:at_most_two_zeros}
Let $u \in \A \backslash \{ 0 \}$ and given $b \in \B$ define $\varrho_b$ as in Proposition \ref{proposition:large_first_derivatives}. Then there are at most two points $b \in \B$ such that
\begin{equation*}
\varrho_b(u) = 0.
\end{equation*}
\end{lemma}

\begin{proof}
Let $\tilde{\phi}$ be defined by
\begin{align*}
\tilde{\phi} : \R^2 \backslash \{ 0\} & \to \Bim\\
\hat{b} & \mapsto \phi([\hat{b}])
\end{align*}
where $[\hat{b}]$ denotes the equivalent class of $\hat{b}$ in $\B$. 

Given $b \in \B$ let $\hat{b} \in \R^2 \backslash \{ 0\}$ be some representative of $b$. Note that this means
\begin{equation*}
\phi(\exp(v) b) = \tilde{\phi}(\exp (v) \hat{b}).
\end{equation*}
This means that $\varrho_b(v) = 0$ if and only if $D(\exp(u) \hat{b})|_{u=0}(v)$ is in the kernel of $D_{\hat{b}}(\tilde{\phi}(\hat{b}))$. Trivially the kernel of $D_{\hat{b}}(\tilde{\phi}(\hat{b}))$ is just the space spanned by $\hat{b}$. It also follows by the definition of the matrix exponential that for any $v \in \A$ we have
\begin{equation*}
D(\exp(u) \hat{b})|_{u=0}(v) = v \hat{b}.
\end{equation*}
Hence $\varrho_b(v) = 0$ if and only if $\hat{b}$ is an eigenvector of $v$. Clearly for each $v \in \A \backslash \{ 0 \}$ there are at most two $b \in \B$ with this property. The result follows.
\end{proof}

\begin{proof}[Proof of Proposition \ref{proposition:large_first_derivatives}]
Given $a_1, a_2 \in \R$ let $U(a_1, a_2)$ be defined by
\begin{equation*}
U(a_1, a_2) =  \B \backslash \phi^{-1}( ((a_1, a_1 + t) \cup (a_2, a_2 + t))).
\end{equation*}
In other words $U(a_1, a_2)$ is all of $\B$ except for two arcs of length $t$ starting at $a_1$ and $a_2$ respectively. Given some $v \in \A$ let $f(v)$ be given by
\begin{equation*}
f(v) := \max_{a_1, a_2 \in \R} \min_{b \in U(a_1, a_2)} |\varrho_b(v)|.
\end{equation*}
Both the $\min$ and the $\max$ are achieved due to a trivial compactness argument. By Lemma \ref{lemma:at_most_two_zeros} we know that $f(v) > 0$ whenever $\lnb v \rnb = 1$. Note that $\left \{ \varrho_b(\cdot) : b \in \B\right \}$ is a bounded set of linear maps and so is uniformly equicontinuous. This means that $f$ is continuous. Since the set of all $v \in \A$ with $\lnb v \rnb = 1$ is compact this means that there is some $\delta > 0$ such that $f(v) \geq \delta$. Finally note that trivially we can choose the $a_1$ and $a_2$ using this construction in such a way that they are measurable as functions of $v$.
\end{proof}

We will now prove Proposition \ref{proposition:principal_component_u_t}.

\begin{proof}[Proof of Proposition \ref{proposition:principal_component_u_t}]
By elementary linear algebra we can write $X$ as
\begin{equation*}
X = X_1v_1 + X_2v_2 + X_3v_3
\end{equation*}
where $X_1$, $X_2$ and $X_3$ are uncorrelated random variables taking values in $\R$ and $v_1$, $v_2$, and $v_3$ are the eigenvectors of the covariance matrix of $X$ with corresponding eigenvalues $\var X_1$, $\var X_2$, and $\var X_3$. Furthermore we may assume that $\var X_1 \geq \var X_2 \geq \var X_3$ and so in particular $\var X_1 \geq \frac{1}{3} \Tr \var X$. Without loss of generality we may assume that $X_1$, $X_2$, $X_3$, and $X$ have mean $0$. We also note that since $v_1$ is a principal component of $X$ by Proposition \ref{proposition:large_first_derivatives} we have $|\rho_b(v_1)| \geq \delta$.

We then compute
\begin{align*}
\var \rho_b(X) &= \mathbb{E} \left[ |\rho_b(X)|^2 \right] \\
& = \mathbb{E} \left[ X_1^2 |\rho_b(v_1)|^2 + X_2^2 |\rho_b(v_2)|^2 + X_3^2 |\rho_b(v_3)|^2\right]\\
& \geq \mathbb{E} \left[ X_1^2 |\rho_b(v_1)|^2\right]\\ 
& \geq \frac{1}{3} \delta^2 \Tr \var X.
\end{align*}
This gives the required result.
\end{proof}

\section{Disintegration Argument} \label{section:disintegration_argument}

The purpose of this section is to prove Theorem \ref{theorem:prob_exist}. We first discuss some basic properties of entropy and variance for random variables taking values in $\Gp$. After these preparations, which occupy most of the section, the proof of Theorem \ref{theorem:prob_exist} will be short.

Before we begin we outline the main steps of the proof of Theorem \ref{theorem:prob_exist}.

The first step is the following simple lemma.

\begin{lemma} \label{lemma:rewrite_conditional_entropy}
Let $g$, $s_1$ and $s_2$ be independent random variables taking values in $\Gp$. Suppose that $s_1$ and $s_2$ are absolutely continuous with finite entropy and that $gs_1$ and $gs_2$ have finite entropy. Define $k$ by
\begin{equation*}
k := H(g s_1) - H(s_1) - H(g s_2) + H(s_2).
\end{equation*}
Then
\begin{equation*}
\mathbb{E}[H((gs_1|gs_2))] \geq k + H(s_1).
\end{equation*}
\end{lemma}
Here $(g s_1 | g s_2)$ denotes the regular conditional distribution which is defined in Section \ref{section:r_c_p}. We prove this lemma in Section \ref{section:entropy}.

We will apply this lemma when $s_1$ and $s_2$ are smoothing random variables, and $s_2$ corresponds to a larger scale than $s_1$. The quantity $k$ can be thought of as the difference between the information of $g$ discretized at the scales corresponding to $s_1$ and $s_2$.

It is well known that amongst all random vectors whose covariance matrix has a given trace, the spherical normal distribution has the largest (differential) entropy. This allows us to estimate the variance of a random vector in terms of its entropy from below. Once the definitions are in place, we can translate this to random elements of $\Gp$.

\begin{lemma} \label{lemma:entvartbound}
Let $\varepsilon > 0$ and suppose that $g$ is an absolutely continuous random variable taking values in $\Gp$ such that $g_0^{-1} g$ takes values in the ball of radius $\varepsilon$ and centre $\id$ for some $g_0 \in \Gp$. Then providing $\varepsilon$ is sufficiently small we have
\begin{equation*}
H(g) \leq \frac{3}{2} \log \frac{2 \pi e}{3} \Tr \vart_{g_0} [g] + O(\varepsilon).
\end{equation*}
\end{lemma}
We will prove this in Section \ref{section:entropy}. Combining the above two lemmas, we can get a lower bound on $\Tr \vart_{g s_2}[ g s_1 | g s_2]$. Here $\vart_{\cdot}[\cdot | \cdot]$ denotes the conditional variance of a random variable taking values in $\Gp$ which we will define in Definition \ref{definition:vart_random}. The last part of the proof of Theorem \ref{theorem:prob_exist} is the following.
\begin{lemma} \label{lemma:vart_split_add}
Let $\varepsilon > 0$ be sufficiently small and let $a$ and $b$ be random variables taking values in $\Gp$ and let $\mathcal{A}$ be a $\sigma$-algebra. Suppose that $b$ is independent from $a$ and $\mathcal{A}$. Let $g_0$ be an $\mathcal{A}$-measurable random variable taking values in $\Gp$. Suppose that $g_0^{-1} a $ and $b$ are almost surely contained in a ball of radius $\varepsilon$ around $\id$.
Then
\begin{equation*}
\Tr \vart_{g_0} [ab|\mathcal{A}] = \Tr \vart_{g_0}[a | \mathcal{A}] + \Tr \vart_{\id}[b] + O(\varepsilon^3).
\end{equation*}
\end{lemma}
We prove this in Section \ref{section:variance}.

\subsection{Variance on $\Gp$} \label{section:variance}

Recall from the introduction that given some random variable $g$ taking values in $\Gp$ and some fixed $g_0 \in \Gp$ such that $g_0^{-1} g$ is always in the domain of $\log$ we define $\vart_{g_0}[g]$ to be the covariance matrix of $\log [g_0^{-1} g]$.

We need the following lemma.

\begin{lemma} \label{lemma:vartadd}
Let $\varepsilon >0 $ be sufficiently small and let $g$ and $h$ be independent random variables taking values in $\Gp$. Suppose that the image of $g$ is contained in a ball of radius $\varepsilon$ around $\id$ and the image of $h$ is contained in a ball of radius $\varepsilon$ around some $h_0 \in \Gp$. Then
\begin{equation*}
\Tr \vart_{h_0}[hg] = \Tr \vart_{h_0}[h] + \Tr \vart_{\id}[g]  + O(\varepsilon^3).
\end{equation*}
\end{lemma}

\begin{proof}
Let $X = \log(h_0 ^ {-1} h)$ and let $Y = \log(g)$. Then by Taylor's theorem
\begin{equation*}
\log(\exp(X)\exp(Y)) = X + Y + E
\end{equation*}
where $E$ is some random variable with $|E| \leq O(\varepsilon^2)$ almost surely. Note that we also have $|X|, |Y| \leq O(\varepsilon)$. Therefore
\begin{align*}
\Tr \vart_{h_0}[hg] &= \mathbb{E}[|X+Y+E|^2] - |\mathbb{E}[X+Y+E]|^2\\
&= \mathbb{E}[|X+Y|^2] - |\mathbb{E}[X+Y]|^2 + 2 \mathbb{E}[(X+Y) \cdot E] + \mathbb{E}[|E|^2] \\
& \,\,\,\,\, - 2 \mathbb{E}[X+Y] \cdot \mathbb{E}[E] - |\mathbb{E}[E]|^2\\
&= \var[X+Y] + O(\varepsilon^3)
\end{align*}
as required.
\end{proof}

We also need to describe the variance of a regular conditional distribution.

\begin{definition} \label{definition:vart_random}
Given some random variable $g$ taking values in $\Gp$, some $\sigma$-algebra $\mathcal{A}$ and some $\mathcal{A}$-measurable random variable $g_0$ taking values in $\Gp$ we let $\Tr \vart_{g_0} [g | \mathcal{A}]$ to be the $\mathcal{A}$-measurable random variable given by
\begin{equation*}
\Tr \vart_{g_0} [g | \mathcal{A}](\omega) = \Tr \vart_{g_0(\omega)} [(g | \mathcal{A})(\omega )].
\end{equation*}
Similarly given a random variable $h$ and some $\sigma(h)$-measurable random variable $g_0$ taking values in $\Gp$ we let $\Tr \vart_{g_0} [g | h] = \Tr \vart_{g_0} [g | \sigma(h)]$.
\end{definition}

Lemma \ref{lemma:vart_split_add} now follows easily from Lemma \ref{lemma:vartadd}.

\begin{proof}[Proof of Lemma \ref{lemma:vart_split_add}]
This follows immediately from Lemma \ref{lemma:vartadd} and Lemma \ref{lemma:r_c_d_conditional_independence}.
\end{proof}

\subsection{Entropy} \label{section:entropy}

First we need the following well known result.

\begin{lemma} \label{lemma:rdentvar}
If $X$ is an absolutely continuous random variable taking values in $\R^d$ and $\Tr \var X = r^2$ then
\begin{equation*}
H(X) \leq \frac{d}{2} \log \left( \frac{2 \pi e}{d} r^2 \right)
\end{equation*}	
with equality if and only if $X$ is a spherical normal distribution.
\end{lemma}
\begin{proof}
This is well known and follows trivially from \cite[Example 12.2.8]{COVER_THOMAS_2006}.
\end{proof}

We now wish to prove a similar result for random variables taking values in $\Gp$. First we need the following.

\begin{lemma} \label{lemma:klbound}
Let $\lambda_1$ be a probability measure on some measurable space $E$ and let $ \lambda_2$ and $\lambda_3$ be measures on $E$ and let $U \subset E$. Suppose that the support of $\lambda_1$ is contained in $U$. Then,
\begin{equation*}
\left| \mathcal{KL}(\lambda_1, \lambda_2) - \mathcal{KL}(\lambda_1, \lambda_3) \right| \leq \sup \limits_{x \in U} \left| \log \frac{d \lambda_2}{d \lambda_3} \right|.
\end{equation*}
\end{lemma}
	
\begin{proof}
We have
\begin{align*}
\left| \mathcal{KL}(\lambda_1, \lambda_2) - \mathcal{KL}(\lambda_1, \lambda_3) \right| & = \left| \int_U \log \frac{d \lambda_1}{d \lambda_2} \, d \lambda_1 - \int_U \log \frac{d \lambda_1}{d \lambda_3} \, d \lambda_1 \right|\\
& \leq\int_U  \left| \log \frac{d \lambda_1}{d \lambda_2} - \log \frac{d \lambda_1}{d \lambda_3}  \right| \, d \lambda_1 \\
& = \int_{U} \left| \log \frac{d \lambda_2}{d \lambda_3} \right| \, d \lambda_1\\
& \leq \sup \limits_{x \in U} \left| \log \frac{d \lambda_2}{d \lambda_3} \right|.
\end{align*}
\end{proof}
	
We can now prove Lemma \ref{lemma:entvartbound}.

\begin{proof}[Proof of Lemma \ref{lemma:entvartbound}]
This follows easily from Lemma \ref{lemma:rdentvar} and Lemma \ref{lemma:klbound}. 

Let $U$ be the ball in $\Gp$ of centre $\id$ and radius $\epsilon$. Due to properties of the Haar measure we have $H(g) = H(g_0^{-1} g)$ and by definition $\Tr \vart_{g_0}[g] = \Tr \vart_{\id}[g_0^{-1} g]$. This means that it is sufficient to show that
\begin{equation*}
H(g_0^{-1} g) \leq \frac{3}{2} \log \frac{2 \pi e}{3} \Tr \vart_{\id} [g_0^{-1} g] + O(\varepsilon).
\end{equation*}

Recall that $\frac{d \tilde{m}}{d m \circ \log}$ is smooth and equal to $1$ at $\id$. This means that providing $\varepsilon < 1$ on $U$ we have
\begin{equation*}
\frac{d \tilde{m}}{d m \circ \log} = 1 + O(\varepsilon).
\end{equation*}
In particular providing $\varepsilon$ is sufficiently small we have
\begin{equation*}
\sup \limits_{U} \left| \log \frac{d \tilde{m}}{d m \circ \log}  \right| < O(\varepsilon).
\end{equation*}
Clearly 
\begin{equation*}
\mathcal{KL}(g_0^{-1}g, m \circ \log) = \mathcal{KL}(\log(g_0^{-1}g), m).
\end{equation*}
We have by definition that $H(g_0^{-1} g) = \mathcal{KL}(g_0^{-1}g, \tilde{m})$ and by Lemma  \ref{lemma:klbound} we have $\left| \mathcal{KL}(g_0^{-1}g, m \circ \log) - \mathcal{KL}(g_0^{-1}g, \tilde{m}) \right| \leq O(\varepsilon)$. By Lemma \ref{lemma:rdentvar} we know that
\begin{equation*}
\mathcal{KL}(\log(g_0^{-1}g), m) \leq \frac{3}{2} \log \frac{2 \pi e}{3} \Tr \vart_{\id} [g_0^{-1} g].
\end{equation*}
Therefore
\begin{equation*}
H(g_0^{-1} g) \leq \frac{3}{2} \log \frac{2 \pi e}{3} \Tr \vart [g_0^{-1} g] + O(\varepsilon)
\end{equation*}
as required.
\end{proof}

We now have all the tools required to prove Lemma \ref{lemma:rewrite_conditional_entropy}.

\begin{proof}[Proof of Lemma \ref{lemma:rewrite_conditional_entropy}]
First note that we have
\begin{equation*}
H(gs_2 | gs_1) \geq H(g s_2 | g, s_1) = H(s_2)
\end{equation*}
and so
\begin{equation*}
H(gs_2 , gs_1) \geq H(gs_1) + H(s_2).
\end{equation*}
This means that
\begin{align*}
H(gs_1 | g s_2) & = H(gs_2 , gs_1) - H(g s_2)\\
& \geq H(gs_1) - H(gs_2) + H(s_2)\\
& = k + H(s_1).
\end{align*}

Recalling that by Lemma \ref{lemma:relative_ent_expectation} $H(gs_1 | g s_2) = \mathbb{E}[H ((gs_1 | g s_2))]$ we get
\begin{equation*}
\mathbb{E}[H ((gs_1 | g s_2))] \geq k + H(s_1)
\end{equation*}
as required.
\end{proof}

\subsection{Proof of Theorem \ref{theorem:prob_exist}}

We now have everything needed to prove Theorem \ref{theorem:prob_exist}.

\begin{proof}[Proof of Theorem \ref{theorem:prob_exist}]
Note that by Lemma \ref{lemma:rewrite_conditional_entropy} we have
\begin{equation*}
\mathbb{E}[H ((gs_1 | g s_2))] \geq k + H(s_1)
\end{equation*}
and so by Lemma \ref{lemma:entvartbound} we have
\begin{equation}
\mathbb{E}\left[\frac{3}{2} \log\frac{2}{3} \pi e \Tr \vart_{g s_2} [g s_1 | g s_2] \right] + O(\varepsilon) \geq k + H(s_1). \label{eq:vart_thing}
\end{equation}

Note that $(g s_2)^{-1} g = s_2^{-1}$ which is contained in a ball of radius $\varepsilon$ centred on the identity. Therefore by Lemma \ref{lemma:vart_split_add} we have 
\begin{equation*}
\Tr \vart_{g s_2} [g s_1 | g s_2] \leq \Tr \vart_{g s_2} [g | g s_2] + \Tr \vart_{\id} [s_1] +O(\varepsilon^3).
\end{equation*}
Putting this into \eqref{eq:vart_thing} gives
\begin{equation*}
\mathbb{E}\left[\frac{3}{2} \log\frac{2}{3} \pi e (\Tr \vart_{g s_2} [g | g s_2] + \Tr \vart_{\id} [s_1] +O(\varepsilon^3)) \right]+ O(\varepsilon) \geq k + H(s_1) 
\end{equation*}
which becomes
\begin{equation*}
\mathbb{E}\left[\log\ (1 + \frac{\Tr \vart_{g s_2} [g | g s_2]}{\Tr \vart_{\id} [s_1]} +O_A(\varepsilon)) \right] + O(\varepsilon) \geq \frac{2}{3}(k + H(s_1) - \frac{3}{2} \log \frac{2}{3} \pi e \Tr \vart_{\id}[s_1]).
\end{equation*}

Noting that for $x \geq 0$ we have $x \geq \log(1+x)$ we get
\begin{equation*}
\mathbb{E}[\Tr \vart_{g s_2} [g | g s_2] ] \geq \frac{2}{3}(k -c - O_A(\varepsilon))\Tr \vart_{\id}[s_1] 
\end{equation*}
as required.
\end{proof}

\section{Entropy Gap} \label{section:entropy_gap}
The purpose of this section is to prove Proposition \ref{proposition:lots_of_v}. This shows that for a stopped random walk $\gamma_1 \gamma_2 \dots \gamma_{\tau}$ there are many choices of $s$ such that $v(\gamma_1 \gamma_2 \dots \gamma_{\tau};s)$ is large. 

Recall that $v(g;s)$ is defined to be the supremum of all $v \geq 0$ such that we can find some $\sigma$-algebra $\mathcal{A}$ and some $\mathcal{A}$- measurable random variable $a$ taking values in $\Gp$ such that $|\log(a^{-1} g)| \leq s$ and
\begin{equation*}
\mathbb{E} \left[\Tr \vart_{a} \left[ g | \mathcal{A} \right] \right] \geq v s^2.
\end{equation*}

We apply Theorem \ref{theorem:prob_exist} with a careful choice of $s_1$ and $s_2$. We will take these to be compactly supported approximations to the image of spherical normal random variables on $\A$ under $\exp$. More precisely we have the following.

\begin{definition} \label{definition:eta_r_a}
Given $r > 0 $ and $a \geq 1$ let $\eta_{r, a}$ be the random variable on $\R^3$ with density function $f:\R^3 \to \R$ given by
\begin{equation*}
f(x) = 
\begin{cases}
C e^{-\frac{\lnb x \rnb^2}{2r^2}} & \text{ if } \lnb x \rnb \leq ar \\
0 & \text{ otherwise }
\end{cases} 
\end{equation*}
where $C$ is a normalizing constant chosen to ensure that $f$ integrates to $1$.
\end{definition}

We can then define the following family of smoothing functions.

\begin{definition}
Given $r > 0 $ and $a \geq 1$ let $\s_{r, a}$ be the random variable on $\Gp$ given by
\begin{equation*}
\s_{r, a} = \exp(\eta_{r,a}).
\end{equation*}
\end{definition}

In this definition we use our identification of $\A$ with $\R^3$.

After doing some computations on the entropy and variance of the $\eta_{r, a}$ we can prove the following proposition by putting these estimates into Theorem \ref{theorem:prob_exist}.

\begin{proposition} \label{proposition:ent_gap_to_v}
There is some constant $c> 0$ such that the following holds. Let $g$ be a random variable taking values in $\Gp$, let $a \geq 1$ and let $r>0$. Define $k$ by
\begin{equation*}
k = H(g \s_{r, a}) - H(\s_{r, a}) - H(g \s_{2r, a}) + H(\s_{2r, a}).
\end{equation*}
Then
\begin{equation*}
v(g;2ar) \geq c a^{-2} (k - O(e^{-\frac{a^2}{4}}) - O_a(r)).
\end{equation*}
\end{proposition}

This will be proven in Section \ref{section:smoothing_rvs}.

To make this useful we will need a way to bound $k$ from Proposition \ref{proposition:ent_gap_to_v} from below for appropriately chosen scales. We will do this by bounding
\begin{equation*}
H(g \s_{r, a}) - H(\s_{r, a}) - H(g \s_{2^nr, a}) + H(\s_{2^nr, a})
\end{equation*}
for some carefully chosen $n$ and $r$ and then noting the identity
\begin{align*}
\MoveEqLeft H(g \s_{r, a}) - H(\s_{r, a}) - H(g \s_{2^n r, a}) + H(\s_{2^n r, a})\\ & = \sum_{i=1}^n H(g \s_{2^{i-1}r, a}) - H(\s_{2^{i-1}r, a}) - H(g \s_{2^i r, a}) + H(\s_{2^i r, a}).
\end{align*}
We use this to find scales where we can apply Proposition \ref{proposition:ent_gap_to_v}. Specifically we will prove the following.

\begin{proposition} \label{proposition:stopping_time_gap}
Let $\mu$ be a finitely supported Zariski-dense measure on $\Gp$. Suppose that $M_{\mu} < \infty$ and $h_{RW}/ \chi$ is sufficiently large. Let $\gamma_1, \gamma_2, \dots$ be i.i.d. samples from $\mu$. Let $P>0$, let $w \in \B$ and let $\tau = \tau_{P, w}$ be as in Definition \ref{definition:tau_t_v}. Suppose that $0< r_1 < r_2 < 1$. Suppose that $r_1 < M ^{- \log P / \chi}$. Let $a \geq 1$. Then
\begin{equation}
H (\gamma_1 \gamma_2 \dots \gamma_{\tau} \s_{r_1, a}) \geq \frac{h_{RW}}{\chi} \log P + H ( \s_{a, r_1})  + o_{M, \mu, a, w}(\log P) \label{eq:r1_ent_lower_bound}
\end{equation}
and
\begin{equation}
H (\gamma_1 \gamma_2 \dots \gamma_{\tau} \s_{r_2, a}) \leq 2 \log P + o_{M, \mu, a, w}(\log P). \label{eq:r2_ent_upper_bound}
\end{equation}
In particular
\begin{align}
 \MoveEqLeft H (\gamma_1 \gamma_2 \dots \gamma_{\tau} \s_{r_1, a}) - H ( \s_{r_1, a})  - H (\gamma_1 \gamma_2 \dots \gamma_{\tau} \s_{r_2, a}) + H(\s_{r_2, a}) \nonumber \\ &\geq \left(\frac{h_{RW}}{ \chi} -2 \right) \log P   + 3 \log r_2 + o_{M, \mu, a, w}(\log P). \label{eq:ent_gap}
\end{align}
\end{proposition}
This is proven in Section \ref{section:subsection_entropy_gap}.
This proposition is unsurprising. To motivate \eqref{eq:r1_ent_lower_bound} note that it is well known that with high probability $\tau \approx \log P / \chi$. We also know by the definition of $h_{RW}$ that
\begin{equation*}
H(\gamma_1 \gamma_2 \dots \gamma_{\floor{\log P / \chi}}) \geq h_{RW} \floor{\log P / \chi}.
\end{equation*}
Providing $P$ is sufficiently large $\s_{r_1, a}$ is contained in a ball with centre $\id$ and radius $O_{M, \mu, a}(M^{- \log P / \chi})$. In particular providing $P$ is sufficiently large this radius is less than half the minimum distance between points in the image of $\gamma_1 \gamma_2 \dots \gamma_{\floor{\log P / \chi}}$ and so $H(\gamma_1 \gamma_2 \dots \gamma_{\floor{\log P / \chi}} \s_{r_1, a}) = H(\gamma_1 \gamma_2 \dots \gamma_{\floor{\log P / \chi}}) + H(\s_{r_1, a}) $. It turns out we can prove something similar when $\floor{\log P / \chi}$ is replaced by $\tau$.

The bound \eqref{eq:r2_ent_upper_bound} follows easily from the fact that the Haar measure of most of the image of $\gamma_1 \gamma_2 \dots \gamma_{\tau} \s_{r_2, a}$ is at most $O_{\mu, a}(P^2)$.

Finally \eqref{eq:ent_gap} follows from combining \eqref{eq:r1_ent_lower_bound} and \eqref{eq:r2_ent_upper_bound} and noting that $H(\s_{r_2, a}) = 3 \log r_2 + O(1)$.

We then combine Propositions \ref{proposition:ent_gap_to_v} and \ref{proposition:stopping_time_gap} to get the following.

\begin{proposition} \label{proposition:int_of_v}
There is some absolute constant $c > 0$ such that the following is true. Suppose that $\mu$ finitely supported Zariski-dense probability measure. Suppose that $M_{\mu} < \infty$ and that $h_{RW}/\chi$ is sufficiently large. Let $M > M_{\mu}$. Suppose that $M$ is chosen large enough that $h_{RW} \leq \log M$. Let $\gamma_1, \gamma_2, \dots$ be i.i.d.\ samples from $\mu$ and let $b \in \B$. Then for all sufficiently large (depending on $M$, $\mu$ and $w$) $P$ we have
\begin{equation*}
\int_{P^{ - \frac{\log M}{\log \chi}}}^{P^{ - \frac{h_{RW}}{10 \log \chi}}} \frac{1}{u} v(\gamma_1 \gamma_2 \dots \gamma_{\tau_{P, b}}; u) \, du \geq c \left(\frac{h_{RW}}{ \chi}\right) \left( \max \left \{ 1,  \log \frac{\log M}{\chi} \right \} \right)^{-1} \log P .
\end{equation*}

\end{proposition}

We prove this in Section \ref{section:variance_of_disint_stopped_rw}. Proposition \ref{proposition:lots_of_v} follows easily from this.

\subsection{Smoothing random variables} \label{section:smoothing_rvs}

In this subsection we give bounds on the variance and entropy of the $\s_{r, a}$ and use this to prove Proposition \ref{proposition:ent_gap_to_v}.

Recall the definition of $\eta_{r, a}$ from Definition \ref{definition:eta_r_a}. First we have the following.

\begin{lemma} \label{lemma:var_bound_trunc_normal}
Let $r > 0 $ and $a \geq 1$. Then
\begin{equation*}
\Theta( r^2) \leq \Tr \var \eta_{r, a} \leq 3 r^2.
\end{equation*}
\end{lemma}

The proof of this lemma is trivial and is left to the reader.

\begin{lemma} \label{lemma:ent_bound_trunc_normal}
There is some constant $c>0$ such that the following is true. Let $r > 0 $ and $a \geq 1$. Then
\begin{equation*}
H(\eta_{r, a}) = \frac{3}{2} \log 2 \pi e r^{2} + O(e^{ - \frac{a^2}{4}}).
\end{equation*} 
\end{lemma}

The proof of Lemma \ref{lemma:ent_bound_trunc_normal} is a simple computation which we will do later. We deduce the following about $\s_{r, a}$.

\begin{lemma} \label{lemma:var_bound_trunc_normal_g}
Let $r > 0 $ and $a \geq 1$. Suppose that $ar$ is sufficiently small. Then
\begin{equation*}
\Theta( r^2) \leq \Tr \vart_{\id} \s_{r, a} \leq  3 r^2.
\end{equation*}
\end{lemma}
\begin{proof}
This follows immediately from substituting Lemma \ref{lemma:var_bound_trunc_normal} into the definition of $\vart_{\id}$.
\end{proof}

\begin{lemma} \label{lemma:ent_bound_trunc_normal_g}
Let $r > 0 $ and $a \geq 1$. Then
\begin{equation*}
H(\s_{r, a}) = \frac{3}{2} \log 2 \pi e r^{2} + O(e^{-\frac{a^2}{4}}) + O_a(r).
\end{equation*} 
\end{lemma}

\begin{proof}
This follows immediately from Lemma \ref{lemma:ent_bound_trunc_normal} and Lemma \ref{lemma:klbound}.
\end{proof}
We also have the following fact.

\begin{lemma} \label{lemma:psi_m_small}
Let $r > 0 $ and $a \geq 1$. Suppose that $ar$ is sufficiently small. Then
\begin{equation*}
\lnb \log(\s_{r, a}) \rnb \leq ar
\end{equation*}
almost surely.
\end{lemma}

\begin{proof}
This is trivial from the definition of $\s_{r, a}$.
\end{proof}

We now have enough to prove Proposition \ref{proposition:ent_gap_to_v}.

\begin{proof}[Proof of Proposition \ref{proposition:ent_gap_to_v}]
We apply Theorem \ref{theorem:prob_exist} with $s_1 = \s_{r, a}$ and $s_2 = \s_{2r, a}$. We also take $\varepsilon = 3ar$.

By Lemma \ref{lemma:var_bound_trunc_normal_g} we know that
\begin{equation*}
\Tr \vart_{\id} [s_1] \geq \Theta(r^2) \geq \Theta_a(\varepsilon^2)
\end{equation*}
and by Lemmas \ref{lemma:ent_bound_trunc_normal_g} and \ref{lemma:var_bound_trunc_normal_g} we know that
\begin{equation*}
c = \frac{3}{2} \log \frac{2}{3} \pi e \Tr \vart[s_1] - H(s_1) \leq O(e^{- \frac{a^2}{4}}).
\end{equation*}
This means that
\begin{equation*}
\mathbb{E}[\Tr \vart_{g s_2}[g | g s_2]] \geq \frac{2}{3} (k - O(e^{- \frac{a^2}{4}}) - O_a(r)) ( c r^2)
\end{equation*}
for some absolute constant $c >0$.

We know that
\begin{equation*}
\lnb \log \left( (g s_2)^{-1} g \right) \rnb = \lnb \log s_2 \rnb \leq 2ar
\end{equation*}
and so by the definition of $v(\cdot; \cdot)$ we have
\begin{align*}
v(g; 2ar) &\geq (2ar)^{-2} \mathbb{E}[\Tr \vart_{g s_2}[g | g s_2]]\\
& \geq c' a^{-2} (k - O(e^{- \frac{a^2}{4}}) - O_a(r))
\end{align*}
for some absolute constant $c'>0$.
\end{proof}

To finish the subsection we just need to prove Lemma \ref{lemma:ent_bound_trunc_normal}.

\begin{proof}[Proof of Lemma \ref{lemma:ent_bound_trunc_normal}]
Recall that $\eta_{a,r}$ has density function $f:\R^3 \to \R$ given by
\begin{equation*}
f(x) = 
\begin{cases}
C e^{-\frac{\lnb x \rnb^2}{2r^2}} & \text{ if } \lnb x \rnb \leq ar \\
0 & \text{ otherwise }
\end{cases} 
\end{equation*}
where $C$ is a normalizing constant chosen to ensure that $f$ integrates to $1$. 

First we will deal with the case where $r=1$. Note that
\begin{equation*}
\int_{x \in \R^3 : \lnb x \rnb \leq a} e^{-\frac{x^2}{2}} \, dx \leq \int_{ \R^3} e^{-\frac{x^2}{2}} \, dx = \left( 2 \pi \right)^{\frac{3}{2}}
\end{equation*}
and
\begin{align*}
\int_{x \in \R^3 : \lnb x \rnb \geq a} e^{-\frac{x^2}{2}} \, dx &=  \int_{u=a}^{\infty} 4 \pi u^2 e^{ - \frac{ u^2}{2}} \, du \\
&\leq O\left( \int_{u=a}^{\infty} 4 \pi a^2 e^{ - \frac{a u}{3}} \, du \right)\\
& \leq O \left( e^{ - \frac{a^2}{4}} \right).
\end{align*}
This means
\begin{equation*}
\int_{x \in \R^3 : \lnb x \rnb \leq a} e^{-\frac{x^2}{2}} \, dx =\left( 2 \pi \right)^{\frac{3}{2}} - \int_{x \in \R^3 : \lnb x \rnb \geq a} e^{-\frac{x^2}{2}} \, dx \geq \left( 2 \pi \right)^{\frac{3}{2}} - O \left( e^{ - \frac{a^2}{4}} \right).
\end{equation*}
Therefore
\begin{equation*}
 C = \left( 2 \pi \right)^{-3/2} + O \left( e^{ - \frac{a^2}{4}} \right).
\end{equation*}

Note that
\begin{align*}
H(\eta_{1, a}) & = \int_{\lnb x \rnb \leq a} - C e^{- \lnb x \rnb^2/2} \log \left( C e^{- \lnb x \rnb^2/2} \right) \, dx\\
& = \int_{\lnb x \rnb \leq a} C \left(\frac{\lnb x \rnb^2}{2} - \log C \right) e^{- \lnb x \rnb^2/2} \, dx.
\end{align*}
We have
\begin{align*}
\MoveEqLeft \int_{x \in \R^3} C \left(\frac{\lnb x \rnb^2}{2} - \log C \right) e^{- \lnb x \rnb^2/2} \, dx\\ & = \left( 2 \pi \right)^{3/2} C \left( \frac{3}{2} - \log C \right) \\
&= \left(1 + O \left( e^{ - \frac{a^2}{4}} \right) \right) \left( \frac{3}{2} \log e + \frac{3}{2} \log 2 \pi + O \left( e^{ - \frac{a^2}{4}} \right) \right)\\
& = \frac{3}{2} \log 2 \pi e +  O \left( e^{ - \frac{a^2}{4}} \right).
\end{align*}
We also have
\begin{align*}
\MoveEqLeft \int_{x \in \R^3: \lnb x \rnb \geq a} C \left(\frac{\lnb x \rnb^2}{2} - \log C \right) e^{- \lnb x \rnb^2/2} \, dx\\
& = \int_{u=a}^{\infty} 4 \pi u^2 C \left( \frac{u^2}{2} - \log C \right) e^{-u^2/2} \, du\\
& \leq O \left( \int_{u=a}^{\infty} a^4 e^{-au/3} \, du \right)\\
\
& \leq  O \left( e^{-a^2/4}\right).
\end{align*}
This gives
\begin{equation*}
H(\eta_{1, a}) \geq \frac{3}{2} \log 2 \pi e + O( e^{- a^2 /4}).
\end{equation*}
From this we may immediately deduce that
\begin{equation*}
H(\eta_{r, a}) \geq \frac{3}{2} \log 2 \pi e r^2 + O( e^{- a^2 /4})
\end{equation*}
as required. The fact that $H(\eta_{r, a}) \leq \frac{3}{2} \log 2 \pi e r^2$ follows immediately from Lemmas \ref{lemma:rdentvar} and \ref{lemma:var_bound_trunc_normal}.
\end{proof}

\subsection{Entropy gap} \label{section:subsection_entropy_gap}

We now prove Proposition \ref{proposition:stopping_time_gap}. This Proposition bounds the difference in entropy of $\gamma_1 \gamma_2 \dots \gamma_{\tau}$ smoothed at two different scales. Before proving this we need the following results about entropy.

\begin{lemma} \label{lemma:relative_ent_small}
Let $X$ and $Y$ be discrete random variables defined on the same probability space each having finitely many possible values. Suppose that $K$ is an integer such that for each $y$ in the image of $Y$ there are at most $K$ elements $x$ in the image of $X$ such that
\begin{equation*}
\mathbb{P} \left[ \{X = x\} \cap \{Y = y\} \right] > 0.
\end{equation*}
Then
\begin{equation*}
H(X|Y) \leq \log K.
\end{equation*}
\end{lemma}

\begin{proof}
Note that $(X|Y)$ is almost surely supported on at most $K$ points. This means that
\begin{equation*}
H((X|Y)) \leq  \log K
\end{equation*}
almost surely. The result now follows by Lemma \ref{lemma:relative_ent_expectation}.
\end{proof}

\begin{lemma} \label{lemma:haar_small}
Given $u> 0$ let $K_u$ denote the set
\begin{equation*}
K_u := \{ g \in \Gp : \lnb g \rnb \leq u\}.
\end{equation*}
Then
\begin{equation*}
\tilde{m}(K_u) \leq O(u^2).
\end{equation*}
Here $\tilde{m}$ is the Haar measure on $\Gp$ defined in \ref{definition:tilde_m}.
\end{lemma}

The proof of Lemma \ref{lemma:haar_small} is a simple computation involving the Haar measure which we will carry out later in this section.

We now have everything we need to prove Proposition \ref{proposition:stopping_time_gap}.

\begin{proof}[Proof of Proposition \ref{proposition:stopping_time_gap}]
First we will deal with \eqref{eq:r1_ent_lower_bound}. Fix some $\varepsilon > 0$ which is sufficiently small depending on $M$ and $\mu$. Let $m = \floor{\frac{\log P}{\chi}}$ and define $\tilde{\tau}$ by
\begin{equation*}
\tilde{\tau} = 
\begin{cases}
\ceil{(1+\varepsilon) m} & \text{if } \tau > \ceil{(1+\varepsilon) m}\\
\floor{(1-\varepsilon) m} & \text{if } \tau < \floor{(1-\varepsilon) m}\\
\tau & \text{otherwise.}
\end{cases}
\end{equation*}
Given some random variable $X$ let $\mathcal{L}(X)$ denote its law. If we are also given some event $A$ we will let $\mathcal{L}(X)|_{A}$ denote the (not necessarily probability) measure given by the push forward of the restriction of $\mathbb{P}$ to $A$ under the random variable $X$. Note that $\lnb \mathcal{L}(X)|_{A} \rnb_1 = \mathbb{P}[A]$.

Given $n \in \Z_{>0}$ let $q_n = \gamma_1 \dots \gamma_n$. We have the following inequality.
\begin{align}
H(q_{\tau} \s_{r_1, a}) &= H(\mathcal{L}(q_{\tau}) * \mathcal{L}( \s_{r_1, a})) \nonumber \\
& \geq H(\mathcal{L}(q_{\tau})|_{\tau = \tilde{\tau}} * \mathcal{L}( \s_{r_1, a})) + H(\mathcal{L}(q_{\tau})|_{\tau \neq \tilde{\tau}} * \mathcal{L}( \s_{r_1, a})) \label{eq:some_random_temp}\\
& \geq H(\mathcal{L}(q_{\tau})|_{\tau = \tilde{\tau}} * \mathcal{L}( \s_{r_1, a})) + \mathbb{P}[\tau \neq \tilde{\tau}] H(\mathcal{L}( \s_{r_1, a})) \label{eq:r1_ent_lower_bound_ex}
\end{align}
Here \eqref{eq:some_random_temp} follows from Lemma \ref{lemma:entropy_is_concave} and \eqref{eq:r1_ent_lower_bound_ex} follows from Lemmas \ref{lemma:invariance_entropy} and \ref{lemma:entropy_is_concave}. 

First we will bound $H(\mathcal{L}(q_{\tau})|_{\tau = \tilde{\tau}})$. To do this we let for $i \in \Z_\geq 0$ we let $q_i := \gamma_1 \gamma_2 \dots \gamma_i$ and we introduce the random variable $\tilde{X}$ which is defined by
\begin{equation*}
\tilde{X} = \left( q_{\floor{(1-\varepsilon) m}}, \gamma_{\floor{(1-\varepsilon) m}+1}, \gamma_{\floor{(1-\varepsilon) m}+2}, \dots, \gamma_{\ceil{(1+\varepsilon) m}} \right).
\end{equation*}
We know that $q_{\tilde{\tau}}$ is completely determined by $\tilde{X}$ so
\begin{equation}
H(\tilde{X} | q_{\tilde{\tau}}) = H(\tilde{X}) - H(q_{\tilde{\tau}}). \label{eq:q_tau_tilde_x_1}
\end{equation}

Let $K$ be the number of points in the support of $\mu$. Clearly if $$\gamma_{\floor{(1-\varepsilon) m}+1}, \gamma_{\floor{(1-\varepsilon) m}+2}, \dots, \gamma_{\ceil{(1+\varepsilon) m}}$$ and $\tilde{\tau}$ are fixed then for any possible value of $q_{\tilde{\tau}}$ there is at most one choice of $q_{\floor{(1-\varepsilon) m}}$ which would lead to this value of $q_{\tilde{\tau}}$. Therefore for each $y$ in the image of $q_{\tilde{\tau}}$ there are at most
\begin{equation*}
(2 \varepsilon m + 2) K^{(2 \varepsilon m + 2)}
\end{equation*}
elements $x$ in the image of $\tilde{X}$ such that $\mathbb{P}[\tilde{X} = x \cap  q_{\tilde{\tau}} = y] > 0$. By Lemma \ref{lemma:relative_ent_small} this gives
\begin{equation}
H(\tilde{X} | q_{\tilde{\tau}}) \leq \log \left( (2 \varepsilon m + 2) K^{(2 \varepsilon m + 2)} \right) \leq \frac{ 2 \varepsilon \log K}{\chi} \log P + o_{\mu}(\log P). \label{eq:q_tau_tilde_x_2}
\end{equation}
We also know that
\begin{equation}
H(\tilde{X}) \geq H(q_{m}) \geq h_{RW} \cdot m \geq \frac{h_{RW}}{\chi} \log P - o_{\mu}(\log P). \label{eq:q_tau_tilde_x_3}
\end{equation}
Combining equations \eqref{eq:q_tau_tilde_x_1}, \eqref{eq:q_tau_tilde_x_2} and \eqref{eq:q_tau_tilde_x_3} gives
\begin{equation*}
H(q_{\tilde{\tau}}) \geq \frac{h_{RW} - 2 \varepsilon \log K}{\chi } \log t - o_{\mu}(\log t).
\end{equation*}

We note by Lemma \ref{lemma:entropy_is_almost_convex} that
\begin{equation*}
H(\mathcal{L}(q_{\tilde{\tau}})) \leq H(\mathcal{L}(q_{\tilde{\tau}})|_{\tau = \tilde{\tau}}) + H(\mathcal{L}(q_{\tilde{\tau}})|_{\tau \neq \tilde{\tau}}) + H(\mathbb{I}_{\tau = \tilde{\tau}}).
\end{equation*}
We wish to use this to bound $H(\mathcal{L}(q_{\tilde{\tau}})|_{\tau = \tilde{\tau}})$ from below. First note that trivially $H(\mathbb{I}_{\tau = \tilde{\tau}}) \leq \log 2 \leq o(\log P)$. Note that by \eqref{eq:large_deviations_stopping_time_v} from Lemma \ref{lemma:first_large_deviations_estimates} we have that providing $P$ is sufficiently large depending on $\varepsilon$ and $\mu$
\begin{equation*}
\mathbb{P}\left[ \tau \neq \tilde{\tau} \right] \leq \alpha^m
\end{equation*}
for some $\alpha \in (0, 1)$ which depends only on $\varepsilon$ and $\mu$. We also know that conditional on $\tau \neq \tilde{\tau}$ there are at most $K^{\ceil{(1+\varepsilon) m}} + K^{\floor{(1-\varepsilon) m}}$ possible values for $q_{\tilde{\tau}}$. This means that
\begin{equation*}
H(\mathcal{L}(q_{\tilde{\tau}})|_{\tau \neq \tilde{\tau}}) \leq \alpha^m \log \left( K^{\ceil{(1+\varepsilon) m}} + K^{\floor{(1-\varepsilon) m}} \right) \leq o_{\mu, \varepsilon}(\log P).
\end{equation*}
Therefore
\begin{equation*}
H(\mathcal{L}(q_{\tilde{\tau}})|_{\tau = \tilde{\tau}}) \geq \frac{h_{RW} - 2 \varepsilon \log K}{\chi} \log P - o_{\mu, \varepsilon}(\log P).
\end{equation*}

Recall that $d$ is the distance function of some left invariant Riemannian metric and that by the definition of $M_{\mu}$ given any $N \in \Z_{>0}$ and any two distinct $x, y \in \Gp$ such that for each of them there is some $n \leq N$ such that they are in the support of $\mu^{*n}$ we have
\begin{equation*}
d(x, y) \geq M_{\mu}^{-N+o_{\mu}(N)}
\end{equation*}
In particular this means that if $x$ and $y$ are both in the image of $q_{\tilde{\tau}}$ then
\begin{equation*}
d(x, y) \geq M_{\mu}^{-m(1+\varepsilon)+o_{\mu}(N)}.
\end{equation*}

Note also that trivially for all sufficiently small $r$ we have $d(\exp(u), \id) \leq O (r)$ whenever $u\in \A$ satisfies $\lnb u \rnb \leq r$. In particular since $r_1 < M^{-m}$ this means that providing $P$ is sufficiently large depending on $M$ and $a$ we have
\begin{equation*}
d(\s_{r_1, a}, \id) \leq O( a M^{-m})
\end{equation*}
almost surely. Therefore, providing $\varepsilon$ is small enough that $M_{\mu}^{(1+\varepsilon)} < M$ and $t$ is sufficiently large depending on $\mu$, $a$, $\varepsilon$ and $M$ we have
\begin{equation*}
d(\s_{r_1, a}, \id) < \frac{1}{2} \min_{x, y \in \supp \mathcal{L}(q_{\tilde{\tau}}), x \neq y} d(x, y).
\end{equation*}
In particular by Lemma \ref{lemma:entropy_adds} and Definition \ref{definition:non_prob_entropy} we have
\begin{equation*}
H(\mathcal{L}(q_{\tau})|_{\tau = \tilde{\tau}} * \mathcal{L}( \s_{r_1, a})) = H(\mathcal{L}(q_{\tau})|_{\tau = \tilde{\tau}}) + \mathbb{P}[\tau = \tilde{\tau}] H(\mathcal{L}( \s_{r_1, a})).
\end{equation*}
Putting this into the estimate \eqref{eq:r1_ent_lower_bound_ex} for $H(q_{\tau} \s_{r_1, a})$ we get 
\begin{equation*}
H(q_{\tau} \s_{r_1, a}) \geq \frac{h_{RW} - 2 \varepsilon \log K}{\chi} \log P + H(\s_{\s_1, a}) - o_{\mu, M, a, \varepsilon}(\log P).
\end{equation*}
Since $\varepsilon$ can be made arbitrarily small this becomes
\begin{equation*}
H(q_{\tau} \s_{r_1, a}) \geq \frac{h_{RW}}{\chi} \log P + H(\s_{r_1, a}) - o_{\mu, M, a}(\log P)
\end{equation*}
as required.

Now to prove \eqref{eq:r2_ent_upper_bound}. Fix some $\varepsilon > 0$ and let $A$ be the event that $$ \lnb q_{\tau} \rnb < P^{1+ \varepsilon}.$$
First note that by \eqref{eq:large_deviations_estimate} and \eqref{eq:large_deviations_stopping_time_v} from Lemma \ref{lemma:first_large_deviations_estimates} there is some $\delta$ depending on $\mu$ and $\varepsilon$ such that for all sufficiently large (depending on $\mu$, $\varepsilon$ and $b$) $t$ we have $$\mathbb{P}[A^C] < t^{-\delta}.$$

Note that when $A$ occurs $\lnb q_{\tau} \s_{r_2, a} \rnb \leq P^{1+\varepsilon} a r_2$. Therefore by Lemma \ref{lemma:haar_small} when $A$ occurs $q_{\tau} \s_{r_2, a}$ is contained in a set of $\tilde{m}$-measure at most $O_{\mu, a}(P^{2+2 \varepsilon})$ where $\tilde{m}$ is our normalised Haar measure. Trivially by Jensen's inequality this gives
\begin{equation}
H(\mathcal{L}(q_{\tau} \s_{r_2, a})|_A) \leq (2+ 2 \varepsilon)\log P + o_{\mu, M, a}(\log P). \label{eq:bound_something_on_a}
\end{equation}

Now we need to bound $H(\mathcal{L}(q_{\tau} \s_{r_2, a})|_{A^C})$. We will do this by bounding the Shannon entropy $H(\mathcal{L}(q_{\tau})|_{A^C})$. It is easy to see that the contribution to this from the case where $\tau < \frac{2 \log P}{\chi}$ is at most $P^{-\delta} \frac{2 \log P}{\chi} \log K$. By \eqref{eq:large_deviations_stopping_time_v} from Lemma \ref{lemma:first_large_deviations_estimates} the contribution from the case where $\tau = n$ for some $n \geq \frac{2 \log P}{\chi}$ can be bounded above by $\alpha^{n} n \log K$ where $\alpha \in (0, 1)$ is some constant depending only on $\mu$. From summing over $n$ it is easy to see that $$H(\mathcal{L}(q_{\tau})|_{A^C}) \leq o_{\mu}(\log P).$$ This gives
 $H(\mathcal{L}(q_{\tau} \s_{r_2, a})|_{A^C}) < o_{\mu, M, a}(\log P)$. Combining this with \eqref{eq:bound_something_on_a} and noting that $\varepsilon$ can be arbitrarily small gives \eqref{eq:r2_ent_upper_bound}.

Subtracting \eqref{eq:r2_ent_upper_bound} from \eqref{eq:r1_ent_lower_bound} gives
\begin{equation*}
H(q_{\tau} \s_{r_1, a}) - H(q_{\tau} \s_{r_2, a}) \geq \left( \frac{h_{RW}}{\chi} -2\right) \log P + H(\s_{r_1, a}) - o_{M, \mu, a}(\log P).
\end{equation*}

Noting that $|H(\s_{r_2, 1}) - 3 \log r_2| \leq O_{a}(1) \leq o_{M, \mu, a}(\log P)$ gives \eqref{eq:ent_gap} as required.
\end{proof}

We will now prove Lemma \ref{lemma:haar_small}. To do this we will use the following explicit formula for the Haar measure on $\Gp$.

\begin{definition}[Iwasawa decomposition]
Each element of $\Gp$ can be written uniquely in the form
\begin{equation*}
\begin{pmatrix}
1 & x\\
0& 1
\end{pmatrix}
\begin{pmatrix}
y^{\frac{1}{2}} & 0\\
0& y^{-\frac{1}{2}} 
\end{pmatrix}
\begin{pmatrix}
\cos \theta & - \sin \theta\\
\sin \theta & \cos \theta
\end{pmatrix}
\end{equation*}
with $x \in \R$, $y \in \R_{>0}$ and $\theta \in \Bim$. This is called the Iwasawa decomposition.
\end{definition}
	
\begin{lemma} \label{lemma:haar_nak_form}
There is a Haar measure for $\Gp$ which is given in the Iwasawa decomposition by
\begin{equation*}
\frac{1}{y^2} \,dx\,dy\,d\theta.
\end{equation*}
\end{lemma}

\begin{proof}
This is proven in for example \cite[Chapter III]{LANG_1985}.
\end{proof}

\begin{proof}[Proof of Lemma \ref{lemma:haar_small}]
First let
\begin{equation*}
M_{x, y, \theta} : =
\begin{pmatrix}
1 & x\\
0& 1
\end{pmatrix}
\begin{pmatrix}
y^{\frac{1}{2}} & 0\\
0& y^{-\frac{1}{2}} 
\end{pmatrix}
\begin{pmatrix}
\cos \theta & - \sin \theta\\
\sin \theta & \cos \theta
\end{pmatrix}.
\end{equation*}
Note that we have
\begin{equation*}
M_{x, y, \theta}
\begin{pmatrix}
\cos \theta \\-\sin \theta
\end{pmatrix} = 
\begin{pmatrix}
y^{\frac{1}{2}}\\0
\end{pmatrix}
\end{equation*}
and
\begin{equation*}
M_{x, y, \theta}
\begin{pmatrix}
\sin \theta \\ \cos \theta
\end{pmatrix} = 
\begin{pmatrix}
xy^{-\frac{1}{2}}\\y^{-\frac{1}{2}}
\end{pmatrix}
\end{equation*}
meaning that
\begin{equation*}
\lnb M_{x, y, \theta} \rnb \geq \max \{ y^{\frac{1}{2}}, |x|y^{-\frac{1}{2}}, y^{-\frac{1}{2}} \}.
\end{equation*}
By Lemma \ref{lemma:haar_nak_form} and the fact that any two Haar measures differ only by a positive multiplicative constant we have
\begin{align*}
\tilde{m} (K_P) & \leq O \left( \int_{P^{-2}}^{P^2} \int_{-P y^{\frac{1}{2}}}^{P y^{ \frac{1}{2}}} \int_0^{2 \pi} \frac{1}{y^2} \, d \theta\, dx \, dy \right) \\
&= O \left( P \int_{P^{-2}}^{P^2} y ^{- \frac{3}{2}} \, dy \right) \\
&\leq O \left( P \int_{P^{-2}}^{\infty} y ^{- \frac{3}{2}} \, dy \right) \\
&= O(P^2)
\end{align*}
as required.
\end{proof}

\subsection{Variance of a disintegration of a stopped random walk} \label{section:variance_of_disint_stopped_rw}
In this subsection we will prove Proposition \ref{proposition:int_of_v} and then use this to prove Proposition \ref{proposition:lots_of_v}.

\begin{proof}[Proof of Proposition \ref{proposition:int_of_v}]
Let $\tau = \tau_{P,b}$ and let $a \geq 1$ be a number we will choose later. Let $r_1 = a^{-1}M^{-\frac{\log P}{\chi}}$ and let
\begin{equation*}
N = \floor{(1-\frac{h_{RW}}{10 \log M}) \frac{\log M \log P}{\chi \log 2}} - 1.
\end{equation*}
Note that
\begin{equation*}
 \frac{1}{4} P^{\frac{\log M}{\chi}} / P^{\frac{h_{RW}}{10 \chi}} \leq 2^N \leq \frac{1}{2} P^{\frac{\log M}{\chi}} / P^{\frac{h_{RW}}{10 \chi}}.
\end{equation*}

Given $u \in [1, 2)$ and an integer $i \in [1, N]$ let
\begin{equation*}
k_i(u) := H(q_{\tau} m_{2^{i-1} ur_1, a}) - H( m_{2^{i-1}u r_1, a}) - H(q_{\tau} m_{2^{i} ur_1, a}) + H(m_{2^{i} ur_1, a}).
\end{equation*}
Note that by Proposition \ref{proposition:ent_gap_to_v} there is some absolute constant $c > 0$ such that we have
\begin{equation}
v(q_{\tau} ; a 2^{i} ur_1) \geq c a^{-2} (k_i(u) - O(e^{-\frac{a^2}{4}}) - O_a( 2^{i} r_1)). \label{eq:c_def}
\end{equation}
This means that
\begin{equation*}
\sum_{i=1}^N v(q_{\tau} ; a 2^{i} ur_1) \geq c a^{-2} \sum_{i=1}^{N} k_i(u) - O(N e^{-\frac{a^2}{4}} a^{-2}) - O_a( N  2^{N} r_1).
\end{equation*}
Note that for $u \in [1, 2)$ we have
\begin{equation*}
a 2^{N} ur_1 \leq t^{-\frac{h_{RW} }{10 \chi}}
\end{equation*}
and
\begin{equation*}
a 2^{1} ur_1 \geq t^{-\frac{\log M }{\chi}}.
\end{equation*}
This means that
\begin{equation}
\int_{P^{ - \frac{\log M}{\log \chi}}}^{t^{ - \frac{h_{RW}}{10 \log \chi}}} \frac{1}{u} v(q_{\tau}; u) \, du \geq c a^{-2} \int_{1}^{2} \frac{1}{u} \sum_{i=1}^{N} k_i(u) \, du - O(N e^{-\frac{a^2}{4}} a^{-2}) - O_a( N  2^{N} r_1). \label{eq:first_int_exp}
\end{equation}

Clearly for any fixed $u \in [1, 2)$ we have
\begin{equation*}
\sum_{i=1}^{N} k_i(u) = H(q_{\tau} m_{ ur_1, a}) - H( m_{u r_1, a}) - H(q_{\tau} m_{2^{N} ur_1, a}) + H(m_{2^{N} ur_1, a}).
\end{equation*}
This means that by Proposition \ref{proposition:stopping_time_gap} we have
\begin{align}
 \sum_{i=1}^{N} k_i(u) &\geq \left(\frac{h_{RW}}{\chi} -12\right) \log P + 3 \log 2^{N} ur_1 +o_{M, \mu, a, w}(\log P)\nonumber\\
 &\geq \left(\frac{h_{RW}}{\chi} -2 - \frac{3 h_{RW} }{10 \chi} \right) \log P  + o_{M, \mu, a, w}(\log P). \label{eq:k_sum_thing}
\end{align}

Let $C$ be chosen such that the error term $O(N e^{- \frac{a^2}{4}} a^{-2})$ in \eqref{eq:first_int_exp} can be bounded above by $CN e^{- \frac{a^2}{4}} a^{-2}$. Note that this is at most $ O \left( \frac{\log M}{\chi \log 2} e^{- \frac{a^2}{4}} a^{-2} \log P \right)$. Let $c$ be as in \eqref{eq:c_def}. We take our value of $a$ to be
\begin{equation*}
a = 2 \sqrt{\log \left( \frac{100 C}{c \log 2} \frac{\log M}{h_{RW}} \right)}.
\end{equation*}
Note that $a$ depends only on $\mu$ and $M$. This means 
\begin{equation*}
CN e^{- \frac{a^2}{4}} a^{-2} \leq a^{-2} \frac{h_{RW}}{100 \chi} c \log P.
\end{equation*}
Note also that $N  2^{N} r_1 \leq o_{\mu, M}(\log P)$. Therefore putting \eqref{eq:k_sum_thing} into \eqref{eq:first_int_exp} we get
\begin{equation*}
\int_{P^{ - \frac{\log M}{\chi}}}^{P^{ - \frac{h_{RW}}{10\chi}}} \frac{1}{u} v(q_{\tau}; u) \, du \geq c a^{-2} \left(\frac{h_{RW}}{\chi} -2 -  \frac{3 h_{RW} }{10 \chi}  - \frac{h_{RW}}{100 \chi} \right) \log P  + o_{M, \mu, w}(\log P). 
\end{equation*}
In particular providing $\frac{h_{RW}}{\chi} > 10$ we have
\begin{equation*}
\int_{P^{ - \frac{\log M}{\chi}}}^{P^{ - \frac{h_{RW}}{10\chi}}} \frac{1}{u} v(q_{\tau}; u) \, du \gtrsim a^{-2} \left(\frac{h_{RW}}{\chi} \right) \log P  + o_{M, \mu, w}(\log P). 
\end{equation*}
Noting that $a^{2} \leq O(\max \left \{ 1,  \log \frac{\log M}{h_{RW}} \right \}  )$ we have  that for all sufficiently large (depending on $\mu$, $M$, and $w$) $P$ we have
\begin{equation*}
\int_{P^{ - \frac{\log M}{\log \chi}}}^{P^{ - \frac{h_{RW}}{10 \log \chi}}} \frac{1}{u} v(q_{\tau}; u) \, du \gtrsim \left(\frac{h_{RW}}{ \chi}\right) \left( \max \left \{ 1,  \log \frac{\log M}{h_{RW}} \right \} \right)^{-1} \log P 
\end{equation*}
as required.

\end{proof}

We wish to prove Proposition \ref{proposition:lots_of_v}. First we need the following corollary of Proposition \ref{proposition:int_of_v}.

\begin{corollary} \label{corollary:int_of_v}
Suppose that $\hat{\nu}$ is a probability measure on $\B$. Suppose that $\mu$ is a finitely supported Zariski-dense probability measure. Suppose further that $M_{\mu} < \infty$ and let $M > M_{\mu}$. Suppose that $M$ is chosen large enough that $h_{RW} \leq \log M$. Then for all sufficiently large (depending on $\mu$, $\hat{\nu}$, and $M$) $P$ we have
\begin{align*}
\int_{\B} \int_{P^{ - \frac{\log M}{\log \chi}}}^{P^{ - \frac{h_{RW}}{10 \log \chi}}} \frac{1}{u} v(q_{\tau_{P, b}}; u) \, du \, \hat{\nu} (db) \gtrsim\\  \left(\frac{h_{RW}}{ \chi}\right) \left( \max \left \{ 1,  \log \frac{\log M}{\chi} \right \} \right)^{-1} \log P  .
\end{align*}

\end{corollary}

\begin{proof}
Given $\mu$ and $M$ let 
\begin{equation*}
S(P) := \{b \in \B : P \text{ is large enough to satisfy Proposition \ref{proposition:int_of_v} for this } b, \mu \text{ and } M \}.
\end{equation*}
By Proposition \ref{proposition:int_of_v} we know that $S(P) \nearrow \B$. Therefore $\hat{\nu}(S(P)) \nearrow 1$. In particular providing $P$ is sufficiently large (depending on $\mu$ and $M$) we have $\hat{\nu}(S(P)) \geq \frac{1}{2}$. This, along with the fact that $v(\cdot ; \cdot)$ is always non-negative, is enough to prove Corollary \ref{corollary:int_of_v}.
\end{proof}

This is enough to prove Proposition \ref{proposition:lots_of_v}.

\begin{proof}[Proof of Proposition \ref{proposition:lots_of_v}]
Recall that $\hat{m} = \floor{\frac{\log M}{100 \chi}}$. Let

\begin{equation*}
A := P^{\frac{\log M}{2\hat{m} \chi} - \frac{h_{RW}}{20 \hat{m} \chi}}.
\end{equation*}
Define $a_1, a_2 ,\dots, a_{2 \hat{m} + 1}$ by
\begin{equation*}
a_i := P^{- \frac{\log M}{\chi}} A^{i-1}.
\end{equation*}
Note that this means $a_1 = P^{- \frac{\log M}{\chi}}$ and $a_{2 \hat{m} + 1} = P^{- \frac{h_{RW}}{10 \chi}}$. Furthermore, providing $h_{RW} / \chi$ is sufficiently large we have
\begin{equation*}
P^3 \leq A \leq P^{50}.
\end{equation*}
In particular $a_{i+1} \geq P^3 a_i$.

Let $U, V$ be defined by
\begin{equation*}
U := \bigcup_{i=1}^{\hat{m}} [a_{2i - 1}, a_{2i})
\end{equation*}
and 
\begin{equation*}
V := \bigcup_{i=1}^{\hat{m}} [a_{2i}, a_{2i+1}).
\end{equation*}
Note that $U$ and $V$ partition $\left[ P^{- \frac{\log M}{\chi}}, P^{- \frac{h_{RW}}{10 \chi}} \right)$.

Let $c>0$ be the absolute constant in Corollary \ref{corollary:int_of_v}. By Corollary \ref{corollary:int_of_v} providing $P$ is sufficiently large depending on $\mu$ and $M$ we have
\begin{equation*}
 \int_{U \cup V} \int_{\B} \frac{1}{u} v(q_{\tau_{P, b}}; u) \, \hat{\nu}( db) \, du \geq c \left( \frac{h_{RW}}{\chi} \right) \left( \max \left \{ 1, \log \frac{\log M}{h_{RW}} \right \} \right) ^{-1} \log P.
\end{equation*}
In particular either
\begin{equation}
 \int_{U } \int_{\B} \frac{1}{u} v(q_{\tau_{P, b}}; u) \, \hat{\nu}( db) \, du \geq \frac{1}{2} c \left( \frac{h_{RW}}{\chi} \right) \left( \max \left \{ 1, \log \frac{\log M}{h_{RW}} \right \} \right) ^{-1} \log P. \label{eq:int_on_u_large}
\end{equation}
or
\begin{equation*}
 \int_{V} \int_{\B} \frac{1}{u} v(q_{\tau_{P, b}}; u) \, \hat{\nu}( db) \, du \geq \frac{1}{2} c \left( \frac{h_{RW}}{\chi} \right) \left( \max \left \{ 1, \log \frac{\log M}{h_{RW}} \right \} \right) ^{-1} \log P.
\end{equation*}

Without loss of generality assume that \eqref{eq:int_on_u_large} holds. For $i = 1, 2, \dots, \hat{m}$ let $s_i \in (a_{2i-1}, a_{2i})$ be chosen such that
\begin{equation*}
\int_{\B} v(q_{\tau_{P, b}}; s_i) \, \hat{\nu}( db) \geq \frac{1}{2} \sup_{u \in (a_{2i-1}, a_{2i})} \int_{\B} v(q_{\tau_{P, b}}; u) \, \hat{\nu}( db).
\end{equation*}
In particular this means that
\begin{equation*}
\int_{\B} v(q_{\tau_{P, b}}; s_i) \, \hat{\nu}( db) \geq \frac{1}{2 \log A} \int_{a_{2i-1}}^{a_{2i}} \int_{\B} \frac{1}{u} v(q_{\tau_{P, b}}; u) \, \hat{\nu}( db) \, du.
\end{equation*}
Summing over $i$ gives
\begin{align*}
\sum_{i=1}^{\hat{m}} \int_{\B} v(q_{\tau_{P, b}}; s_i) \, \hat{\nu}( db)& \geq \frac{1}{2 \log A} \int_{U } \int_{\B} \frac{1}{u} v(q_{\tau_{P, b}}; u) \, \hat{\nu}( db) \, du \\
& \geq \frac{1}{4 \log A} c \left( \frac{h_{RW}}{\chi} \right) \left( \max \left \{ 1, \log \frac{\log M}{h_{RW}} \right \} \right) ^{-1} \log P.
\end{align*}
Noting that $\log A \leq O(\log t)$ we get that providing $P$ is sufficiently large depending on $\mu$ and $M$ that
\begin{equation*}
\sum_{i=1}^{\hat{m}} \int_{\B} v(q_{\tau_{P, b}}; s_i) \, \hat{\nu}( db)  \geq c' \left( \frac{h_{RW}}{\chi} \right) \left( \max \left \{ 1, \log \frac{\log M}{h_{RW}} \right \} \right) ^{-1} 
\end{equation*}
for some absolute constant $c' > 0$. Finally note that $A \geq P^3$ means that $s_{i+1} \geq P^3 s_i$.
\end{proof}

\section{Variance Sum} \label{section:sum_of_variances}

Recall from the introduction that the strategy of the proof is as follows. We let $\left( \gamma_i \right)_{i=1}^{\infty}$ be i.i.d.\ samples from $\mu$ and let $b$ be an independent sample from $\nu$ and for each sufficiently small scale $r>0$ we construct some $\sigma$-algebra $\mathcal{A}$ and some stopping time $\tau$. We also construct some $n \in \Z_{>0}$, some $\mathcal{A}$-measurable random variables $g_1, g_2, \dots, g_n$ taking values in $\Gp$ and some random variables $U_1, U_2, \dots, U_n$ taking values in $\A$ such that 
\begin{equation}
\gamma_1 \gamma_2 \dots \gamma_{\tau} b = g_1 \exp(U_1) g_2 \exp(U_2) \dots g_n \exp(U_n) b. \label{eq:decomp_for_explanation}
\end{equation}
We also require the $U_i$ to be small and have at least some variance after conditioning on $\mathcal{A}$. We then condition on $\mathcal{A}$ and Taylor expand in the $U_i$ so that after disintegrating we may express the Furstenberg measure as the law of the sum of many small random variables each of which have at least some variance.

In order to carry out this Taylor expansion we will use Proposition \ref{proposition:intro_decomp_detail}. This requires the $g_i$ to satisfy a number of conditions. We wish to construct a class of ways of expressing random variables of the form $\gamma_1 \dots \gamma_{\tau}$ in the form $g_1 \exp(U_1) \dots g_n \exp(U_n)$ such that the $g_i$ and $U_i$ satisfy (amongst other things) the conditions of Proposition \ref{proposition:intro_decomp_detail} and so that this class is closed under concatenation. To this end we define the following.

\begin{definition} \label{definition:proper_decomposition}
    Let $\mu$ be a probability measure on $\Gp$, let $n, K \in \Z_{\geq 0}$, let $a$ and $\overline{a}$ be random variables taking values in $\Gp$ and let $C, t, \varepsilon, r>0$. Let $f=(f_i)_{i=1}^{n}$ and $h=(h_i)_{i=1}^{n}$ be sequences of random variables taking values in $\Gp$. Let $U=(U_i)_{i=1}^{n}$ be a sequence of random variables taking values in $\A$, let $\mathcal{A}=\left(\mathcal{A}_i\right)_{i=0}^{n}$ be a sequence of $\sigma$-algebras, let $A$ be an $\mathcal{A}_n$-measurable event, let $I$ be a random subset of $[1,n] \cap \Z$ and let $m=\left( m_i \right)_{i=1}^{n}$ be a sequence of non-negative real numbers. Let $\gamma=\left( \gamma_i \right)_{i=1}^{\infty}$ be i.i.d.\ samples from $\mu$ and let $\mathcal{F}=\left(\mathcal{F}_i\right)_{i=1}^{\infty}$ be a filtration for $\gamma$ and suppose that for all $i$ we have that $\gamma_{i+1}$ is independent of $\mathcal{F}_i$. Let $S=\left(S_i\right)_{i=1}^{n}$ and $T=\left(T_i\right)_{i=1}^{n}$ be sequences of stopping times for the filtration $\mathcal{F}$. Let $\ell$ be a random variable taking values in $\Gp$. Then we say that $$\left(f, h, U, m, \mathcal{A}, A, I,\gamma, \mathcal{F}, S, T, \ell \right)$$ is a \emph{proper} decomposition for $\left( \mu, n, K, a, \overline{a}, t, C , \varepsilon \right)$ at scale $r$ if $\mathbb{P}[A] \geq 1 - \varepsilon$ and on $A$ the following conditions are satisfied.
    \begin{condenum}
        \item We have $S_1 \leq T_1 \leq S_2 \leq T_2 \leq \dots \leq S_n \leq T_n$. \label{condition:increasing_s_t}
        \item We have $f_1 \exp(U_1) = \gamma_1 \dots \gamma_{S_1}$ and for $i = 2, \dots, n$ we have $f_i \exp(U_i) = \gamma_{T_{i-1}+1} \dots \gamma_{S_i}$. \label{condition:shape_of_f}
        \item We have $h_i = \gamma_{S_i+1} \dots \gamma_{T_i}$. \label{condition:shape_of_h}
        \item The $\mathcal{A}_i$ are nested - that is $\mathcal{A}_0 \subset \mathcal{A}_1 \subset \dots \subset \mathcal{A}_n$. \label{condition:As_nested}
        \item For each $i = 1, 2, \dots, n$ we have that $U_i$ is conditionally independent of $\mathcal{A}_n$ given $\mathcal{A}_i$. \label{condition:u_i_conditional_indep_of_a_i}
        \item The $U_i$ are conditionally independent given $\mathcal{A}_n$. \label{condition:u_i_conditional_indep_of_u_i}
        \item We have that $a$ and $\overline{a}$ are $\mathcal{A}_0$ measurable and for each $i = 1, \dots, n$ the $f_i$ and $h_i$ are $\mathcal{A}_i$-measurable. \label{condition:f_and_h_a_measurable}
        \item For each $i = 1, 2, \dots, n$ we have $$\mathbb{E} \left[ \frac{\var [U_{i} | \mathcal{A}_i]}{\lnb a f_1 h_1 f_2 h_2 \dots f_{i-1} h_{i-1} f_i \rnb ^4 r^2} | \mathcal{A}_{i-1} \right] \geq m_i.$$ \label{condition:average_variance_large}
        \item For each $i \in [1, n] \cap \Z \backslash I$ we have $U_i = 0$. \label{condition:u_i_zero}
        \item For each $i \in I$ we have $$\lnb U_i \rnb \leq \lnb a f_1 h_1 f_2 h_2 \dots f_{i-1} h_{i-1} f_i \rnb ^2 r$$ almost surely and $b^{+}(h_i) \in U_{t/4, t/8}(U_i|\mathcal{A}_n)$. \label{condition:in_u_t}
        \item When $I$ is not empty if we enumerate $I$ as $\{ j_1, \dots, j_{p} \}$ with $j_1 < \dots < j_p$ and define $g_1 := \overline{a} f_1 h_1 f_2 h_2 \dots f_{j_1} $ and for $i = 2, \dots, p$ define $g_i :=h_{j_{i-1}} f_{j_{i-1}+1} h_{j_{i-1}+1} \dots f_{j_{i} - 1} h_{j_{i} -1} f_{j_{i}}$. Then for each $i=1, \dots, p$ we have $$\lnb g_i \rnb \geq C.$$ \label{condition:breaks_up_large}
        \item With $g_i$ defined as above when $I$ is not empty for $i = 1, \dots, \ell$ we have $$d(b^{-}(g_i), b^{+}(h_{j_i})) > t/4.$$ \label{condition:breaks_up_alligned}
        \item For $i = 1, \dots, n$ we have $T_i \geq S_i+K$. \label{condition:k_gap}
        \item We have $\ell = h_{j_m} f_{j_m+1} h_{j_m+1} \dots f_{n} h_{n}$ \label{condition:l_def}
    \end{condenum}
\end{definition}

We refer to $\ell$ as the \emph{tail} of the decomposition.

This definition is chosen such that given a proper decomposition we can write
\begin{equation*}
    a \gamma_1 \dots \gamma_{T_n} = a g_1 \exp(U_1) g_2 \exp(U_2) \dots g_m \exp(U_m) g_{m+1}
\end{equation*}
and then Taylor expand in the $U_i$ after conditioning on $\mathcal{A}_n$. The $\sigma$- algebra $\mathcal{A}_n$ will play a similar role to the $\sigma$-algebra $\mathcal{A}$ in \eqref{eq:decomp_for_explanation}.

We will now briefly discuss the purpose of each of these conditions. Conditions \ref{condition:increasing_s_t}, \ref{condition:shape_of_f} and \ref{condition:shape_of_h} are needed to describe the shape of the decomposition. We require Conditions \ref{condition:As_nested} and \ref{condition:u_i_conditional_indep_of_a_i} in order to ensure that $\var[U_i|\mathcal{A}_n] = \var[U_i|\mathcal{A}_i]$ and in particular is an $\mathcal{A}_i$ measurable random variable. This enables us to apply a quantitative version Cramer's Theorem (see Lemma \ref{lemma:cramer}) to show that after conditioning on $\mathcal{A}_n$ the sum of the variances of the random variables produced by Taylor expanding \eqref{eq:decomp_for_explanation} in the $U_i$ will, with very high probability, not be too small. Condition \ref{condition:u_i_conditional_indep_of_u_i} is needed for the small random variables given by this to be independent. Condition \ref{condition:f_and_h_a_measurable} is also important in this step and is needed to ensure that the $g_i$ are $\mathcal{A}_n$-measurable.

We need to introduce the set $I$ because if $b^{-}(f_i)$ is too close to $b^{+}(h_i)$ then we will not have good control on the derivatives with respect to $U_i$. This will prevent us from being able to use our Taylor expansion. We cannot get around by for example replacing $f_i$ by $$\tilde{f}_i := \begin{cases}
    f_i & \text{ if $i \in I$}\\
    f_i \exp(U_i) h_i f_{i+1} & \text{ otherwise}
\end{cases}$$ 
and replacing $U_i$ by
$$\tilde{U}_i:= \begin{cases}
    U_i & \text{ if $i \in I$}\\
    U_{i+1} & \text{ otherwise}
\end{cases}$$ 
in this case because we will not know if we want $i \in I$ or not until after we define $h_i$. This means that $S_i$ will not be a stopping time.

Condition \ref{condition:average_variance_large} is needed to ensure that the small random variables we aquire after Taylor expanding in the $U_i$ have at least some variance.

Conditions \ref{condition:u_i_zero}, \ref{condition:in_u_t}, \ref{condition:breaks_up_large}, and \ref{condition:breaks_up_alligned} are needed to ensure that the conditions of Proposition \ref{proposition:intro_decomp_detail} are satisfied. Condition \ref{condition:k_gap} is needed to ensure that $b^{+}(h_{j_i})$ is a good approximation of $b^{+}(g_{i+1})$.

We introduce the filtration $\left(\mathcal{F}_i\right)_{i=1}^{\infty}$ instead of just taking $\mathcal{F}_i = \sigma(\gamma_1, \gamma_2, \dots, \gamma_i)$ because in our construction of a proper decomposition in Proposition \ref{proposition:v_to_varaince_sum} we need the $f_i$ to be $\mathcal{F}_{S_i}$ - measurable. The $f_i$ are not in general products of $\gamma_j$ and so are not in general $\sigma( \gamma_1, \gamma_2, \dots, \gamma_{S_i})$ - measurable.

Note that when $n=0$ a proper decomposition will always exist. We will call this the trivial proper decomposition.

\begin{definition} \label{definition:variance_sum}
    Given some probability measure $\mu$ on $\Gp$, some $P\geq 1$ some fixed $a, \overline{a}\in \Gp$ such that $\lnb a \rnb \leq P$, some $n, K \in \Z_{\geq 0}$, and some  $t, C, \varepsilon > 0$ we define the \emph{variance sum} for $\mu, n, K, t, C, \varepsilon$ from $a, \overline{a}$ to $P$ at scale $r$ to be the  maximum for $k = 0, 1, \dots, n$ of the supremum of all possible values of $$\sum_{i=1}^{k} m_i$$
    where $$\left(f, h, U, m, \mathcal{A}, A, I, \gamma, \mathcal{F}, S, T, \ell \right)$$ is a proper decomposition for $\left( \mu, k, K, a, \overline{a}, t, C, \varepsilon \right)$ at scale $r$ with $\lnb a f_1 h_1 \dots f_k h_k \rnb \leq P$ on the event $A$. We denote this by $W(\mu, n, K, a, \overline{a}, P, t, C, \varepsilon; r)$.
\end{definition}

To avoid trivial obstructions we also take this supremum over all possible underlying probability spaces. In particular we allow the probability space to be a regular space.

Note that since a proper decomposition always exists when $k=0$ we have $W(\mu, n, K, a, P, t, C, \varepsilon; r) \geq 0$. We now introduce the following.

\begin{definition}
    Given a probability measure $\mu$ on $\Gp$, $n \in \Z_{\geq 0}$, $P_1, P_2 \in \R$ with $1 \leq P_1 \leq P_2$ and some $t, C, \varepsilon, r > 0$ we define $$V(\mu, n, K, P_1, P_2, t, C, \varepsilon; r) := \inf \limits_{a, \overline{a} \in \Gp, \lnb a \rnb \leq P_1} W(\mu, n, K,  a, \overline{a}, P_2, t, C, \varepsilon; r).$$
\end{definition}

Trivially $V(\mu, n, P_1, P_2, t, C, \varepsilon; r) \geq 0$ due to the existence of the trivial decomposition. It is also clear that it is decreasing in $P_1$ and increasing in $P_2$. The quantity $V(\cdot, \cdot, \cdot, \cdot, \cdot, \cdot, \cdot; \cdot)$ will play an important role in the proof as is shown by the following propositions.

\begin{proposition} \label{proposition:varaince_sum_adds}
    Suppose that $\mu$ is a probability measure on $\Gp$, $n_1,  n_2 \in \Z_{\geq 0}$, $P_1, P_2, P_3 \in \R$ with $1 \leq P_1 \leq P_2 \leq P_3$ and $t, C, r, \varepsilon_1, \varepsilon_2 > 0$. Then we have
    \begin{align*}
        \MoveEqLeft V(\mu, n_1+n_2, P_1, P_3, t, C, \varepsilon_1+\varepsilon_2; r)\\
        &\geq V(\mu, n_1, P_1, P_2, t, C,\varepsilon_1; r) + V(\mu, n_2, P_2, P_3, t, C, \varepsilon_2; r).
    \end{align*}
\end{proposition}

We also wish to show that when the variance sum is large the order $k$ detail is small.

\begin{proposition} \label{proposition:variance_sum_to_order_k_detail}
    For every $\alpha, t > 0$ there are some constants $C, Q > 0$ such that the following is true. Suppose that $\mu$ is a finitely supported Zariski-dense probability measure on $\Gp$. Then there is some $c = c(\mu) >0$ such that whenever $P \geq 1$ and $k, K, n  \in \Z_{>0}$ with $K$ and $n$ sufficiently large (in terms of $t$, $\alpha$ and $\mu$), $r > 0$ is sufficiently small (in terms of $t$, $\alpha$ and $\mu$) and $$V(\mu, n, K, 1, P, t, C, \varepsilon; r) > Ck$$ we have
    \begin{equation}
        s_{Qr}^{(k)}(\nu) < \alpha^k + n \exp(- c K) + P^2 r C^n + \varepsilon. \label{eq:srk_var_bound}
    \end{equation}
\end{proposition}

When we apply this proposition the most important term in \eqref{eq:srk_var_bound} will be $\alpha^k$. Finally we need the following.

\begin{proposition} \label{proposition:lots_of_v_varaince_sum}
    For any $\alpha_0 \in (0, 1/3)$ and any $t, R > 0$ there exists some $c = c (\alpha_0, t, R) > 0$ such that the following is true. Suppose that $\mu$ is a finitely supported Zariski-dense probability measure. Suppose further that $\mu$ is $\alpha_0, t$-non-degenerate and that the operator norm is at most $R$ on the support of $\mu$. Suppose that $M_{\mu} < \infty$ and that $h_{RW} / \chi$ is sufficiently large.  Then there is some constant $c_2 = c_2(\mu) > 0$ such that the following holds. Let $M > M_{\mu}$ be chosen large enough that $\log M \geq h_{RW}$. Suppose that $P$ is sufficiently large (depending on $\mu$, $M$, $C$, $\alpha_0$, $t$ and $R$) and let $\hat{m} = \floor{\frac{\log M}{100 \chi}}$.

    Suppose that $r \in \left(0, P^{-\frac{ \log M}{\chi} - 4}\right)$ and that $K$ is a positive integer with $K \leq \frac{\log P}{10 \chi}$ and $K$ is sufficiently large (depending on $\mu$, $M$, $C$, $\alpha_0$, $t$ and $R$). Then
    \begin{align*}
        \MoveEqLeft V(\mu, \hat{m}, K, P^{-\frac{\log M}{ \chi}} r^{-1/2}, P^{-\frac{h_{RW}}{40 \chi}} r^{- 1/2}, t, C, \exp(-c_2 K); r)\\ &\geq c \left( \frac{h_{RW}}{\chi} \right)  \left( \max \left\{ 1,  \log \frac{ \log M}{h_{RW}} \right\} \right)^{-1}.
    \end{align*}

\end{proposition}

The rest of this section will be devoted to proving these three propositions. Later we will prove Theorem \ref{theorem:main_furstenberg} by using these three propositions to bound the order $k$ detail of the Furstenberg measure and then applying Lemma \ref{lemma:ind_srk_to_sr} and Lemma \ref{lemma:suff_abs_cont}.

\subsection{Proof of Proposition \ref{proposition:varaince_sum_adds}}

The proof of Proposition \ref{proposition:varaince_sum_adds} follows easily from the following Lemma.

\begin{lemma} \label{lemma:concaninate}
    Let $\mu$ be a probability measure on $\Gp$, let $n_1, n_2, K \in \Z_{\geq 0}$, let $a, \overline{a}$ be a random variables taking values in $\Gp$, and let $t, C, r, \varepsilon_1, \varepsilon_2 >0$. Suppose that $$\left(f^{(1)}, h^{(1)}, U^{(1)}, m^{(1)}, \mathcal{A}^{(1)}, A_1, I_1,  \gamma^{(1)} , \mathcal{F}^{(1)}, S^{(1)}, T^{(1)}, \ell_1\right)$$ is a proper decomposition for $\left( \mu, n_1, K, a, \overline{a}, t, C, \varepsilon_1 \right)$ at scale $r$ and denote it by $D_1$. Suppose that 
    $$\left(f^{(2)}, h^{(2)}, U^{(2)}, m^{(2)}, \mathcal{A}^{(2)}, A_2, I_2,  \gamma^{(2)} , \mathcal{F}^{(2)}, S^{(2)}, T^{(2)}, \ell_2\right)$$
    is a proper decomposition for $\left( \mu, n_2, K, af^{(1)}_1h^{(1)}_1 \dots f^{(1)}_{n_1}h^{(1)}_{n_1}, \ell_1, t, C, \varepsilon_2 \right)$ at scale $r$ and denote it by $D_2$. Suppose that $D_2$ is conditionally independent of $(a, D_1)$ given $af^{(1)}_1h^{(1)}_1 \dots f^{(1)}_{n_1}h^{(1)}_{n_1}$ and $\ell_1$.
    For $i=1, \dots, n_1 + n_2$ define $f_i^{(3)}$ by
    $$f_i^{(3)} = \begin{cases}
        f_i^{(1)} & \text{ if } i \leq n_1\\
        f_{i-n_1}^{(2)} & \text{ otherwise.}
    \end{cases}$$
    and define $h_i^{(3)}$, $m_i^{(3)}$, $S_i^{(3)}$ and $T_i^{(3)}$ similarly. Define $\mathcal{A}_i^{(3)}$ by
    $$\mathcal{A}_i^{(3)} = \begin{cases}
        \mathcal{A}_i^{(1)} & \text{ if } i \leq n_1\\
        \sigma(\mathcal{A}_{n_1}^{(1)}, \mathcal{A}_{i-n_1}^{(2)}) & \text{ otherwise.}
    \end{cases}$$
    Define $$I_3 := I_1 \cup \{i + n_1 : i \in I_2 \}.$$
    Let $T := T^{(1)}_{n_1}$ and for $i = 1, 2, \dots$ define $\gamma_i^{(3)}$ by
    $$\gamma_i^{(3)} = \begin{cases}
        \gamma_i^{(1)} & \text{ if } i \leq T\\
        \gamma_{i-T}^{(2)} & \text{ otherwise.}
    \end{cases}$$
    Define $\mathcal{F}_i^{(3)}$ by
    \begin{align*}
    \MoveEqLeft \mathcal{F}_i^{(3)} := \{ A \in \xi: A \cap \{ T \geq i \} \in \mathcal{F}_i^{(1)} \text{ and for all } j < i \text{ we have }\\ &\qquad\qquad\qquad A \cap \{T = j \} \in \sigma(\mathcal{F}_{T}^{(1)}, \mathcal{F}_{i-j}^{(2)}) \}
	\end{align*}      
	where $\xi$ is the set of events in our underlying probability space. Let $\ell_3 = \ell_2$. Then
    $$\left(f^{(3)}, h^{(3)}, U^{(3)}, m^{(3)}, \mathcal{A}^{(3)}, A_1 \cap A_2, I_1,  \gamma^{(3)} , \mathcal{F}^{(3)}, S^{(3)}, T^{(3)}, \ell_1\right)$$
    is a proper decomposition for $\left( \mu, n_1+n_2, K, a, P_2, t, C, \varepsilon_1+\varepsilon_2 \right)$ at scale $r$.
\end{lemma}
\begin{proof}
    It is easy to check that the $\gamma_i^{(3)}$ are independent by standard properties of stopping times. It is clear from checking the definition that $\mathcal{F}^{(3)}$ is a filtration for $\gamma^{(3)}$ and that the $T_i$ and $S_i$ are stopping times for this filtration. All of the conditions in Definition \ref{definition:proper_decomposition} follow immediately from construction.
\end{proof}

This is enough to prove Proposition \ref{proposition:varaince_sum_adds}.

\begin{proof}[Proof of Proposition \ref{proposition:varaince_sum_adds}]
    This follows immediately from Lemma \ref{lemma:concaninate}.
\end{proof}

\subsection{Proof of Proposition \ref{proposition:variance_sum_to_order_k_detail}}

In this subsection we will prove Proposition \ref{proposition:variance_sum_to_order_k_detail}. Before proving the proposition we need the following lemma.

\begin{lemma} \label{lemma:cramer}
Let $a, b, c > 0$ with $c \leq a$ and let $n \in \N$. Let $X_1, \dots, X_n$ be random variables taking values in $\R$ and let $m_1, \dots, m_n \geq 0$ be such that we have almost surely
\begin{equation*}
\mathbb{E}\left[ X_i | X_1, \dots, X_{i-1} \right] \geq m_i.
\end{equation*}
Suppose that $\sum_{i=1}^n m_i = an$. Suppose also that we have almost surely $X_i \in [0, b]$ for all integers $i \in [1, n]$. Then we have
\begin{equation*}
\mathbb{P}[X_1 + \dots +X_n \leq nc] \leq \left( \left( \frac{a}{c} \right) ^{\frac{c}{b}} \left( \frac{b-a}{b-c} \right) ^{1 - \frac{c}{b}} \right)^n.
\end{equation*}
\end{lemma}

The proof of this lemma is very similar to the standard proof of Cramer's Theorem. We will prove it after proving Propositon \ref{proposition:variance_sum_to_order_k_detail}. We also need the following Corollary.

\begin{corollary} \label{corollary:cramer}
    There is some constant $c>0$ such that the following is true for all $a \in [0,1)$. Let $n \in \Z_{>0}$, let $X_1, \dots, X_n$ be random variables taking values in $\R$ with and let $m_1, \dots, m_n \geq 0$ be such that we have almost surely
    \begin{equation*}
    \mathbb{E}\left[ X_i | X_1, \dots, X_{i-1} \right] \geq m_i.
    \end{equation*}
    Suppose that $\sum_{i=1}^n m_i = an$. Suppose also that we have almost surely $X_i \in [0, 1]$ for all integers $i \in [1, n]$. Then
    \begin{equation*}
        \log \mathbb{P}[X_1 + \dots + X_n \leq \frac{1}{2} n a] \leq - cna.
    \end{equation*}
\end{corollary}

We are now ready to prove Proposition \ref{proposition:variance_sum_to_order_k_detail}.

\begin{proof}[Proof of Proposition \ref{proposition:variance_sum_to_order_k_detail}]
The strategy of the proof is to apply Proposition \ref{proposition:intro_decomp_detail} to write our sample from the Furstenberg measure after conditioning on $\mathcal{A}$ as a sum of small independent random variables with at least some variance. We then use Lemma \ref{lemma:small_rvs_to_detail} and Lemma \ref{lemma:detail_order_k_detail_bound_many} to bound the order $k$ detail of this in terms of the sum of the variances of the small independent random variables. We then use Lemma \ref{lemma:cramer} to show that the sum of the variances is large with high probability and conclude by using the concavity of order $k$ detail.

First let $$(f, h, U, m, \mathcal{A}, A, I, \gamma, \mathcal{F}, S, T, \ell)$$ be a proper decomposition for $(\mu, n, K, \id, \id, t, C, \varepsilon)$ at scale $r$ such that $$\sum_{i=1}^{n} m_i \geq \frac{1}{2} C k.$$
Let $\overline{b}$ be an independent sample from $\nu$, let $b = \ell \overline{b}$ and let $\hat{\mathcal{A}} = \sigma(\mathcal{A}_n, b)$.

Let $p = |I|$ (note that this is an $\mathcal{A}_n$ measurable random variable) and let $g_1, \dots, g_p$ and $j_1, \dots, j_p$ be as in Definition \ref{definition:proper_decomposition}. For $i = 1, \dots, m$ let $u^{(i)} = U_{j_i}$. Let $x$ be defined by
$$x := g_1 \exp(u^{(1)}) \dots g_p \exp(u^{(p)}) b.$$
By Lemma \ref{lemma:stopping_time_invariant} $x$ is a sample from $\nu$.

Let $E_1$ be the event that for each $i = 1, \dots, p -1$ we have $$d(b^{+}(h_{j_i}), b^{+}(g_{i+1})) < t / 100$$
$$d(b^{+}(h_{j_i}), g_{i+1} g_{i+2} \dots g_p b) < t/ 100 $$
and  $$d(b^{+}(h_{j_p}), b) < t / 100.$$
Clearly $E_1$ is an $\hat{\mathcal{A}}$-measurable event and by \eqref{eq:c_exponential_convergence_to_fm} from Lemma \ref{lemma:convergence_to_fm} there is some $c>0$ depending only on $\mu$ such that providing $K$ is sufficiently large (in terms of $\mu$) we have $$\mathbb{P}[E_1] \geq 1 - n \exp(- c K).$$

Let $C_1$ be the $C$ from Proposition \ref{proposition:intro_decomp_detail} with $\frac{1}{8} t$ in the role of $t$. It is easy to check that, providing we choose $C$ to be sufficiently large, when $A \cap E_1$ occurs all of the conditions of Proposition \ref{proposition:intro_decomp_detail} are satisfied with $\frac{1}{8}t$ in the role of $t$ and $C_1$ in the role of $C$. This means that if for $i = 1, \dots, p$ we define $$\zeta_i := D_u (\phi(g_1 \dots g_i u g_{i+1} \dots g_p b))|_{u=0}$$
and we define $S \in \Bim$ by $$S := \phi(g_1g_2 \dots g_p) + \sum_{i=1}^{p} \zeta_i(u^{(i)})$$ then $$d(\phi(x), S) \leq C_1^n P^2 r^2.$$

In particular by Lemma \ref{lemma:detail_wasserstein} there is some absolute constant $C_2>0$ such that on $A \cap E_1$ we have $$s_{Qr}^{(k)}(x|\hat{\mathcal{A}}) \leq s_{Qr}^{(k)}(S|\hat{\mathcal{A}}) + C_2 C_1^n P^2 r.$$

We now wish to bound $s_{Qr}^{(k)}(S|\hat{\mathcal{A}})$ using Corollary \ref{corollary:small_rvs_to_order_k_detail}. To do this we need to estimate the variance of the $\zeta_i(u^{(i)})$ after conditioning on $\hat{\mathcal{A}}$. 

As in Definition \ref{definition:e_vec} given $y \in \B$ define $\rho_y \in \As$ by $$\rho_y := D_u(\phi(\exp(u) y))|_{u=0}.$$
By the chain rule we know that $$\zeta_i(u) = \left. \frac{\partial}{\partial y} \phi(g_1 g_2 \dots g_i y) \right|_{y = g_{i+1} \dots g_p b} \cdot \rho_{g_{i+1} \dots g_m p}(u).$$
By Proposition \ref{proposition:singular_value_shape} we know that providing $C$ is sufficiently large in terms of $t$ on the event $E_1$ we have
\begin{equation*}
    d(b^{-}(g_1 g_2 \dots g_i), g_{i+1} \dots g_p b) > t / 10.
\end{equation*}
In particular by Lemma \ref{lemma:derivs} there is some $c_1$ depending only on $t$ such that on the event $E_1$ we have
\begin{equation*}
    c_1 \lnb g_1 g_2 \dots g_i \rnb^{-2} \leq \left. \frac{\partial}{\partial y} \phi(g_1 g_2 \dots g_i y) \right|_{y = g_{i+1} \dots g_p b} \leq \lnb g_1 g_2 \dots g_i \rnb^{-2}.
\end{equation*}
Combining this with the first part of Condition \ref{condition:in_u_t} and the fact that for all $y$ we have $\lnb \rho_y \rnb \leq 1$ we see that on $A \cap E_1$ we have $$|\zeta_i(u^{(i)})| < r.$$ 

We also have that $$\var [\zeta_i(u^{(i)})|\hat{\mathcal{A}}] \geq c_1^{2} \lnb g_1 g_2 \dots g_i \rnb^{-4} \var [\rho_{g_{i+1} \dots g_p b}(u^{(i)})|\hat{\mathcal{A}}]. $$
By Proposition \ref{proposition:principal_component_u_t} there is some constant $c_2>0$ depending only on $t$ such that on the event $A_1$ we have $$\var [\rho_{g_{i+1} \dots g_p b}(u^{(i)})|\hat{\mathcal{A}}] \geq c_2 \var[u^{(i)} | \hat{\mathcal{A}}].$$

Now let $C_3$ be the $C$ from Corollary \ref{corollary:small_rvs_to_order_k_detail} with the same value for $\alpha$. Let $Q = C_3$. Let $E_2$ be the event that
$$\sum_{i=1}^{p} \frac{\var [u^{(i)} |\mathcal{A}]}{\lnb g_1 g_2 \dots g_i \rnb^4 r^2} > C_3 Q^2 c_1^{-2} c_2^{-1} k . $$
Note that on $A \cap E_1 \cap E_2$ by Corollary \ref{corollary:small_rvs_to_order_k_detail} we have $$s_{Qr}^{(k)}(S | \mathcal{A}) < \alpha^k$$ and so on $A \cap E_1 \cap E_2$ we have $$s_{Qr}(x|\mathcal{A}) < \alpha^k + C_2 C_1^n P^2 r.$$ To conclude we simply need to show that $E_2$ occurs with high probability.

Note that
$$\sum_{i=1}^{p} \frac{\var [u^{(i)} |\mathcal{A}]}{\lnb g_1 g_2 \dots g_i \rnb^4 r^2} = \sum_{i=1}^{n} \frac{\var[U_i|\mathcal{A}]}{\lnb f_1 h_1 f_2 h_2 \dots f_i \rnb^4 r^2}.$$
For $i = 1, \dots, n$ let $$X_i := \frac{\var[U_i|\mathcal{A}]}{\lnb f_1 h_1 f_2 h_2 \dots f_i \rnb^4 r^2}.$$

Note that by Condition \ref{condition:in_u_t} we have $X_i \leq 1$ and by Condition \ref{condition:average_variance_large} and the fact that each $X_i$ is $\mathcal{A}_i$-measurable we have $$\mathbb{E}[X_i|X_1, X_2, \dots, X_{i-1}] \geq m_i.$$
Let $c_3$ be the $c$ in Corollary \ref{corollary:cramer}. Note that by Corollary \ref{corollary:cramer} if we choose $C$ sufficiently large then 
$$\mathbb{P}[E_2] \geq 1 - \exp( - c_3 C k).$$
In particular if we take $C$ to be sufficiently large in terms of $\alpha$ then $$\mathbb{P}[E_2] \geq 1 - \alpha^k.$$
We now conclude by noting that 
\begin{align*}
    s_{Qr}^{(k)}(x) &\leq \mathbb{E}[s_{Qr}^{(k)}(x |\mathcal{A})]\\
    & \leq \alpha^k + C_2 C_1^n P^2 r + \mathbb{P}[A^C] + \mathbb{P}[E_1^C] + \mathbb{P}[E_2^C]\\
    & \leq 2 \alpha^k + C_2 C_1^n p^2 r + \varepsilon + \exp(-c K).
\end{align*}
The result follows by replacing $\alpha$ with a slightly smaller value.
\end{proof}

We now prove Lemma \ref{lemma:cramer}.

\begin{proof}[Proof of Lemma \ref{lemma:cramer}]
First note that by Jensen's inequality for any $\lambda \geq 0$ we have

\begin{equation}
\mathbb{E}[e^{-\lambda X_i} | X_1,\dots, X_{i-1}] \leq \left( 1 - \frac{m_i}{b} \right) + \frac{m_i}{b} e ^{-\lambda b}. \label{eq:better_than_hoeffding}
\end{equation}
Therefore we have
\begin{align}
\mathbb{E}[e^{- \lambda (X_1 + \dots +X_n) }] &\leq \prod_{i=1}^{n} \left( \left( 1 - \frac{m_i}{b} \right) + \frac{m_i}{b} e^{- \lambda b} \right) \nonumber\\
&\leq  \left( \left( 1 - \frac{a}{b} \right) + \frac{a}{b} e ^{-\lambda b}  \right)^n. \label{eq:by_am_gm}
\end{align}
with \eqref{eq:by_am_gm} following from the AM-GM inequality. Applying Markov's inequality for any $\lambda \geq 0$ we have
\begin{align}
\mathbb{P}(X_1 + \dots +X_n \leq nc) & \leq  e^{\lambda n c} \mathbb{E}[e^{- \lambda (X_1 + \dots +X_n) }]\nonumber\\
& \leq \left( e^{\lambda c} \left( \left( 1 - \frac{a}{b} \right) + \frac{a}{b} e ^{-\lambda b} \right) \right)^n.\label{eq:cramer_to_minimize}
\end{align}
We wish to substitute in the value of $\lambda$ which minimizes the right hand side of \eqref{eq:cramer_to_minimize}. It is easy to check by differentiation that this is $\lambda = - \frac{1}{b} \log \frac{c (b-a)}{a (b-c)}$. It is easy to see that this value of $\lambda$ is at least $0$ because $c \leq a$. Note that with this value of $\lambda$ we get $e^{- \lambda b} = \frac {c (b-a)}{a(b-c)}$ and $e ^ {\lambda c} = \left( \frac{c (b-a)}{a (b-c)} \right) ^ {-c/b}$. Hence
\begin{align*}
\left( 1 - \frac{a}{b} \right) + \frac{a}{b} e ^{-\lambda b} & = \left( 1 - \frac{a}{b} \right) + \frac{a}{b} \frac {c (b-a)}{a (b-c)}\\
& = \frac{(b-a)(b-c)}{b(b-c)} + \frac{c(b-a)}{b(b-c)}\\
&= \frac{b-a}{b-c}.
\end{align*}
The result follows.
\end{proof}
From this we deduce Corollary \ref{corollary:cramer}.
\begin{proof}[Proof of Corollary \ref{corollary:cramer}]
    Let $$f(a):= \log \left( 2^{a/2} \left( \frac{1-a}{1-\frac{a}{2}}\right)^{1 - a/2} \right).$$
    Note that by Lemma \ref{lemma:cramer} we have $$\log \mathbb{P}[X_1 + \dots + X_n \leq \frac{1}{2} a] \leq n f(a).$$
    We note that $$f(a) = \frac{a}{2} \log 2 + (1 - \frac{a}{2}) \log (1-a) - (1- \frac{a}{2}) \log (1 - \frac{a}{2})$$ and compute \begin{align*}
        f'(a) &= \frac{1}{2} \log 2 - \frac{1}{2} \log (1 - a) - \frac{1 - \frac{a}{2}}{1 - a} + \frac{1}{2} \log (1 - \frac{a}{2}) + \frac{1}{2}\\
        & = \frac{1}{2}\left(- \frac{1}{1-a} + \log(2 - a) - \log(1-a) \right)
    \end{align*} 
    and
    \begin{align*}
        f''(a) &= \frac{1}{2} \left( - \frac{1}{(1-a)^2} - \frac{1}{2-a} + \frac{1}{1-a} \right).
    \end{align*}
    In particular $f'(0) = -\frac{1}{2}(1 - \log 2) < 0$ and $f''(a) \leq 0$ for all $a \in [0,1)$. This proves the result for $c = \frac{1}{2}(1 - \log 2)$.
\end{proof}

\begin{remark}
We could deduce a result similar to Lemma \ref{lemma:cramer} from the Azuma–Hoeffding inequality. In our application of this result $a$ will be very small compared to $b$. In this regime the Azuma–Hoeffding inequality is inefficient for several reasons the most important of which is the inefficiency of Hoeffding's Lemma in this regime. Indeed using Hoeffding's Lemma to bound the left hand side of \eqref{eq:better_than_hoeffding} would lead to a bound of
\begin{equation*}
\exp \left( - \lambda m_i + \frac{\lambda^2 b^2}{8} \right).
\end{equation*}
When we apply the lemma we end up with $m_i$ being very small, $b = 1$, and $\lambda \approx \log 2$. Clearly this bound is weak when this occurs.
It turns out that the bound from Azuma-Hoeffding is not strong enough to prove Theorem \ref{theorem:main_furstenberg} in its current form but we could prove a similar result with  \eqref{eq:main_theorem_condition} replaced by
\begin{equation*}
\frac{h_{RW}}{ \chi} > C \left(\max \left\{1, \frac{\log M_{\mu}}{h_{RW}} \right\}\right)  \left( \max \left \{ 1 , \log \frac{\log M_{\mu}}{h_{RW}} \right \} \right)^{3}.
\end{equation*}
\end{remark}

\subsection{Proof of Proposition \ref{proposition:lots_of_v_varaince_sum}}

In this subsection we prove Proposition \ref{proposition:lots_of_v_varaince_sum}. First we need the following proposition.

\begin{proposition} \label{proposition:v_to_varaince_sum}
	For any $\alpha_0 \in (0, 1/3)$ and any $t, R > 0$ there exists some $c_1 = c_1(\alpha_0, t, R) > 0$ such that the following is true.  Let $\mu$ be a finitely supported Zariski-dense probability measure and suppose that $\mu$ is $\alpha_0, t$- non-degenerate and that the operator norm is bounded above by $R$ on the support of $\mu$. Then there is a constant $c_2= c_2(\mu)>0$ depending on $\mu$ such that the following holds. Let $\chi$ be the Lyaponuv exponent of $\mu$ and let $C, \delta>0$. Let $P, s > 0$ with $P$ sufficiently large (in terms of $\mu$, $C$ and $\delta$) and $s>0$ sufficiently small (in terms of $\mu$, $C$ and $\delta$). Let $K \in \Z_{>0}$ and suppose that $K$ is sufficiently large (in terms of $\mu$, $C$ and $\delta$). 

    Let $\hat{\nu}$ be as in Theorem \ref{theorem:renewal_theorem}, let $\gamma_1, \gamma_2, \dots$ be i.i.d.\ samples from $\mu$ and let $\tau_{P, y}$ be as in Definition \ref{definition:tau_t_v}. Let $$v = \int_y v(\gamma_1 \gamma_2 \dots \gamma_{\tau_{P, y}};s) \, \hat{\nu}(dy).$$
    Then for any $r \in (0, P^{-2} \exp(-4K\chi) s)$ we have
    \begin{equation*}
        V(\mu, 1, K, \exp(-2K\chi) P^{-1} \sqrt{s/r}, \exp(2K\chi) \sqrt{s/r}, t, C, \exp(- c_2 K); r) > c_1 v - \delta.
    \end{equation*}
\end{proposition}

\begin{proof}
	Suppose that $a, \overline{a} \in \Gp$ with $\lnb a \rnb \leq \exp(-2K\chi) P^{-1} \sqrt{s/r}$. We wish to construct a proper decomposition for $(\mu, 1, K, a, \overline{a}, t, C, \exp(-c_2K))$ at scale $s$. Let $\gamma_1, \gamma_2, \dots$ be i.i.d.\ samples from $\mu$. Let $\underline{S}$ be defined by
	\begin{equation*}
	\underline{S} := \inf \{ n : \lnb  a \gamma_1 \gamma_2 \dots \gamma_n \rnb \geq 8 P^{-1} \sqrt{s / r} \}.
	\end{equation*}
	We take $\varepsilon >0$ to be some small constant which depends on $\mu, \alpha_0, t, R$ and $\delta$ which we will choose later. Let $\hat{\nu}$ be as in Theorem \ref{theorem:renewal_theorem} and let $y$ be a sample from $\hat{\nu}$ such that
	\begin{equation*}
	\mathbb{P}[d(y, b^{-}(a \gamma_1 \gamma_2 \dots \gamma_{\underline{S}})^{\perp}) \geq \varepsilon] < \varepsilon
	\end{equation*}
	and $y$ is independent from $\gamma_{\underline{S}+1}, \gamma_{\underline{S}+2}, \dots$. This is possible by Corollary \ref{corollary:renewal_theorem}. Let $S_1$ be defined by
	\begin{equation*}
	S_1 := \inf \{ n \geq \underline{S} : \lnb (\gamma_{\underline{S}+1} \gamma_{\underline{S}+2} \dots \gamma_n)^T y \rnb \geq P\}.
	\end{equation*}
	Define
	\begin{equation*}
	\underline{f} := \gamma_1 \dots \gamma_{\underline{S}}
	\end{equation*}
	and define
	\begin{equation*}
	g := \gamma_{\underline{S}+1} \gamma_{\underline{S}+2} \dots \gamma_{S_1}.
	\end{equation*}
	By the definition of $v(\cdot ; \cdot)$ we can construct some $\sigma$-algebra $\hat{\mathcal{A}}$ which is conditionally independent of $\gamma_1, \gamma_2, \dots, \gamma_{\underline{S}}$ given $y$, some $\hat{\mathcal{A}}$-measurable random variable $\overline{f}$ taking values in $\Gp$ and some random variable $V$ taking values in $\A$ such that
	\begin{equation*}
	g = \overline{f} \exp(U),
	\end{equation*}
	\begin{equation*}
	\lnb V \rnb \leq r
	\end{equation*}
	and
	\begin{equation*}
	\mathbb{E} \left[ \var \left[ V | \hat{\mathcal{A}}, y \right] \right] \geq \frac{1}{2} v r^2.
	\end{equation*}
	We define $T_1$ by $T_1 := S_1 + K$ and define $h_1$ by
	\begin{equation*}
		h_1 = \gamma_{S_1+1} \gamma_{S_1+2} \dots \gamma_{T_1}.
	\end{equation*}
	
	We take $I$ to be $\{1\}$ if and only if the following conditions hold
	\begin{itemize}
	\item $d(y, b^{-}(a \underline{f})) < \varepsilon$
	\item $d(y, b^{+}(\overline{f})) > 100 \varepsilon$
    \item $b^{+}(h) \in U_{t/4, t/8}(V)$
    \item $d(b^{-}(\overline{f}), b^{+}(h_1)) > t/4$.
	\end{itemize}
   Otherwise we take $I = \emptyset$. Let $E_1$ be the event that $d(y, b^{-}(a \gamma_1 \gamma_2 \dots \gamma_n)) < \varepsilon$ and $d(y, b^{+}(\overline{f})) > 100 \varepsilon$ and let $E_2$ be the event that $b^{+}(h) \in U_{t/4, t/8}$ and $d(b^{-}(\overline{f}), b^{+}(h)) > t/4$. Clearly $\{1 \in I\} = E_1 \cap E_2$.

	We now define $U_1$ by
	\begin{equation*}
	U_1 = \begin{cases}
	V & \text{ if } I = \{1\}\\
	0 & \text{ if } I = \emptyset
	\end{cases}
	\end{equation*}
	and define $f_1$ by
	\begin{equation*}
	f_1 = \begin{cases}
	\underline{f} \overline{f} & \text{ if } I = \{1\}\\
	\underline{f} g & \text{ if } I = \emptyset.
	\end{cases}
	\end{equation*}

    We define $\mathcal{A}_1 := \sigma(f_1, h_1, a, \overline{a})$ and take $\mathcal{A}_0 := \sigma(a, \overline{a})$. Take $A$ to be the event that $\lnb a f h \rnb \leq \exp(2 K \chi) \sqrt{s/r}$, $\lnb\overline{a} f \rnb \geq C$. This is clearly $\mathcal{A}_1$ measurable and it is easy to see by applying \eqref{eq:large_deviations_estimate} from Lemma \ref{lemma:first_large_deviations_estimates} and \eqref{eq:products_not_too_close} from Lemma \ref{lemma:convergence_to_fm} that providing $P$ and $K$ are sufficiently large (depending on $\mu$) $\mathbb{P}[A] \geq 1 - \exp(-c_2 K)$ for some constant $c_2 >0$ depending only on $\mu$.
    
    We wish to show that we can choose $m_1 \geq \Theta_{\alpha_0, t, R}(v) - \delta$ and construct some filtration $\mathcal{F}=\left(\mathcal{F}_i\right)_{i=1}^{\infty}$ such that if we take $f = \left( f_i \right)_{i=1}^{1}$, define $h, U, m, S$ and $T$ similarly and take $\mathcal{A} : = \left(\mathcal{A}_i \right)_{i=0}^{1}$ then
    $$(f, h, U, m, \mathcal{A}, I, \gamma, \mathcal{F}, S, T, h_1)$$ is a proper decomposition for $(\mu, 1, a, \overline{a}, t, C, \exp(-c_2K))$ at scale $s$.
	
	Conditions \ref{condition:increasing_s_t}, \ref{condition:shape_of_f}, \ref{condition:shape_of_h}, \ref{condition:As_nested}, \ref{condition:u_i_conditional_indep_of_a_i}, \ref{condition:u_i_conditional_indep_of_u_i}, \ref{condition:f_and_h_a_measurable}, \ref{condition:k_gap} and \ref{condition:l_def} follow immediately from our construction. 
    Providing $\varepsilon$ is sufficiently small on $E_1$ we have
    \begin{align*}
    \lnb a \underline{f} \overline{f} \rnb & \geq \frac{1}{2} \lnb a \underline{f} \rnb \cdot \lnb \overline{f} \rnb \sin d(b^{-1}(a \underline{f}), b^{+}(\overline{f})) \\
    & \geq \frac{1}{4} \lnb a \underline{f} \rnb \cdot \lnb \overline{f} \rnb \cos d(y, b^{+}(\overline{f}))\\
    &= \frac{1}{4} \lnb a \underline{f} \rnb \cdot \lnb \overline{f}^T y \rnb\\
    & \geq \frac{1}{8} \lnb a \underline{f} \rnb \cdot \lnb g^T y \rnb\\
    & \geq \frac{1}{8}(8 \sqrt{s/r} P^{-1})\cdot P\\
    & = \sqrt{s/r}.
	\end{align*}
	In particular this means that $\lnb U_1 \rnb \leq \lnb a f_1 \rnb^2 r$. This together with the definition of $I$ shows that Condition \ref{condition:in_u_t} is satisfied. Condition \ref{condition:breaks_up_large} follows from our definition of $A$ and Condition \ref{condition:breaks_up_alligned} follows from our definition of $I$.
    
    We now show that Condition \ref{condition:average_variance_large} is satisfied. To do this we bound $\mathbb{E}[\frac{\var [U | \mathcal{A}]}{\lnb af_1 \rnb^{4} r^2}]$ from below.

    By Lemma \ref{lemma:simple_sl2_prod_size_bound} we know that providing $P$ and $K$ are sufficiently large and $\varepsilon$ and $r$ are sufficiently small whenever we have $1 \in I$ we have
    \begin{align}
        \lnb a\underline{f} \overline{f} \rnb & \leq  2 \lnb a\underline{f} \rnb \cdot \lnb \overline{f} \rnb \sin d(b^{-}(\underline{f}), b^{+}(\overline{f})) \nonumber\\
        & \leq 4 \lnb a\underline{f} \rnb \cdot \lnb \overline{f} \rnb \sin d(y, b^{+}(\overline{f})) \nonumber \\
        & = 4 \lnb a\underline{f} \rnb \cdot \lnb \overline{f}^T y \rnb \nonumber\\
        & \leq 8 \lnb a \underline{f} \rnb \cdot \lnb g^T y \rnb \nonumber\\
        & \leq 8 \cdot (R 8 P^{-1} \sqrt{s / r}) \cdot (R P)\nonumber\\
        & \leq 64 R^2 \sqrt{s / r}. \label{eq:ff_size_upperbound}
    \end{align}

    Clearly $\var [U | \mathcal{A}_1] = \var [V | \mathcal{A}_1] \mathbb{I}_{E_2} -  \var [V | \mathcal{A}] \mathbb{I}_{E_2} \mathbb{I}_{E_1^C}$. We know that $\var [V | \mathcal{A}]$ is $\mathcal{A}$- measurable and at most $s^2$. It is also clear from \eqref{eq:c_exponential_convergence_to_fm} from Lemma \ref{lemma:convergence_to_fm} and the definition of $\alpha_0, t$ - non-degeneracy that
    \begin{equation*}
        \mathbb{P}[E_2 | \mathcal{A}_1] \geq (1 - 3 \alpha_0)
    \end{equation*}
    almost surely. We also know by \eqref{eq:products_not_too_close} from Lemma \ref{lemma:convergence_to_fm} that
    \begin{equation*}
        \mathbb{P}[E_1^C] \leq \delta
    \end{equation*}
    for some $\delta = \delta(\varepsilon)$ such that $\delta \to 0$ as $\varepsilon \to 0$. In particular this means that
    \begin{align*}
        \mathbb{E}[\var[U|\mathcal{A}_1]] & \geq (1 - 3 \alpha_0) \mathbb{E}[\var[V | \mathcal{A}_1]] - \delta s^2\\
        & \geq \frac{1}{2}(1 - 3 \alpha_0) v s^2 - \delta s^2.
    \end{align*}
    Combining this with our estimate \eqref{eq:ff_size_upperbound} we see that there is some constant $c_1 >0$ depending only on $R$ and $\alpha_0$ such that
    \begin{equation*}
        \mathbb{E}[\frac{\var [U | \mathcal{A}]}{\lnb f \rnb^{4} r^2}] \geq c_1 v - \delta.
    \end{equation*}
    We take $m_1 = \max \{ c_1 v - \delta, 0\}$. 

    Finally we construct our $\mathcal{F}_i$. Suppose that $\xi$ is the set of events in our underlying probability space and define $\left(\mathcal{F}_i\right)_{i=1}^{\infty}$ by
    \begin{align*}
        \MoveEqLeft \mathcal{F}_i := \{F \in \xi : F \cap \{i < \underline{S}\} \in \sigma(\gamma_1, \gamma_2, \dots, \gamma_i),\\ &F \cap \{ \underline{S} \leq i < S \} \in \sigma(\gamma_1, \gamma_2, \dots, \gamma_i, y), F \cap \{ i \geq S \} \in \sigma(\gamma_1, \gamma_2, \dots, \gamma_i, y, \hat{\mathcal{A}}) \}.
    \end{align*}
    
    Applying Lemma \ref{lemma:generates_independent_filtration} twice shows that this is a filtration for the $\gamma_i$ and that $\gamma_{i+1}$ is independent from $\mathcal{F}_i$.
    
    This means that $$(f, h, U, m, \mathcal{A}_1, I, \gamma, \mathcal{F}, S, T, h)$$ is a proper decomposition for $(\mu, 1, a, \overline{a}, t, C, \exp(-c_2K))$ at scale $r$. By the definition of $V(\cdot)$ this means that
    \begin{equation*}
        V(\mu, 1, \chi^{-K} P^{-1} \sqrt{s/r}, \chi^{K} \sqrt{s/r}, t, C, \exp(- c_2 K); r) > c_1 v - \delta
    \end{equation*}
    as required.
\end{proof}

We can combine this result with Proposition \ref{proposition:lots_of_v} to prove Proposition \ref{proposition:lots_of_v_varaince_sum}.

\begin{proof}[Proof of Proposition \ref{proposition:lots_of_v_varaince_sum}]
    Let $s_1, s_2, \dots, s_{\hat{m}}$ be as in Proposition \ref{proposition:lots_of_v} and let $$v_i := \int v(\gamma_1 \gamma_2 \dots \gamma_{\tau_{P, y}};s_i) \, dy.$$ By Proposition \ref{proposition:v_to_varaince_sum} we know that there is some constant $c_1>0$ depending only on $R, \alpha_0$ and $t$ and some constant $c_2>0$ depending only on $\mu$ such that for every $\delta >0$ providing $P$ and $K$ are sufficiently large in terms of $\delta$, $\mu$ and $C$ we have
    \begin{equation*}
        V(\mu, 1, \chi^{-K} P^{-1} \sqrt{s_i/r}, \chi^{K} \sqrt{s_i/r}, t, C, \exp(- c_2 K); r) > c_1 v_i - \delta.
    \end{equation*}
    In particular providing $P$ is sufficiently large depending on $\mu$, $\delta$ and $C$ we have
    \begin{align*}
        \MoveEqLeft \sum_{i=1}^{\hat{m}} V(\mu, 1, \exp(-2\chi K) P^{-1} \sqrt{s_i/r}, \exp(2\chi K) \sqrt{s_i/r}, t, C, \exp(- c_1 K); r) \\ &> c_3 \left( \frac{h_{RW}}{\chi} \right)  \left( \max \left\{ 1,  \log \frac{ \log M}{h_{RW}} \right\} \right)^{-1} - \hat{m}\delta
    \end{align*}
    for some constant $c_3$ depending only on $R, \alpha_0$ and $t$. We now note that for $i = 1, \dots, \hat{m}-1$ we have
    \begin{align*}
        \exp(2\chi K) \sqrt{s_i/r} & \leq \exp(2\chi K) \sqrt{P^{-3}s_{i+1}/r}\\
        & = P^{-3/2} \exp(2\chi K) \sqrt {s_{i+1}/r}\\
        & \leq P^{-1} \exp(-2\chi K) \sqrt {s_{i+1}/r}.
    \end{align*}
    Letting $\delta = \frac{c_3}{2 \hat{m}}$ and applying Proposition \ref{proposition:varaince_sum_adds} we see that
    \begin{align*}
       \MoveEqLeft V(\mu, \hat{m}, K, P^{-\frac{\log M}{2 \chi}-2} r^{-1/2}, P^{-\frac{h_{RW}}{20 \chi}+1} r^{- 1/2}, t, C, \hat{m} \exp(-c_1 K); r)\\ &\geq \frac{c_3}{2} \left( \frac{h_{RW}}{\chi} \right)  \left( \max \left\{ 1,  \log \frac{ \log M_{\mu}}{h_{RW}} \right\} \right)^{-1}.
    \end{align*}
\end{proof}

\section{Proof of main theorem} \label{section:proof_of_main_theorem}

We now have all the tools required to prove Theorem \ref{theorem:main_furstenberg}. First we will prove the following.

\begin{proposition} \label{proposition:order_k_detail_small_bound}
    For all $\alpha_0 \in (0, 1/3)$ and every $t, R>0$ there exists some constant $C>0$ such that the following is true. Suppose that $\mu$ is a finitely supported Zariski-dense probability measure. Suppose that $\mu$ is $\alpha_0, t$-non-degenerate and that the operator norm is bounded above by $R$ on the support of $\mu$. Let $h_{RW}$ be its random walk entropy, let $\chi$ be its Lyapunov exponent and let $M_{\mu}$ be its splitting rate. Suppose that
    \begin{equation*}
        \frac{h_{RW}}{\chi} > C \left(\max \left\{1, \log \frac{\log M_{\mu}}{h_{RW}} \right\} \right)^2.
    \end{equation*}
    Then for all sufficiently small (in terms of $\mu, R, \alpha_0$ and $t$) $r>0$ and all $k \in [\log \log r^{-1}, 2 \log \log r^{-1}] \cap \Z$ we have $$s_r^{(k)}(\nu) < \left( \log r^{-1} \right)^{-10}.$$
\end{proposition}

\begin{proof}
    
    Let $C_1$ be the $C$ from Proposition \ref{proposition:variance_sum_to_order_k_detail} with $\exp(-11)$ in the role of $\alpha$ and $t$ in the role of $t$. Note that by Proposition \ref{proposition:variance_sum_to_order_k_detail} it is sufficient to show that there is some constant $c_1 = c_1(\mu) > 0$ and some constant $A_1 = A_1(\mu, R, \alpha_0, t)>0$ such that for all sufficiently small $r>0$ we can find some $n < A_1 \log \log r^{-1}$ such that if we let $K = \exp(\sqrt{\log \log r^{-1}})$ then
    \begin{equation}
        V(\mu, n, K, 1, r^{-1/2} \exp(-c_1K), t, C_1, \exp(-c_1K); r) > 2 C_1 \log \log r^{-1}. \label{eq:required_varaince_sum}
    \end{equation}
    Indeed when this occurs by Proposition \ref{proposition:variance_sum_to_order_k_detail} for all $k \in [\log \log r^{-1}, 2 \log \log r^{-1}] \cap \Z$ we have
    \begin{align*}
        \MoveEqLeft s_{Qr}^{(k)}(\nu) < \exp(- 11 k) + A_1 \log \log r^{-1} \exp(- c_2 K) \\&+  C_1^{A_1 \log \log r^{-1}}\exp(-c_1K) + \exp(-c_1K)
    \end{align*}
    for some constant $c_2 >0$ depending only on $\mu$. Clearly this is less than $\left( \log (Qr)^{-1} \right)^{-10}$ whenever $r$ is sufficiently small.

    We will prove \eqref{eq:required_varaince_sum} by repeatedly applying Proposition \ref{proposition:lots_of_v_varaince_sum} and Proposition \ref{proposition:varaince_sum_adds}. Given $r$ we wish to construct some $m \in \Z_{>0}$ and some decreasing sequence $\left(P_i\right)_{i=1}^{m}$ such that for each $i = 1, 2, \dots, m$ we can apply Proposition \ref{proposition:lots_of_v_varaince_sum} with $P_i$ in the role of $P$ and then apply Proposition \ref{proposition:varaince_sum_adds} to the resulting bounds on the variance sums.

    First we let $P_1  = r^{- \frac{\chi}{2 \log M}}$ and inductively we take $P_{i+1} = P_i^{\frac{h_{RW}}{40 \log M}}$. Note that this gives
    \begin{equation*}
        P_i = \exp \left( \frac{\chi \log r^{-1}}{2 \log M} \left( \frac{h_{RW}}{40 \log M} \right)^{i - 1} \right).
    \end{equation*}
    We then choose $m$ as large as possible so that we may ensure that $P_m \geq \exp((\max \{1, 10 \chi \})K)$. Note that this means
    \begin{align*}
        m = \left \lfloor \frac{\log \frac{\chi \log r^{-1}}{2 (\max \{1, 10 \chi \}) K \log M}}{\log \frac{40 \log M}{h_{RW}}} \right \rfloor + 1.
    \end{align*}
    In particular there is some absolute constant $c_3 > 0$ such that for all sufficiently small (depending on $\mu$) $r> 0$ we have
    \begin{equation*}
        m \geq c_3 \left( \max \left\{ 1,  \log \frac{ \log M_{\mu}}{h_{RW}} \right\} \right)^{-1} \log \log r^{-1}
    \end{equation*}
    and $m \leq O_{\mu}(\log \log r^{-1})$.

    Note that our construction of the $P_i$ gives
    \begin{equation*}
        P_{i+1}^{- \frac{\log M}{\chi}} r^{- 1/2} \geq P_{i}^{- \frac{h_{RW}}{40 \chi}} r^{-1/2}
    \end{equation*}
    and so applying Proposition \ref{proposition:v_to_varaince_sum} and Proposition \ref{proposition:varaince_sum_adds} repeatedly we get
    \begin{align*}
        \MoveEqLeft V(\mu, m \hat{m}, K, P_1^{-\frac{\log M}{\chi}} r^{-1/2}, P_m^{-\frac{h_{RW}}{40 \chi}} r^{-1/2}, t, C, m \exp(- c_1 K); r)\\&>c_4 \frac{h_{RW}}{\chi} \left( \max \left\{ 1,  \log \frac{ \log M_{\mu}}{h_{RW}} \right\} \right)^{-1} \log \log r^{-2}.
    \end{align*}
    By Proposition \ref{proposition:variance_sum_to_order_k_detail} this is enough to complete the proof.
\end{proof}

We will now prove Theorem \ref{theorem:main_furstenberg}.

\begin{proof}[Proof of Theorem \ref{theorem:main_furstenberg}]
    We will prove this by combining Proposition \ref{proposition:order_k_detail_small_bound} with Lemma \ref{lemma:ind_srk_to_sr} to get an upper bound on $s_r(\nu)$ for all sufficiently small $r$. We will then conclude using Lemma \ref{lemma:suff_abs_cont}.

    Given $r>0$ sufficiently small let $k = \frac{3}{2} \log \log r^{-1}$, let $a = r / \sqrt{k}$, let $b = r \exp(k \log k)$ and let $\alpha = (\log r^{-1})^{-10}$. We wish to apply Lemma \ref{lemma:ind_srk_to_sr} with this choice of $a, b$ and $\alpha$.

    Suppose that $s\in [a, b]$. It follows by a simple computation that $k \in [\log \log s^{-1}, 2 \log \log s^{-1}]$ and so by Proposition \ref{proposition:order_k_detail_small_bound} providing $r$ is sufficiently small we have $$s_s^{(k)}(\nu) < \alpha.$$
    By Lemma \ref{lemma:ind_srk_to_sr} this means that
    \begin{align*}
        s_r(\nu) & \leq \left( \log r^{-1} \right)^{-10} \left( \frac{ 2 e}{\pi} \right)^{\frac{k-1}{2}} + k! \cdot k a^2 b^{-2}.
    \end{align*}
    We then compute
    \begin{align*}
        \left( \log r^{-1} \right)^{-10} \left( \frac{ 2 e}{\pi} \right)^{\frac{k-1}{2}} + k! \cdot k a^2 b^{-2} &\leq \left( \log r^{-1} \right)^{-10} e^{k/2} + k ^{-k}.
    \end{align*}
    Clearly this is less than $\left( \log r^{-1} \right)^{-2}$ providing $r$ is sufficiently small. By Lemma \ref{lemma:suff_abs_cont} we have that $\nu$ is absolutely continuous.
\end{proof}

\section{Examples} \label{section:examples}

In this section we will give examples of measures $\mu$ on $\Gp$ which satisfy the conditions of Theorem \ref{theorem:main_furstenberg}.

\subsection{Heights and separation}

In this subsection we will review some techniques for bounding $M_{\mu}$ using heights. First we need the following definition.

\begin{definition}[Height]
Let $\alpha_1$ be algebraic with algebraic conjugates $\alpha_2, \alpha_3, \dots, \alpha_d$. Suppose that the minimal polynomial for $\alpha_1$ over $\Z[X]$ has positive leading coefficient $a_0$. Then we define the \emph{height} of $\alpha_1$ by
\begin{equation*}
\mathcal{H}(\alpha_1) := \left( a_0 \prod_{i=1}^{n} \max \{ 1, |\alpha_i| \} \right)^{1/d}.
\end{equation*}
\end{definition}

We wish to use this to bound the size of polynomials of algebraic numbers. To do this we need the following way of measuring the complexity of a polynomial.

\begin{definition}
Given some polynomial $P \in \Z [X_1, X_2, \dots, X_n]$ we define the \emph{length} of $P$, which we denote by $\mathcal{L}(P)$, to be the sum of the absolute values of the coefficients of $P$.
\end{definition}
We also need the following basic fact about heights.
\begin{lemma} \label{lemma:height_reciprical}
Let $\alpha \neq 0$ be an algebraic number. Then
\begin{equation*}
\mathcal{H}(\alpha^{-1}) = \mathcal{H}(\alpha).
\end{equation*}
\end{lemma}
\begin{proof}
This follows easily from the definition and is proven in \cite[Section 14]{MASSER_2016}.
\end{proof}

\begin{lemma} \label{lemma:height_poly_bound}
Given $P \in \Z[X_1, X_2, \dots, X_n]$ of degree at most $L_1 \geq 0$ in $X_1$, $\dots$, $L_n \geq 0$ in $X_n$ and algebraic numbers $\xi_1, \xi_2, \dots, \xi_n$ we have
\begin{equation*}
\mathcal{H}(P(\xi_1, \xi_2, \dots, \xi_n)) \leq \mathcal{L}(P) \mathcal{H}(\xi_1)^{L_1} \dots \mathcal{H}(\xi_n)^{L_n}
\end{equation*}
\end{lemma}

\begin{proof}
This is \cite[Proposition 14.7]{MASSER_2016}.
\end{proof}
To make the above lemma useful for bounding the absolute value of expressions we need the following.
\begin{lemma}
Suppose that $\alpha \in \mathbb{C} \backslash \{ 0 \}$ is algebraic and that its minimal polynomial has degree $d$. Then
\begin{equation*}
\mathcal{H}(\alpha)^{-d} \leq |\alpha| \leq \mathcal{H}(\alpha)^d.
\end{equation*}
\end{lemma}

\begin{proof}
The fact that $|\alpha| \leq \mathcal{H}(\alpha)^d$ is immediate from the definition of height. The other side of the inequality follows from Lemma \ref{lemma:height_reciprical}.
\end{proof}

\begin{proposition} \label{proposition:m_mu_bound}
Suppose that $\mu$ is a measure on $\Gp$ supported on a finite set of points. For each element in the support of $\mu$ choose a representative in $\Sl_2(\R)$. Let $S \subset \Sl_2(\R)$ be the set of these representatives.

Suppose that all the entries of the elements of $S$ are algebraic. Let $(\xi_1, \xi_2, \dots, \xi_k)$ be the set of these entries. Let $K = \mathbb{Q}[\xi_1, \xi_2, \dots, \xi_k]$ be the number field generated by the $\xi_i$ and let
\begin{equation*}
C = \max \{  \mathcal{H}(\xi_i) : i \in [k] \}.
\end{equation*}
Then
\begin{equation*}
M_{\mu} \leq 4^{[K:\mathbb{Q}]} C^{8 [K:\mathbb{Q}]}.
\end{equation*}
\end{proposition}

\begin{proof}

Let $a \in S^m$ and $b \in S^n$. We find an upper bound for $d(a, b)$ where $d$ is the distance function of our left-invariant Riemannian metric introduced in the introduction. We have that
\begin{equation*}
d(a, b) = d(\id, a^{-1} b) \geq \Theta \left(\min \left\{ \lnb I - a^{-1} b \rnb_2, \lnb I + a^{-1} b \rnb_2 \right\} \right).
\end{equation*}

For $i \in [|S|]$ and $j, k \in \{1, 2\}$ let $\zeta_{i, j, k}$ be the $(j, k)$-th entry of the $i$-th element of $S$. Let $L_i$ be the sum of the number of times the $i$-th element of $S$ appears in our word for $a$ and the number of times it appears in our word for $b$. Note that the components of $a^{-1}$ are components of $a$ possibly with a sign change. We know that each component of $I \pm a^{-1} b $ is of the form $P(\zeta_{1, 1, 1}, \dots, \zeta_{|S|, 2, 2})$ where $P$ is some polynomial of degree at most $L_i$ in $\zeta_{i, j, k}$. We also know that the $L_i$ sum to $m+n$.

It is easy to see by induction that $\mathcal{L}(P) \leq 2^{m+n} + 1$. In particular $\mathcal{L}(P) \leq 2^{m+n+1}$. By Lemma \ref{lemma:height_poly_bound} this means that if $\alpha$ is a coefficient of $I \pm a^{-1}b$ then
\begin{equation*}
\mathcal{H}(\alpha) \leq 2^{m+n+1}  C^{4 (m+n)}.
\end{equation*}

We know that $\alpha \in K$ and so in particular the degree of its minimal polynomial is at most $[K: \mathbb{Q}]$. This means that if $\alpha \neq 0$ then
\begin{equation*}
|\alpha| \geq 2^{-(m+n+1)  [K: \mathbb{Q}] } C^{- 4(m+n)  [K: \mathbb{Q}]}.
\end{equation*}
In particular this means that if $a \neq b$ then
\begin{equation*}
d(a, b) \geq \Theta \left( 2^{-(m+n + 1) [K: \mathbb{Q}] } C^{- 4 (m+n)  [K: \mathbb{Q}]} \right)
\end{equation*}
and so
\begin{equation*}
M_{\mu} \leq 4^{ [K:\mathbb{Q}]} C^{8 [K:\mathbb{Q}]}. \qedhere
\end{equation*}
\end{proof}

\subsection{Bounding the random walk entropy using the Strong Tits alternative}

In this subsection we will combine Breuillard's strong Tits alternative \cite{BREUILLARD_2008} with the results of Kesten \cite{KESTEN_1959} in order to obtain an estimate on the random walk entropy. The main result of this section will be the following.

\begin{proposition} \label{proposition:random_walk_entropy_lower_bound}
There is some $c >0$ such that the following is true. Let $\mu$ be a finitely supported probability measure on $\Gp$ and let $h_{RW}$ be its random walk entropy. Let $K > 0$ and suppose that for every virtually solvable subgroup $H < \Gp$ we have
\begin{equation*}
\mu(H) < 1 - K.
\end{equation*}
Suppose further that $\mu(\id) > K$. Then
\begin{equation*}
h_{RW} > c K.
\end{equation*}
\end{proposition}

A similar result which further requires $\mu$ to be symmetric is discussed in \cite[Chapter 7]{Sert_2017}. In \cite{Sert_2017} much of the proof of their result is done by citing unpublished lecture notes so we give a full proof of Proposition \ref{proposition:random_walk_entropy_lower_bound} here.

$\Gp$ acts on the closed complex half plane $\overline{\mathbb{H}} = \{ z \in \mathbb{C} : \im z \geq 0 \}$ by M\"{o}bius transformations. It is well known that the virtually solvable subgroups of $\Gp$ are precisely those which either have a common fixed point in $\overline{\mathbb{H}}$ or for which there exists a pair of points in $\overline{\mathbb{H}}$ such that each element in the subgroup either fixes both points or maps them both to each other.

To prove Proposition \ref{proposition:random_walk_entropy_lower_bound} we introduce the following. We let $G$ be a countable group and let $\mu$ be a finite measure on $G$. We let $T_{\mu, G} : l^2(G) \to l^2(G)$ be the operator defined by $T_{\mu, G}(f)(g) = \int_{G} f(gh) d \mu(h)$. It is clear that $T_{\mu, G}$ is a bounded linear operator and that when $\mu$ is symmetric $T_{\mu, G}$ is self-adjoint. To prove Proposition \ref{proposition:random_walk_entropy_lower_bound} we need the following results.

\begin{lemma} \label{lemma:t_op_linear}
The operator $T_{\mu, G}$ is linear in $\mu$. In other words
\begin{equation*}
T_{\lambda_1 \mu_1 + \lambda_2 \mu_2, G} = \lambda_1 T_{\mu_1, G} + \lambda_2 T_{\mu_2, G}.
\end{equation*}
\end{lemma}
This lemma is trivial and its proof is left to the reader.

\begin{lemma} \label{lemma:random_walk_bound_op_norm}
Let $\mu$ be a finitely supported probability measure on some group $G$. Let $h_{RW}$ be the random walk entropy of $\mu$. Then
\begin{equation*}
h_{RW} \geq - 2\log \lnb T_{\mu, G} \rnb.
\end{equation*}
\end{lemma}
This lemma is proven by Avez in \cite[Theorem IV.5]{Avez_1976}.

\begin{lemma} \label{lemma:norm_less_than_one}
There is some $\varepsilon > 0$ such that the following is true. Suppose that $a, b, c \in \Gp$ generate a non-virtually solvable subgroup. Let $G$ be the group generated by $a$, $b$, and $c$. Let
\begin{equation*}
\mu = \frac{1}{4} \delta_a + \frac{1}{4} \delta_b + \frac{1}{4} \delta_c + \frac{1}{4} \delta_{\id}.
\end{equation*}
Then
\begin{equation*}
\lnb T_{\mu, G} \rnb < 1 - \varepsilon.
\end{equation*}
\end{lemma}

\begin{lemma} \label{lemma:decomposes_correctly}
Let $\lambda$ be a finite non-negative measure on $\Gp$ with finite support. Let $T$ be the total mass of $\lambda$. Let $K \geq 0$ and suppose that for every virtually solvable subgroup $H < \Gp$ we have
\begin{equation}
\lambda(H) < T - K. \label{eq:decompose_condition}
\end{equation}
Then there exists some $n \in \mathbb{Z}_{\geq 0}$ such that for each integer $i \in [1, n]$ there exists $a_i, b_i, c_i \in \Gp$ and $k_i >0$ such that
\begin{equation*}
\lambda = \lambda' + \sum_{i=1}^{n} k_i \left( \frac{1}{3} \delta_{a_i} + \frac{1}{3} \delta_{b_i} + \frac{1}{3} \delta_{c_i} \right)
\end{equation*}
for some non-negative measure $\lambda'$ and for each integer $i \in [1, n]$ the set $\{a_i, b_i, c_i \}$ generates a non-virtually solvable group. Furthermore the sum of the $k_i$ is at least $K$.
\end{lemma}

Proposition \ref{proposition:random_walk_entropy_lower_bound} follows immediately by combining these lemmas. The rest of this subsection will be concerned with proving Lemma \ref{lemma:norm_less_than_one} and Lemma \ref{lemma:decomposes_correctly}.

First we will prove Lemma \ref{lemma:norm_less_than_one}. A proof of a similar result for symmetric measures may be found in \cite{Breuillard_2015}. The key ingredient is the following result of Breuillard.

\begin{theorem} \label{theorem:strong_tits}
There exists some $N \in \N$ such that if $F$ is a finite symmetric subset of $\Gp$ containing $\id$, either $F^N$ contains two elements which freely generate a non-abelian free group, or the group generated by $F$ is virtually solvable (i.e. contains a finite index solvable subgroup).
\end{theorem}
\begin{proof}
This is a special case of \cite[Theorem 1.1]{BREUILLARD_2008}.
\end{proof}

We also need the following result of Kesten and a corollary of it.

\begin{theorem} \label{theorem:kesten_symmetric}
Let $G$ be a countable group. Suppose that $a, b \in G$ freely generate a free group. Let $A < G$ be the subgroup generated by $a$ and $b$. Let $\mu$ be the measure on $A$ given by
\begin{equation*}
\mu = \frac{1}{4} \left( \delta_a + \delta_{a^{-1}} + \delta_b + \delta_{b^{-1}} \right).
\end{equation*}
Then $\lnb T_{\mu, A} \rnb = \frac{\sqrt{3}}{2}$.
\end{theorem}
\begin{proof}
This follows from \cite[Theorem 3]{KESTEN_1959} and the fact that the spectral radius of a self-adjoint operator is its norm.
\end{proof}

\begin{corollary} \label{corollary:kesten_symmetric}
Let $G$ be a countable group. Suppose that $a, b \in G$ freely generate a free group. Let $A < G$ be the subgroup generated by $a$ and $b$. Let $\mu$ be the measure on $G$ given by
\begin{equation*}
\mu = \frac{1}{4} \left( \delta_a + \delta_{a^{-1}} + \delta_b + \delta_{b^{-1}} \right).
\end{equation*}
Then $\lnb T_{\mu, G} \rnb = \frac{\sqrt{3}}{2}$.
\end{corollary}

\begin{proof}
Let $H \subset G$ be chosen such that each left coset of $A$ in $G$ can be written uniquely as $h A$ for some $h \in H$. This means that
\begin{equation*}
l^2(G) \cong \bigoplus_{h \in H} l^2(hA).
\end{equation*}
We also note that for any $h \in H$ the map $T_{\mu, G}$ maps $l^2(hA)$ to $l^2(hA)$ and its action on $l^2(hA)$ is isomorphic to the action of $T_{\mu|_A, A}$ on $l^2(A)$. This means that $\lnb T_{\mu, G} \rnb = \lnb T_{\mu|_A, A} \rnb$. The result now follows by Theorem \ref{theorem:kesten_symmetric}.
\end{proof}

One difficulty we need to overcome is that Theorems \ref{theorem:strong_tits} and \ref{theorem:kesten_symmetric} require symmetric sets and measures but symmetry is not a requirement of Proposition \ref{proposition:random_walk_entropy_lower_bound}. We will do this by bounding $\lnb T_{\mu, G} T^{\dag}_{\mu, G} \rnb$. First we need the following two simple lemmas.

\begin{lemma} \label{lemma:convolve_product_t}
Let $G$ be a countable group and let $\mu_1, \mu_2$ be measures on $G$. Then
\begin{equation}
T_{\mu_1, G} T_{\mu_2, G} = T_{\mu_1 * \mu_2, G}.
\end{equation}
\end{lemma}

\begin{lemma} \label{lemma:inverse_hermitian}
Let $G$ be a group, let $n \in \N$, and let $\left( p_i \right)_{i=1}^n$ be a probability vector. Let $g_1, g_2, \dots, g_n \in G$ and let $\mu$ be defined by
\begin{equation*}
\mu = \sum_{i=1}^n p_i g_i
\end{equation*}
and let $\hat{\mu}$ be defined by
\begin{equation*}
\hat{\mu} = \sum_{i=1}^n p_i g_i^{-1}.
\end{equation*}
Then
\begin{equation*}
T_{\mu, G}^{\dag} = T_{\hat{\mu}, G}.
\end{equation*}
\end{lemma}

These lemmas are trivial and their proofs are left to the reader.

We are now ready to prove Lemma \ref{lemma:norm_less_than_one}.

\begin{proof}[Proof of Lemma \ref{lemma:norm_less_than_one}]
We will prove this by bounding $\lnb (T_{\mu, G} T_{\mu, G}^{\dag} )^N \rnb$ where $N$ is as in Theorem \ref{theorem:strong_tits}. Note that this is equal to $\lnb T_{\mu, G} \rnb^{2N}$.

Let $\hat{\mu}$ be as in Lemma \ref{lemma:inverse_hermitian}. Note that we may write
\begin{equation*}
\mu * \hat{\mu}  = \eta + \frac{1}{16}(\delta_{\id} + \delta_{a} + \delta_{a^{-1}} + \delta_b + \delta_{b^{-1}} + \delta_{c} + \delta_{c^{-1}})
\end{equation*}
where $\eta$ is some positive measure of total mass $\frac{9}{16}$.

By applying Theorem \ref{theorem:strong_tits} with $F = \{\id, a, a^{-1}, b, b^{-1}, c, c^{-1} \}$ we know that there is some $f, g \in F^N$ which freely generate a free group. We write
\begin{equation*}
(\mu * \hat{\mu} )^{*N} = \eta' + \frac{1}{16^N}(\delta_f + \delta_{f^{-1}} + \delta_g + \delta_{g^{-1}})
\end{equation*}
where $\eta'$ is some positive measure with total mass $1 - \frac{4}{16^N}$.

By Theorem \ref{theorem:kesten_symmetric} and Lemma \ref{lemma:t_op_linear} we know that
\begin{equation*}
\lnb T_{\frac{1}{16^N}(\delta_c + \delta_{c^{-1}} + \delta_d + \delta_{d^{-1}}), G} \rnb \leq  \frac{2 \sqrt{3}}{16^N}.
\end{equation*}
Therefore
\begin{equation*}
\lnb T_{(\mu * \hat{\mu} )^{*N}, G} \rnb \leq 1 - \frac{4}{16^N} (1 - \frac{\sqrt{3}}{2})
\end{equation*}
and therefore
\begin{equation*}
\lnb T_{\mu , G} \rnb \leq \left( 1 - \frac{4}{16^N} (1 - \frac{\sqrt{3}}{2}) \right)^{1/2N} < 1. \qedhere
\end{equation*}
\end{proof}

Finally we need to prove Lemma \ref{lemma:decomposes_correctly}.

\begin{proof}[Proof of Lemma \ref{lemma:decomposes_correctly}]
We prove this by induction on the number of elements in the support of $\lambda$. If $\lambda$ is the zero measure then the statement is trivial so we have our base case. If $K=0$ then the statement is trivial so assume $K > 0$ . Let $a \in \supp \lambda$ be chosen such that $\lambda (a)$ is minimal amongst all non-identity elements in the support of $\lambda$.

Now choose some $b \in \supp \lambda$ such that $a$ and $b$ do not share a common fixed point. This is possible by \eqref{eq:decompose_condition} and the fact that $K>0$.

If $a$ and $b$ generate a non virtually solvable group then we may write $$\lambda = \lambda' + \lambda(a) \left( \frac{1}{3} \delta_{a} + \frac{1}{3} \delta_{a}  + \frac{1}{3} \delta_{b}  \right) + \lambda(a) \left( \frac{1}{3} \delta_{a} + \frac{1}{3} \delta_{b}  + \frac{1}{3} \delta_{b}  \right)$$ where $\lambda'$ is a non-negative measure with smaller support that $\lambda$. We then apply the inductive hypothesis to $\lambda'$ with $\max \{K - 2 \lambda(a), 0 \}$ in the role of $K$ and $T- 2 \lambda(a)$ in the role of $T$.

If $a$ and $b$ generate a virtually solvable group then there must be two distinct points $g_1, g_2 \in \Gp$ such that the set $\{g_1, g_2 \}$ is stationary under both $a$ and $b$. If this is the case then choose some $c \in \supp \lambda$ such that $\{g_1, g_2 \}$ is not stationary under $c$. This is possible by \eqref{eq:decompose_condition}. Note that $a, b$ and $c$ generate a non virtually solvable group. Write
\begin{equation*}
\lambda = \lambda' + 3 \lambda(a) \left( \frac{1}{3} \delta_{a} + \frac{1}{3} \delta_{b}  + \frac{1}{3} \delta_{c}  \right).
\end{equation*}
We then apply the inductive hypothesis to $\lambda'$ with $\max \{K - 3 \lambda(a), 0 \}$ in the role of $K$ and $T- 3\lambda(a)$ in the role of $T$.
\end{proof}

\subsection{Symmetric and nearly symmetric examples}

The purpose of this subsection is to prove Corollary \ref{corollary:symmetric_examples}. We will do this using Theorem \ref{theorem:main_furstenberg}. First we need the following proposition.

\begin{proposition} \label{proposition:small_symmetric_is_non_degenerate}
For all $\alpha_0, c, A > 0$ there exists $t>0$ such that for all sufficiently small (depending on $\alpha_0$, $c$, and $A$) $r>0$ the following is true.

Suppose that $\mu$ is a compactly supported probability measure on $\Gp$ and that $U$ is a random variable taking values in $\A$ such that $\exp(U)$ has law $\mu$.  Suppose that $\lnb U \rnb \leq r$ almost surely and that $\lnb \mathbb{E}[U] \rnb \leq c r^2$. Suppose that the smallest eigenvalue of the covariance matrix of $U$ is at least $A r^2$. Then $\mu$ is $\alpha_0$, $t$ - non-degenerate.
\end{proposition}

This is enough to prove Corollary \ref{corollary:symmetric_examples}.

\begin{proof}[Proof of Corollary \ref{corollary:symmetric_examples}]
Note that by Proposition \ref{proposition:small_symmetric_is_non_degenerate} there is some $t>0$ such that providing $r$ is sufficiently small $\mu$ is $\frac{1}{4}$, $t$ - non-degenerate. Note that we can make $r$ arbitrarily small be choosing our $C$ to be arbitrarily large.

Note that by Proposition \ref{proposition:random_walk_entropy_lower_bound} 
\begin{equation*}
h_{RW} \geq \Theta(T).
\end{equation*}

Note that by Proposition \ref{proposition:m_mu_bound}
\begin{equation*}
M_{\mu} \leq 4^k M^{8k}.
\end{equation*}

Note that trivially
\begin{equation*}
\chi \leq O(r).
\end{equation*}

The result now follows from Theorem \ref{theorem:main_furstenberg}.
\end{proof}

In order to prove Proposition \ref{proposition:small_symmetric_is_non_degenerate} we first need the following result and a corollary of it.

\begin{theorem} \label{theorem:martingale_berry_essen}
For all $\gamma \in (1, \infty)$ there is some $L >0$ such that the following is true. Suppose that $X_1, X_2, \dots, X_n$ are random variables taking values in $\R$ and suppose that for each integer $i \in [1, n]$
\begin{equation*}
\mathbb{E}[X_i| X_1, X_2, \dots, X_{i-1}] = 0,
\end{equation*}
\begin{equation*}
\mathbb{E}[X_i^2| X_1, X_2, \dots, X_{i-1}] = 1,
\end{equation*}
and
\begin{equation*}
|X_i| \leq \gamma
\end{equation*}
almost surely. Then
\begin{equation*}
\sup_t \left| \Phi(t) - \mathbb{P}\left[ \frac{X_1 + X_2 + \dots + X_n}{\sqrt{n}} < t \right] \right| \leq L n^{-1/2} \log n
\end{equation*}
where
\begin{equation*}
\Phi(t) := \frac{1}{\sqrt{2 \pi}} \int_{-\infty}^{t} \exp ( -x^2 / 2) dx
\end{equation*}
is the c.d.f.\ of the standard normal distribution. 
\end{theorem}
\begin{proof}
This is a special case of \cite[Theorem 2]{BOLTHAUSEN_1982}.
\end{proof}

\begin{corollary}
For all $\varepsilon, \gamma > 0$ there exists $\delta >0$ and $N \in \N$ such that the following is true. Let $n \geq N$ and let $X_1, \dots, X_n$ be as in Theorem \ref{theorem:martingale_berry_essen} with this value of $\gamma$. Then for all $a \in \R$ we have
\begin{equation*}
\mathbb{P} \left[ \frac{X_1 + X_2 + \dots + X_n}{\sqrt{n}} \in [a, a + \delta] \right] \leq \varepsilon.
\end{equation*}
\end{corollary}

\begin{proof}
This follows immediately from Theorem \ref{theorem:martingale_berry_essen}.
\end{proof}

We will now prove Proposition \ref{proposition:small_symmetric_is_non_degenerate}.

\begin{proof}[Proof of Proposition \ref{proposition:small_symmetric_is_non_degenerate}]
To prove Proposition \ref{proposition:small_symmetric_is_non_degenerate} we will show that there is some $n$ such that for all $b_0 \in \B$ the measure $\mu^{*n} * \delta_{b_0}$ has mass at most $\alpha_0$ on any interval of length at most $t$. To do this, given an $n$-step random walk on $\B$ generated by $\mu$ we will construct an $n$-step random walk on $\R$. Specifically we have the following.

We let $n \in \N$ be some value we will choose later. Let $b_0 \in \B$ and let $\gamma_1, \gamma_2, \dots, \gamma_n$ be i.i.d.\ samples from $\mu$. Let $b_i := \gamma_i \gamma_{i-1} \dots \gamma_1 b_0$. Let $U_i := \log \gamma_i$ and define the real valued random variables $X_1, X_2, \dots, X_n$ by
\begin{equation*}
X_i := \left( \var \left[ \varrho_{b_{i-1}}(U) \right] \right)^{-1/2} \varrho_{b_{i-1}}(U_i)
\end{equation*}
where $\varrho_b \in \As$ is defined to be $D_u(\exp(u)b)|_{u=0}$ as in Definition \ref{definition:e_vec}. We let $Y_1, Y_2, \dots, Y_n$ be defined by
\begin{equation*}
Y_i = X_i - \mathbb{E}[X_i| X_1, X_2, \dots, X_{i-1}]
\end{equation*}
and let $S = Y_1 + Y_2 + \dots + Y_n$. 

Clearly $\mathbb{E}[Y_i|Y_1, Y_2, \dots, Y_{i-1}] = 0$ and $\mathbb{E}[Y_i^2|Y_1, Y_2, \dots, Y_{i-1}] = 1$. This enables us to apply Theorem \ref{theorem:martingale_berry_essen}. We now need to show that understanding $S$ gives us some information about the distribution of $b_n$.

Now let $c_1, c_2, \dots$ denote positive constants which depend only on $\alpha_0$, $c$, and $A$. We define $f: \R \to \R$ by
\begin{equation*}
f: x \mapsto \int_{0}^{x} \left( \var \left[ \varrho_{\phi^{-1}(u)}(U) \right] \right)^{-1/2} du.
\end{equation*}
This definition is chosen such that $f(\phi(b_{i})) - f(\phi(b_{i-1}))$ is approximated $X_i$. We will use this fact along with Theorem \ref{theorem:martingale_berry_essen} to show that there is some $n$ such that $f(b_n)$ can be approximated by a normal distribution.

We have
\begin{align*}
D_u f(\phi( \exp(u) b_{i-1})) |_{u=0} = \left( \var \left[ \varrho_{b_{i-1}}(U) \right] \right)^{-1/2} \varrho_{b_{i-1}}(U_i)
\end{align*}
and so $X_i = D_u f(\phi( \exp(u) b_{i-1})) |_{u=0}(U_i)$. This means that to bound $$|f(\phi(b_{i})) - f(\phi(b_{i-1})) - X_i|$$ it is sufficient to bound $\lnb D^2_u f(\phi( \exp(u) b_{i-1})) \rnb$ for $\lnb u \rnb \leq 1$. 

By compactness the norms of the first and second derivatives of the exponential function are bounded on the unit ball. Note that for all $u \in \R$
\begin{equation}
c_1^{-1} r^2 \leq \var \varrho_{\phi^{-1}(u)} (U) \leq c_1 r^2 \label{eq:var_size_bound}
\end{equation}
for some absolute constant $c_1 > 0$. Therefore
\begin{equation}
c_2^{-1} r^{-1} \leq f' \leq c_2 r^{-1} \label{eq:f_first_deriv_bound}
\end{equation}
for some absolute constant $c_2 >0$. Also note that $\var \varrho_{\phi^{-1}(u)} (U)$ can be written as
\begin{equation*}
\var \varrho_{\phi^{-1}(u)} (U) = v^T \Sigma v
\end{equation*}
where $\Sigma$ is the covariance matrix of $U$ and $v \in \R^3$ depends smoothly on $u$ and depends on nothing else. In particular
\begin{align*}
\left| \frac{d}{du} \var \varrho_{\phi^{-1}(u)} (U) \right| & = \left| v'(u)^T \Sigma v(u) + v(u)^T \Sigma v'(u) \right| \\
& \leq O(r^2).
\end{align*}
Note that
\begin{align*}
f''(x) & = \frac{d}{dx} \left( \var \varrho_{\phi^{-1}(x)} (U) \right)^{-1/2} \\
& =  \left( \var \rho_{\phi^{-1}(x)} (U) \right)^{-3/2} \left( \frac{d}{du} \var \rho_{\phi^{-1}(u)} (U) \right)
\end{align*}
and so in particular
\begin{equation}
|f''(x)| \leq O_A(r^{-1}). \label{eq:f_second_deriv_bound}
\end{equation}

In particular this means that whenever $\lnb u \rnb \leq 1$ we have
\begin{equation*}
\lnb D_u^2 f(\phi( \exp(u) b_{i-1})) \rnb \leq O_A( r^{-1} ).
\end{equation*}

Also note that there is some $M$ with $ M \cong_A r^{-1}$ such that for all $x \in \R$
\begin{equation*}
f(x + \pi) = f(x) + M.
\end{equation*}

Note that by \eqref{eq:f_second_deriv_bound} and Taylor's Theorem
\begin{equation*}
|f(\phi(b_i)) - f(\phi(b_{i-1})) - X_i| \leq O_A(r).
\end{equation*}
Note that by \eqref{eq:var_size_bound} and the conditions of the proposition
\begin{equation*}
|X_i - Y_i| = |\mathbb{E}[X_i]| \leq O_A(r).
\end{equation*}
Therefore
\begin{equation*}
|f(\phi(b_i)) - f(\phi(b_{i-1})) - Y_i| \leq O_A(r).
\end{equation*}
In particular
\begin{equation}
|f(\phi(b_n)) - f(\phi(b_0)) - S| \leq O_A(n r). \label{eq:n_gap}
\end{equation}

We now let $n = \floor{Kr^{-2}}$ where $K$ is some positive constant depending on $\alpha_0$, $A$, and $c$ which we will choose later. Choose $N \in \N$ and $T > 0$ such that by applying Theorem \ref{theorem:martingale_berry_essen} we may ensure that whenever $n \geq N$ and $a \in \R$ we have
\begin{equation*}
\mathbb{P} \left[\frac{S}{\sqrt{n}} \in [a, a+T] \right] \leq \frac{\alpha_0}{2}.
\end{equation*}
Note that
\begin{equation*}
\mathbb{E}[S^2] = n
\end{equation*}
and so
\begin{equation*}
\mathbb{P} \left[ |S| \geq \frac{M}{2} \right] \leq \frac{4 n}{M^2} \leq O_A(K).
\end{equation*}
Therefore whenever $n \geq N$ and $a \in \R$
\begin{equation*}
\mathbb{P} \left[S \in [a, a + T \sqrt{n}] + M \Z \right] \leq \frac{\alpha_0}{2} + O_A(K).
\end{equation*}
Substituting in our value for $n$ gives
\begin{equation*}
\mathbb{P} \left[S \in [a, a +  T \sqrt{K} r^{-1}] + M \Z \right] \leq \frac{\alpha_0}{2} + O_A(K).
\end{equation*}
From \eqref{eq:n_gap} we may deduce that
\begin{equation*}
\mathbb{P} \left[f(\phi(b_n))  \in [a, a + (c_{3} \sqrt{K} - c_{4} K) r^{-1}] + M \Z \right] \leq \frac{\alpha_0}{2} + c_{5} K
\end{equation*}
where $c_3, c_4$ and $c_5$ are positive constants depending only on $A, \alpha_0$ and $c$. By taking $K = \min \left \{ \frac{\alpha_0}{2 c_{3}}, \frac{c_{4}^2}{2 c_{5}^2} \right \}$ we get
\begin{equation*}
\mathbb{P} \left[f(\phi(b_n)) \in [a, a + c_{6} r^{-1}] + M \Z \right] \leq \alpha_0
\end{equation*}
for some positive constant $c_6$ depending only on $A, \alpha_0$ and $c$. By \eqref{eq:f_first_deriv_bound} this means that
\begin{equation*}
\mathbb{P} \left[\phi(b_n) \in [a, a + c_{7}] + \pi \Z \right] \leq \alpha_0
\end{equation*}
for some positive constant $c_6$ depending only on $A, \alpha_0$ and $c$ providing $n \geq N$. Noting that $n \to \infty$ as $r \to 0$ completes the proof.
\end{proof}

\subsection{Examples with rotational symmetry}

One way in which we can ensure that the Furstenberg measure satisfies our $\alpha_0, t$- non-degeneracy condition is to ensure that it has some kind of rotational symmetry. In particular we can prove the following corollary of Theorem \ref{theorem:main_furstenberg}.

\begin{corollary}
For every $a, b \in \N$ with $a \geq 4$ and $K > 0$ there exist some $C, \varepsilon >0$ such that the following is true.

Suppose that $x>C$. Suppose that $A_1, A_2, \dots, A_b \in \Gp$  have operator norms at most $1 + 1/x$ and have entries whose Mahler measures are at most $\exp( \exp(\varepsilon \sqrt{x}))$. Suppose further that the degree of the number field generated by the entries of the $A_i$ is at most $\exp( \varepsilon \sqrt{x} )$.

Let $R \in \Gp$ be a rotation by $\pi / a$ and let $\mu$ be defined by
\begin{equation*}
\mu := \frac{1}{ab} \sum_{i=0}^{a-1} \sum_{j=1}^{b} \delta_{R^{i} A_j R^{-i}}.
\end{equation*}

Suppose further that for every virtually solvable $H < \Gp$ we have $\mu(H) \leq 1 - K$.

Then the Furstenberg measure generated by $\mu$ is absolutely continuous.

\end{corollary}

\begin{proof}
We wish to apply Theorem \ref{theorem:main_furstenberg} to $\frac{1}{2} \mu + \frac{1}{2} \delta_{\id}$.

Note that this measure is clearly $\frac{1}{a}$, $\frac{\pi}{a}$- non-degenerate. Also note that we may assume that $C\geq 1$ and so take $R=2$ in Theorem \ref{theorem:main_furstenberg}. Clearly $\chi < \frac{1}{x}$. 

Note that by Proposition \ref{proposition:random_walk_entropy_lower_bound} we have $h_{RW} \geq \Theta(K)$. 

Note that by Proposition \ref{proposition:m_mu_bound} we know that $M_{\mu} \leq \exp (A \exp(\varepsilon x))$ where $A$ is some constant depending only on $a$ and $b$. The result now follows by Theorem \ref{theorem:main_furstenberg}.
\end{proof}

\subsection{Examples supported on large elements}
 
The purpose of this subsection is to prove Corollary \ref{corollary:large_example}. First we will need the following lemma.

\begin{lemma}[The Ping-Pong Lemma] \label{lemma:ping_pong}
Suppose that $G$ is a group which acts on a set $X$. Let $n \in \Z$ and suppose that we can find $g_1, g_2 , \dots, g_n \in G$ and pairwise disjoint non-empty sets $$A_1^+, A_2^+, \dots, A_n^+, A_1^-, A_2^- \dots, A_n^- \subset X$$ such that for all integers $i \in [1, n]$ and all $x \in X \backslash A_{i}^{-}$ we have $g_i x \in A_{i}^{+}$. Then $g_1, g_2, \dots, g_n$ freely generate a free  semi-group.
\end{lemma}

This lemma is well known and we will not prove it. From this we may deduce the following.

\begin{lemma} \label{lemma:apply_ping_pong}
For every $\varepsilon > 0$ there is some $C \leq O(\varepsilon^{-1})$ such that the following is true. Let $n \in \N$. Suppose that $\theta_1, \theta_2, \dots, \theta_{n} \in \R / \pi \Z$ and that for every $i \neq j$ we have $|\theta_i - \theta_j| \geq \varepsilon$ and $|\theta_i - \theta_j + \pi / 2| \geq \varepsilon$. Let $\lambda_1, \lambda_2, \dots \lambda_n$ be real numbers which are at least $C$. Then the set
\begin{equation*}
\left \{ R_{\theta_{i}} \begin{pmatrix} \lambda_i & 0\\ 0 & \lambda_i^{-1} \end{pmatrix} R_{-\theta_{i}} : i \in [1, n] \cap \Z \right \} \subset \Gp
\end{equation*}
freely generates a free semi-group.
\end{lemma}
\begin{proof}
This follows immediately by applying Lemma \ref{lemma:ping_pong} with $G = \Gp$, $X = \B$, $A_i^{+} = \phi^{-1}((\theta_i - \varepsilon / 2, \theta_i + \varepsilon / 2))$, and $A_i^{-} = \phi^{-1}((\theta_i - \varepsilon / 2, \theta_i + \varepsilon / 2))^{\perp}$ along with Lemma \ref{lemma:new_shape_b_simple_singular_value}.
\end{proof}

\begin{lemma} \label{lemma:choose_theta_n}
For all $n \in \Z$ there exists some $\theta_n \in \left( \frac{1}{2n}, \frac{2}{n} \right)$ such that $\sin \theta_n$ and $\cos \theta_n$ are rational and have height at most $4 n^2 + 1$.
\end{lemma}

\begin{proof}
Choose $\theta_n$ such that
\begin{equation*}
\sin \theta_n = \frac{4n}{4n^2+1}
\end{equation*}
and
\begin{equation*}
\cos \theta_n = \frac{4n^2 -1}{4n^2+1}.
\end{equation*}
\end{proof}

We are now ready to prove Corollary \ref{corollary:large_example}.
\begin{proof}[Proof of Corollary \ref{corollary:large_example}]
Given some $r > 0$ and some $n \in \Z$ define $\beta_0, \dots, \beta_{n-1} > 0$ by letting $\beta_k = \theta_{8^{n+1-k}}$ where $\theta_{\cdot}$ is as in Lemma \ref{lemma:choose_theta_n}. We then define $\alpha_0, \alpha_1, \dots, \alpha_{2^n-1} \geq 0$ by letting
\begin{equation*}
\alpha_k = \sum_{i=0}^{n-1} \xi_i^{(k)} \beta_i
\end{equation*}
where the $\xi_i^{(k)}$ are the binary expansion of $k$. In other words $k = \sum_{i=0}^{n-1} \xi_i^{(k)} 2^i$ with $\xi_i^{(k)} \in \{0, 1\}$. Clearly
\begin{equation*}
0 = \alpha_0 < \alpha_1 < \dots < \alpha_{2^n-1}.
\end{equation*}
Furthermore $\alpha_{i+1} > \alpha_i + \varepsilon$ where $\varepsilon = \frac{1}{2 \cdot 8^{n+1}}$. We also have that 
\begin{align*}
\alpha_{2^n-1} &< \frac{2}{8^2} + \frac{2}{8^3} + \frac{2}{8^4} + \dots \\
& = \frac{1}{32} \cdot \frac{8}{7}\\
& < \frac{\pi}{10} - \varepsilon.
\end{align*}

We now let $C$ be the $C$ from Lemma \ref{lemma:apply_ping_pong} with this value of $\varepsilon$ and we choose some prime number $p$ such that $p \geq C^2$, $p \leq O(8^{2n})$, and $X^2 - p$ is irreducible in the field $\mathbb{Q}[\sin \frac{\pi}{5}, \cos \frac{\pi}{5}]$. 

Now for $i = 0, 1, \dots, 2^n-1$ and $j = 0, 1, \dots, 4$ we let $g_{i, j}$ be defined by
\begin{equation*}
g_{i, j} := R_{\frac{j \pi}{5} + \alpha_i} \begin{pmatrix}
\ceil{r + \sqrt{p}} + \sqrt{p} & 0 \\ 0 & (\ceil{r + \sqrt{p}} + \sqrt{p})^{-1}
\end{pmatrix} R_{-\frac{j \pi}{5} - \alpha_i}.
\end{equation*}
By Lemma \ref{lemma:apply_ping_pong} we know that the $g_{i,j}$ freely generate a free semi-group. Now for $i = 0, 1, \dots, 2^n-1$ and $j = 0, 1, \dots, 4$ we let $\hat{g}_{i, j}$ be defined by
\begin{equation*}
\hat{g}_{i, j} := R_{\frac{j \pi}{5} + \alpha_i} \begin{pmatrix}
\ceil{r + \sqrt{p}} - \sqrt{p} & 0 \\ 0 & (\ceil{r + \sqrt{p}} - \sqrt{p})^{-1}
\end{pmatrix} R_{-\frac{j \pi}{5} - \alpha_i}.
\end{equation*}
Clearly the $\hat{g}_{i, j}$ are Galois conjugates of the $g_{i,j}$ and so also freely generate a free semi-group. We now let $\mu$ be defined by
\begin{equation*}
\mu = \sum_{i=0}^{2^n-1} \sum_{j=0}^{4} \frac{1}{5 \cdot 2^n} \delta_{\hat{g}_{i, j}}.
\end{equation*}
We wish to use Theorem \ref{theorem:main_furstenberg} to show that the Furstenberg measure generated by $\mu$ is absolutely continuous providing $n$ is sufficiently large in terms of $r$.

Let $\nu$ be the Furstenberg measure generated by $\mu$. By the construction of $\mu$ we know that $\nu$ is invariant under rotation by $\pi / 5$. In particular this means that it is $\frac{1}{5}$, $\frac{\pi}{5}$ - non-degenerate. We also know that for each $i, j$ we have $\lnb \hat{g}_{i, j} \rnb = \ceil{r + \sqrt{p}} - \sqrt{p} \leq r + 1$. This means that $\chi \leq r$ and that we may take $R = r+1$. Since the $\hat{g}_{i, j}$ freely generate a free semi-group we know that $h_{RW} = \log \left( 5 \cdot 2^n \right) \geq \Theta(n)$. Finally we need to bound $M_{\mu}$.

To bound the $M_{\mu}$ we will apply Proposition \ref{proposition:m_mu_bound}. We know by Lemma \ref{lemma:choose_theta_n} that the heights of the entries in the $\beta_i$ are at most $O(8^{2n})$. We also know that the height of $\ceil{r + \sqrt{p}} - \sqrt{p}$ is at most $O_r(\sqrt{p})$ which is at most $O_r(8^n)$. By Lemma \ref{lemma:height_poly_bound} this means that the height of entries in the $\hat{g}_{i, j}$ is at most $O_r(2^{2n} \cdot 8^{4n^2 + n} )$ which is at most $O_r(8^{5n^2})$. It is easy to show that $\left[ \mathbb{Q}[\sin \frac{\pi}{5}, \cos \frac{\pi}{5}] : \mathbb{Q} \right] = 4$. This means that by Proposition \ref{proposition:m_mu_bound} we have
\begin{equation*}
M_{\mu} \leq O_r \left( 8^{8 \cdot 4 \cdot 5 n^2} \right) \leq \exp( O_r(n^2) ).
\end{equation*}

Therefore
\begin{align*}
\frac{h_{RW}}{\chi} \left( \max \left \{1, \log \log \frac{M_{\mu}}{h_{RW}} \right \} \right)^{-2} &\gtrsim \frac{n}{r+1} \left( \log \log \exp(O_r(n^2)) \right)^{-2} \\
& \geq \frac{n}{O_r((\log n)^2)}\\
& \to \infty.
\end{align*}
This means that by Theorem \ref{theorem:main_furstenberg} the Furstenberg measure is absolutely continuous providing $n$ is sufficiently large in terms of $r$.
\end{proof}

\subsection{Examples with two generators}

In this subsection we will prove Corollary \ref{corollary:two_gens_ac}.

\begin{proof}[Proof of Corollary \ref{corollary:two_gens_ac}]
First we will show that $\mu$ is Zariski-dense. The compact subgroups of $\Gp$ are exactly those subgroups which are conjugate to the group of rotations. Since the rotations form a subgroup $A$ is only conjugate to a rotation under conjugation by another rotation and $B$ is not conjugate to a rotation under conjugation by a rotation. Therefore support of $\mu$ is not contained in any compact subgroup of $\Gp$. Since $A$ is an irrational rotation the orbit of any $b \in \B$ under $A$ is infinite. Therefore $\mu$ is strongly irreducible.

Next we will show that there is some $\alpha_0 \in \left(0, \frac{1}{3} \right)$ and $t > 0$ such that $\mu$ is $\alpha_0$, $t$ - non-degenerate for all sufficiently large $n$. 

First note that $A$ is a rotation by $\theta_n$ where $\theta_n = \frac{1}{n} + O(\frac{1}{n^2})$. Also note that for all $x \in \B$ we have $d(x, Bx) \leq O(n^{-3})$.

We now let $\tilde{A} : \R \to \R, x \mapsto x + \theta_n$ and choose $\tilde{B} : \R \to \R$ such that $\tilde{B}(x) \in \phi(B\phi^{-1}(x))$ and for all $x \in \R$ we have $|x - \tilde{B}(x)| \leq O(n^{-3})$. We then let $\tilde{\mu} = \frac{1}{2} \delta_{\tilde{A}} + \frac{1}{2} \delta_{\tilde{B}}$.

By Theorem \ref{theorem:srv_normal_wass} (a simple bound on the Wasserstein distance between a sum of independent random variables and a normal distribution) we know that for any $x \in \R$ we have
\begin{equation*}
\mathcal{W}_1 \left( \tilde{\mu}^{*n^2} * \delta_x, N(x + \frac{1}{2} n^2 \theta_n, n^2 \theta_n^2) \right) < O(n^{-1}).
\end{equation*}
Noting that $n^2 \theta_n^2 \to 1$ we can see that there is some $\alpha_0 \in \left(0, \frac{1}{3} \right)$ and $t > 0$ such that $\mu$ is $\alpha_0$, $t$ - non-degenerate for all sufficiently large $n$.

We will apply Theorem \ref{theorem:main_furstenberg} to $\frac{1}{2} \mu + \frac{1}{2} \delta_{\id}$. Note that this generates the same Furstenberg measure as $\mu$ and so in particular it is $\alpha_0$, $t$ - non-degenerate.

Note that by Proposition \ref{proposition:random_walk_entropy_lower_bound} there is some $\varepsilon > 0$ such that for all $n$ we have $h_{RW} \geq \varepsilon$.

Note that by Proposition \ref{proposition:m_mu_bound} we have $M_{\tilde{\mu}} \leq 4(n^3 + 1)^8$. Clearly we may take $R=2$. Also note that $\chi \leq n^{-3}$.

This means that to prove the corollary it is sufficient to prove that
\begin{equation*}
\varepsilon n^{3} \left( \log \log \frac{4 (n^3 + 1)^8}{\varepsilon} \right)^{-2}
\end{equation*}
tends to $\infty$ as $n \to \infty$. This is trivially true.
\end{proof}

\section{Appendix} \label{section:appendix}

\subsection{Proof of Theorem \ref{theorem:renewal_theorem}}

We extend the result of Kesten \cite[Theorem 1]{KESTEN_1974} to show that the convergence is uniform in the vector $v$.

\begin{theorem} \label{theorem:renewal_non_uniform}
Suppose that $\mu$ is a compactly supported Zariski-dense probability measure. Then there exists some probability measure measure $\hat{\nu}$ on $\B$ such that the following is true. Let $\gamma_1, \gamma_2, \dots$ be i.i.d.\ samples from $\mu$. Then given any $\varepsilon > 0$ and $v \in \B$ there exists some $T > 0$ such that given any $P > T$ we can find some random variable $x$ with law $\hat{\nu}$ such that
\begin{equation*}
\mathbb{P}[d((\gamma_1 \gamma_2 \dots \gamma_{\tau_{P, v}})^Tv, x) > \varepsilon] < \varepsilon.
\end{equation*}
\end{theorem}
Here $\tau_{P, v}$ is as in Definition \ref{definition:tau_t_v}. 
\begin{proof}
In \cite[Theorem 1]{KESTEN_1974} it is proven that this holds in a much more general setting providing some conditions are satisfied. In \cite[Section 4]{GUIVARCH_LEPAGE_2016} it is shown that the conditions of \cite[Theorem 1]{KESTEN_1974} are satisfied in this setting.
\end{proof}

We deduce uniform convergence from this fact. To do this we show that if $v, w \in \B$ are close then with high probability $\tau_{P, v} = \tau_{P, w}$ and $(\gamma_1 \gamma_2 \dots \gamma_{\tau_{P, v}})^Tv$ is close to $(\gamma_1 \gamma_2 \dots \gamma_{\tau_{P, v}})^Tw$.

\begin{lemma} \label{lemma:size_occurs_q_n}
Suppose that $\mu$ is a compactly supported Zariski-dense probability measure. Then given any $c_1, c_2 > 0$ there exists $T$ such that for any $P > T$ and any unit vector $b \in \R^2$
\begin{equation*}
\mathbb{P}[\exists n : \log P \leq \log \lnb (\gamma_1 \gamma_2 \dots \gamma_n)^T b \rnb \leq \log P + c_1] \lesssim c_1 / \chi + c_2. 
\end{equation*}
\end{lemma}
\begin{proof}
This follows immediately from \cite[Proposition 4.8]{LI_2018}.
\end{proof}

\begin{lemma} \label{lemma:taus_same}
Let $\mu$ be a finitely supported Zarisk- dense probability measure. Given $v \in \B$ and $P>0$ let $\tau_{P, v}$ be as in Definition \ref{definition:tau_t_v}. Then there exists some $\delta> 0$ depending on $\mu$ such that given any $r > 0$ for all sufficiently large (depending on $r$ and $\mu$) $P$ the following is true. Suppose that $v, w \in \B$ and $d(v, w) < r$. Then
\begin{equation*}
\mathbb{P}[\tau_{P, v} = \tau_{P, w}] \geq 1 - O_{\mu}(r^{\delta}).
\end{equation*}
\end{lemma}

\begin{proof}
Let $A$ be the event that
\begin{equation*}
d(v, b^{-}((\gamma_1 \gamma_2 \dots \gamma_n)^T)) > \sqrt{r}
\end{equation*}
and
\begin{equation*}
d(w, b^{-}((\gamma_1 \gamma_2 \dots \gamma_n)^T)) > \sqrt{r}
\end{equation*}
for all $n \geq \log P / \log R$. By \eqref{eq:products_not_too_close} from Lemma \ref{lemma:convergence_to_fm} we know that providing $P$ is sufficiently large in terms of $\mu$ and $r$ there is some $\delta > 0$ such that
\begin{equation*}
\mathbb{P}[A] \geq 1 - O_{\mu}(r^{\delta}).
\end{equation*}

Let $\hat{v}, \hat{w} \in \R^2$ be unit vectors which are representatives of $v$ and $w$ respectively. By Lemma \ref{lemma:simple_sl2_prod_size_bound} we know that there is some constant $C>0$ such that on the event $A$
\begin{equation*}
|\log \lnb (\gamma_1 \gamma_2 \dots \gamma_n)^T \hat{v} \rnb - \log \lnb (\gamma_1 \gamma_2 \dots \gamma_n)^T \hat{w} \rnb | < C r^{1/2}
\end{equation*}
for all $n \geq \log P / \log R$. Now let $B$ be the event that there exists $n$ such that
\begin{equation*}
|\log \lnb (\gamma_1 \gamma_2 \dots \gamma_n)^T \hat{v} \rnb - P | < 10 C r^{1/2}.
\end{equation*}
By Lemma \ref{lemma:size_occurs_q_n} we know that providing $P$ is sufficiently large in terms of $\mu$ and $r$, $\mathbb{P}[B] \leq O_{\mu}(r^{1/2})$. We also know that $\{ \tau_{P, v} = \tau_{P, w} \} \supset A \backslash B$. Therefore
\begin{equation*}
\mathbb{P}[\tau_{P, v} = \tau_{P, w}] \geq 1 - O_{\mu}(r^{\delta})
\end{equation*}
as required.
\end{proof}

\begin{proof}[Proof of Theorem \ref{theorem:renewal_theorem}]
Given $\varepsilon > 0$ we wish to show that we can find some $T$ (depending on $\mu$ and $\varepsilon$) such that whenever $P > T$ and $v \in \B$ we can find some random variable $x$ with law $\hat{\nu}$ such that
\begin{equation*}
\mathbb{P}[d(x, (\gamma_1 \gamma_2 \dots \gamma_{\tau_{P, v}})^Tv) > \varepsilon] < \varepsilon.
\end{equation*}

First let $\varepsilon > 0$. Choose $k \in \Z_{>0}$ and let $v_1, v_2, \dots, v_k \in \B$ be equally spaced. Let $T_1$ be the greatest of the $T$ from Theorem \ref{theorem:renewal_non_uniform} with $\frac{1}{10} \varepsilon$ in the role of $\varepsilon$ and $v_1, v_2, \dots, v_k$ in the role of $v$ and let $x_1, x_2, \dots, x_k$ be the $x$. Let $T_2$ be the $T$ from Lemma \ref{lemma:taus_same} with $r = \frac{\pi}{k}$. Let $T = \max \{ T_1, T_2 \}$. Thus whenever $t > T$ and $i \in [k]$
\begin{equation*}
\mathbb{P} \left[ d(x_i, (\gamma_1\gamma_2\dots \gamma_{\tau_{P, v_i}})^T v_i) > \frac{\varepsilon}{10} \right] < \frac{\varepsilon}{10}.
\end{equation*}

Now let $P > T$ and let $v \in \B$. Suppose without loss of generality that $v_1$ is the closest of the $v_i$ to $v$. In particular $d(v_1, w) < \frac{\pi}{k}$. By Lemma \ref{lemma:taus_same} this means that
\begin{equation}
\mathbb{P}[\tau_{P, v_1} = \tau_{P, v}] \geq 1 - O(k^{-\delta}) \label{eq:taus_same}
\end{equation}
for some $\delta > 0$ depending only on $\mu$.

We know by for example Lemma \ref{lemma:derivs} that providing
\begin{equation*}
d(b^{-1}((\gamma_1 \gamma_2 \dots \gamma_n)^T), v_1) > 100 k^{-1}
\end{equation*}
we have
\begin{equation*}
d((\gamma_1 \gamma_2 \dots \gamma_n)^T v_1, (\gamma_1 \gamma_2 \dots \gamma_n)^T v) < O_k(\lnb (\gamma_1 \gamma_2 \dots \gamma_n)^T \rnb^{-2}).
\end{equation*}
In particular by \eqref{eq:products_not_too_close} from Lemma \ref{lemma:convergence_to_fm} we know that
\begin{equation*}
\mathbb{P} \left[ d((\gamma_1 \gamma_2 \dots \gamma_{\tau_{P, v_1}})^Tv_1, (\gamma_1 \gamma_2 \dots \gamma_{\tau_{P, v_1}})^T v) < O_k(P^{-2}) \right] \geq 1 - O(k^{-\delta}).
\end{equation*}

Combining this with \eqref{eq:taus_same} we know that providing $P$ is sufficiently large depending on $k$ and $\mu$
\begin{equation*}
\mathbb{P} \left[ d((\gamma_1 \gamma_2 \dots \gamma_{\tau_{P, v_1}})^Tv_1, (\gamma_1 \gamma_2 \dots \gamma_{\tau_{P, v}})^T v) > O_k(P^{-2}) \right] < O(k^{-\delta}).
\end{equation*}
In particular this means that providing $P$ is sufficiently large depending on $k$ and $\mu$
\begin{equation*}
\mathbb{P} \left[ d(x_1, (\gamma_1 \gamma_2 \dots \gamma_{\tau_{P, v}})^T v) > \frac{1}{10} \varepsilon + O_k(P^{-2}) \right] < \frac{1}{10} \varepsilon + O(k^{-\delta})
\end{equation*}
and so if we choose $k$ large enough (depending on $\mu$ and $\varepsilon$) and then choose $P$ large enough (depending on $\mu$, $k$, and $\varepsilon$) then
\begin{equation*}
\mathbb{P} \left[ d((x_1, \gamma_1 \gamma_2 \dots \gamma_{\tau_{P, v}})^T v) >  \varepsilon  \right] < \varepsilon 
\end{equation*}
as required.

\end{proof}

We now wish to deduce Corollary \ref{corollary:renewal_theorem}. First we need the following Lemma.

\begin{lemma} \label{lemma:taus_same_g}
Let $\mu$ be a finitely supported Zariski-dense probability measure. Given $v \in \B$ let $\tau_{P, v}$ be as in Definition \ref{definition:tau_t_v} and given $a \in \Gp$ let $\tau_{P, a}$ be defined by
\begin{equation*}
\tau_{P, a} := \inf \{ n : \lnb a \gamma_1 \gamma_2 \dots \gamma_n \rnb \geq P \lnb a \rnb \}.
\end{equation*}
Then there exists some $\delta> 0$ depending on $\mu$ such that given any $r > 0$ for all sufficiently large (depending on $r$ and $\mu$) $P$ the following is true. Suppose that $v \in \B$, $a \in \Gp$ and $d(v, b^{-}(a)^{\perp}) < r$. Suppose that $a$ is sufficiently large (depending on $r$ and $\mu$). Then
\begin{equation*}
\mathbb{P}[\tau_{P, v} = \tau_{P, a}] \geq 1 - O_{\mu}(r^{\delta}).
\end{equation*}
\end{lemma}

\begin{proof}
This follows by a very similar proof to Lemma \ref{lemma:taus_same}. Let $A$ be the event that
\begin{equation*}
d(v, b^{-}((\gamma_1 \gamma_2 \dots \gamma_n)^T)) > \sqrt{r}
\end{equation*}
and
\begin{equation*}
d(b^{-}(a), b^{+}(\gamma_1 \gamma_2 \dots \gamma_n)) > \sqrt{r}
\end{equation*}
for all $n \geq \log P / \log R$. By \eqref{eq:products_not_too_close} from Lemma \ref{lemma:convergence_to_fm} we know that providing $P$ is sufficiently large in terms of $\mu$ and $r$ there is some $\delta > 0$ such that
\begin{equation*}
\mathbb{P}[A] \geq 1 - O_{\mu}(r^{\delta}).
\end{equation*}

Let $\hat{v} \in \R^2$ be a unit vector which is a representative of $v$. By Lemma \ref{lemma:simple_sl2_prod_size_bound} we know that there is some constant $C>0$ such that on the event $A$
\begin{equation*}
|\log \lnb (\gamma_1 \gamma_2 \dots \gamma_n)^T \hat{v} \rnb - \log \lnb a\gamma_1 \gamma_2 \dots \gamma_n \rnb + \log \lnb a \rnb | < C r^{1/2}
\end{equation*}
for all $n \geq \log P / \log R$. The result now follows by the same argument as Lemma \ref{lemma:taus_same}.
\end{proof}

We now prove Corollary \ref{corollary:renewal_theorem}.
\begin{proof}[Proof of Corollary \ref{corollary:renewal_theorem}]
Let $S$ be defined by
\begin{equation*}
S = \inf \{n : \lnb a \gamma_1 \gamma_2 \dots \gamma_n \rnb \geq \sqrt{P} \} 
\end{equation*}
let $\overline{a} = a \gamma_1 \gamma_2 \dots \gamma_S$ and let  $v = b^{-}(\overline{a})^{\perp}$. Let $\overline{S}$ be defined by $$ \overline{S} := \inf \{n \geq S: \lnb (\gamma_{S+1} \gamma_{S+2} \dots \gamma_n)^T \hat{v} \rnb \geq \frac{P}{\lnb a \gamma_1 \gamma_2 \dots \gamma_n \rnb} \lnb \hat{v} \rnb \}$$ where $\hat{v} \in \R^2 \backslash \{ 0 \}$ is a representative of $v$. Let $r > 0$ be arbitrarily small. By Lemma \ref{lemma:taus_same_g} providing $P$ is sufficiently large (in terms of $\mu$ and $r$) we have
\begin{equation*}
\mathbb{P}[\overline{S} = \tau_{a, P}] \geq 1 - O_{\mu}(r^{\delta_1})
\end{equation*}
for some $\delta_1>0$ depending only on $\mu$. Let $A$ be the event that for all $n \geq \frac{\log P}{2 \log R} - 1$ we have
\begin{equation*}
d(b^{+}(\gamma_{S+1} \gamma_{S+2} \dots \gamma_n), b^{-}(\overline{a})) > r.
\end{equation*}
By \eqref{eq:products_not_too_close} from Lemma \ref{lemma:convergence_to_fm} we know that $\mathbb{P}[A] \geq 1 - O_{\mu}(r^{\delta_2})$ for some $\delta_2>0$ depending only on $\mu$. By Lemmas \ref{lemma:mid_angles} and \ref{lemma:new_shape_b_simple_singular_value} we know that on the event $A$ providing $P$ is sufficiently large (in terms of $r$) we have
\begin{equation*}
d((\gamma_{S+1} \gamma_{S+2} \dots \gamma_{\overline{S}})^T v, b^{-}(\gamma_{S+1} \gamma_{S+2} \dots \gamma_{\overline{S}})^{\perp}) < r
\end{equation*}
and
\begin{equation*}
d(b^{-}(a \gamma_1 \gamma_2 \dots \gamma_{\tau_{P, a}}), b^{-}(\gamma_{S+1} \gamma_{S+2} \dots \gamma_{\tau_{P, a}})) < r.
\end{equation*}
In this means that on the event $A \cap \{ \tau_{P, a} = \overline{S}\}$ we have
\begin{equation*}
d(b^{-}(a \gamma_1 \gamma_2 \dots \gamma_{\tau_{P, a}})^{\perp}, (\gamma_{S+1} \gamma_{S+2} \dots \gamma_{\overline{S}})^T v) < 2r.
\end{equation*}
We are now done by Theorem \ref{theorem:renewal_theorem}.
\end{proof}

\section{Acknowledgements}

First of all I would like to thank my supervisor P\'eter  Varj\'u for his help and detailed comments in preparing this paper. I would also like to thank Emmanuel Breuillard, Constantin Kogler and Ioannis Kontoyiannis for their helpful comments which greatly improved the readability of this paper. Finally I would like to thank the anonymous referees for their helpful and detailed comments.

\bibliographystyle{plain}
\bibliography{references.bib}
\end{document}

%% file: preamble.tex
\usepackage{graphicx}
\usepackage{cite}
\usepackage{amsmath}
\usepackage{amsthm}
\usepackage{amssymb}
\usepackage{array}
\usepackage{graphicx}
\usepackage{mathtools}
\usepackage{amsmath,amssymb}
\usepackage{enumitem}

\newlist{condenum}{enumerate}{1} 
\setlist[condenum]{label=\bfseries A\arabic*., 
                   ref=A\arabic*, wide}
\DeclareMathOperator{\im}{Im}
\DeclareMathOperator{\supp}{supp}

\DeclareMathOperator{\Tr}{Tr}

\DeclareMathOperator{\var}{Var}
\DeclareMathOperator{\vart}{Var}
\DeclareMathOperator{\Psl}{PSL}
\DeclareMathOperator{\Sl}{SL}

\newtheorem {theorem}{Theorem}[section]
\newtheorem {lemma}[theorem]{Lemma}
\newtheorem {corollary}[theorem]{Corollary}
\newtheorem {proposition}[theorem]{Proposition}

\theoremstyle{definition}
\newtheorem {definition}[theorem]{Definition}

\newtheorem{remark}[theorem]{Remark}

\newcommand{\R}{\mathbb{R}}
\newcommand{\N}{\mathbb{Z}_{>0}}
\newcommand{\Z}{\mathbb{Z}}
\newcommand{\lnb}{\left\|}
\newcommand{\rnb}{\right\|}

\newcommand{\id}{\mathrm{Id}}

\newcommand{\Gp}{\Psl_2(\mathbb{R})}
\newcommand{\B}{P^1(\mathbb{R})}

\newcommand{\Bim}{\R / \pi \Z}
\newcommand{\floor}[1]{\left \lfloor #1 \right \rfloor}
\newcommand{\ceil}[1]{\left \lceil #1 \right \rceil}
\newcommand{\A}{\mathfrak{psl}_2(\mathbb{R})}
\newcommand{\As}{\mathfrak{psl}_2^*}
\newcommand{\s}{s}

%% file: main.bbl
\begin{thebibliography}{10}

\bibitem{Avez_1976}
Andr{\'e} Avez.
\newblock Croissance des groupes de type fini et fonctions harmoniques.
\newblock In Jean-Pierre Conze and Michael~S. Keane, editors, {\em Th{\'e}orie Ergodique}, pages 35--49, Berlin, Heidelberg, 1976. Springer Berlin Heidelberg.

\bibitem{BARANY_POLLICOTT_SIMON_2012}
B.~B\'{a}r\'{a}ny, M.~Pollicott, and K.~Simon.
\newblock Stationary measures for projective transformations: the {B}lackwell and {F}urstenberg measures.
\newblock {\em J. Stat. Phys.}, 148(3):393--421, 2012.

\bibitem{BenoistQuint2016}
Yves Benoist and Jean-François Quint.
\newblock {\em Random Walks on Reductive Groups}, volume~62 of {\em Ergebnisse der Mathematik und ihrer Grenzgebiete. 3. Folge / A Series of Modern Surveys in Mathematics}.
\newblock Springer, Cham, 2016.

\bibitem{BenoistQuint2018}
Yves Benoist and Jean-François Quint.
\newblock Stationary measures and uniform distribution on homogeneous spaces.
\newblock {\em Israel Journal of Mathematics}, 226(1):1--14, 2018.

\bibitem{bochi2019anosov}
Jairo Bochi, Rafael Potrie, and Andr{\'e}s Sambarino.
\newblock Anosov representations and dominated splittings.
\newblock {\em J. Eur. Math. Soc.}, 21(11):3343--3414, 2019.

\bibitem{BOLTHAUSEN_1982}
E.~Bolthausen.
\newblock {Exact Convergence Rates in Some Martingale Central Limit Theorems}.
\newblock {\em The Annals of Probability}, 10(3):672 -- 688, 1982.

\bibitem{BOUGEROL_LACROIX_1985}
Philippe Bougerol and Jean Lacroix.
\newblock {\em Products of random matrices with applications to {S}chr\"{o}dinger operators}, volume~8 of {\em Progress in Probability and Statistics}.
\newblock Birkh\"{a}user Boston, Inc., Boston, MA, 1985.

\bibitem{BOURGAIN_2012}
Jean Bourgain.
\newblock Finitely supported measures on {$SL_2(\Bbb R)$} which are absolutely continuous at infinity.
\newblock In {\em Geometric aspects of functional analysis}, volume 2050 of {\em Lecture Notes in Math.}, pages 133--141. Springer, Heidelberg, 2012.

\bibitem{BOUTONNET_IOANA_GOLSEFIDY_2017}
R\'{e}mi Boutonnet, Adrian Ioana, and Alireza~Salehi Golsefidy.
\newblock Local spectral gap in simple {L}ie groups and applications.
\newblock {\em Invent. Math.}, 208(3):715--802, 2017.

\bibitem{Breuillard_2015}
E~Breuillard.
\newblock Lecture notes in ``masterclass on groups, boundary actions and group $c*$-algebras, copenhagen.'', 2015.

\bibitem{BREUILLARD_2008}
Emmanuel Breuillard.
\newblock A strong tits alternative.
\newblock {\em arXiv preprint arXiv:0804.1395}, 2008.

\bibitem{VARJU_BERLIAND_2020}
Emmanuel Breuillard and P\'{e}ter~P. Varj\'{u}.
\newblock Entropy of {B}ernoulli convolutions and uniform exponential growth for linear groups.
\newblock {\em J. Anal. Math.}, 140, 2020.

\bibitem{COVER_THOMAS_2006}
Thomas~M. Cover and Joy~A. Thomas.
\newblock {\em Elements of information theory}.
\newblock Wiley-Interscience [John Wiley \& Sons], Hoboken, NJ, second edition, 2006.

\bibitem{EINSIEDLER_WARD_2010}
Manfred Einsiedler and Thomas Ward.
\newblock {\em Ergodic theory with a view towards number theory}, volume 259 of {\em Graduate Texts in Mathematics}.
\newblock Springer-Verlag London, Ltd., London, 2011.

\bibitem{Erickson_1973}
R.~V. Erickson.
\newblock On an {$L\sb{p}$} version of the {B}erry-{E}sseen theorem for independent and {$m$}-dependent variables.
\newblock {\em Ann. Probability}, 1:497--503, 1973.

\bibitem{FURSTENBERG_KESTEN_1960}
H.~Furstenberg and H.~Kesten.
\newblock {Products of Random Matrices}.
\newblock {\em The Annals of Mathematical Statistics}, 31(2):457 -- 469, 1960.

\bibitem{FURSTENBERG_KIFER_1983}
H.~Furstenberg and Y.~Kifer.
\newblock Random matrix products and measures on projective spaces.
\newblock {\em Israel J. Math.}, 46(1-2):12--32, 1983.

\bibitem{Furstenberg1971}
Harry Furstenberg.
\newblock Random walks and discrete subgroups of lie groups.
\newblock In Peter~E. Ney, editor, {\em Advances in Probability and Related Topics}, volume~1, pages 1--63. Marcel Dekker, 1971.

\bibitem{GUIVARCH_LEPAGE_2016}
Y.~Guivarc'h and \'{E}. Le~Page.
\newblock Spectral gap properties for linear random walks and {P}areto's asymptotics for affine stochastic recursions.
\newblock {\em Ann. Inst. Henri Poincar\'{e} Probab. Stat.}, 52(2):503--574, 2016.

\bibitem{guivarc1990produits}
Yves Guivarc'h.
\newblock Produits de matrices al{\'e}atoires et applications aux propri{\'e}t{\'e}s g{\'e}om{\'e}triques des sous-groupes du groupe lin{\'e}aire.
\newblock {\em Ergodic theory and dynamical systems}, 10(3):483--512, 1990.

\bibitem{HOCHMAN_2014}
Michael Hochman.
\newblock On self-similar sets with overlaps and inverse theorems for entropy.
\newblock {\em Ann. of Math. (2)}, 180(2):773--822, 2014.

\bibitem{HOCHMAN_SOLOMYAK_2017}
Michael Hochman and Boris Solomyak.
\newblock On the dimension of {F}urstenberg measure for {$SL_2(\Bbb R)$} random matrix products.
\newblock {\em Invent. Math.}, 210(3):815--875, 2017.

\bibitem{Johnson_2004}
Oliver~T. Johnson.
\newblock {\em Information Theory And The Central Limit Theorem}.
\newblock World Scientific Publishing Company, 2004.

\bibitem{KAIMANOVICH_LE_PRINCE_2011}
Vadim~A. Kaimanovich and Vincent Le~Prince.
\newblock Matrix random products with singular harmonic measure.
\newblock {\em Geom. Dedicata}, 150:257--279, 2011.

\bibitem{KESTEN_1959}
Harry Kesten.
\newblock Symmetric random walks on groups.
\newblock {\em Transactions of the American Mathematical Society}, 92(2):336--354, 1959.

\bibitem{KESTEN_1974}
Harry Kesten.
\newblock Renewal theory for functionals of a {M}arkov chain with general state space.
\newblock {\em Ann. Probability}, 2:355--386, 1974.

\bibitem{Kittle2024ASENS}
Samuel Kittle.
\newblock Absolutely continuous self-similar measures with exponential separation.
\newblock {\em Annales Scientifiques de l'{\'E}cole Normale Sup{\'e}rieure}, 57(4):1191--1231, 2024.

\bibitem{KLENKE_2014}
Achim Klenke.
\newblock {\em Probability theory}.
\newblock Universitext. Springer, London, second edition, 2014.
\newblock A comprehensive course.

\bibitem{KOGLER_2022}
Constantin Kogler.
\newblock Local limit theorem for random walks on symmetric spaces, 2022.

\bibitem{LANG_1985}
Serge Lang.
\newblock {\em {${\rm SL}_{2}({\bf R})$}}.
\newblock Addison-Wesley Publishing Co., Reading, Mass.-London-Amsterdam, 1975.

\bibitem{LEQUEN_2022}
Félix Lequen.
\newblock Absolutely continuous furstenberg measures for finitely-supported random walks, 2022.

\bibitem{LI_2018}
Jialun Li.
\newblock Decrease of {F}ourier coefficients of stationary measures.
\newblock {\em Math. Ann.}, 372(3-4):1189--1238, 2018.

\bibitem{MASSER_2016}
David Masser.
\newblock {\em Auxiliary polynomials in number theory}, volume 207.
\newblock Cambridge University Press, 2016.

\bibitem{Sert_2017}
Cargi Sert.
\newblock {\em Joint spectrum and large deviation principles for random products of matrice}.
\newblock PhD thesis, Universit\'e Paris-Saclay, 2017.

\bibitem{SHMERKIN_2019}
Pablo Shmerkin.
\newblock On {F}urstenberg's intersection conjecture, self-similar measures, and the {$L^q$} norms of convolutions.
\newblock {\em Ann. of Math. (2)}, 189(2):319--391, 2019.

\bibitem{SOLOMYAK_1995}
Boris Solomyak.
\newblock On the random series {$\sum\pm\lambda^n$} (an {E}rd{\H{o}}s problem).
\newblock {\em Ann. of Math. (2)}, 142(3):611--625, 1995.

\bibitem{VARJU_2019}
P\'{e}ter~P. Varj\'{u}.
\newblock Absolute continuity of {B}ernoulli convolutions for algebraic parameters.
\newblock {\em J. Amer. Math. Soc.}, 32(2):351--397, 2019.

\bibitem{VIGNEAUX_2021}
Juan~Pablo Vigneaux.
\newblock Entropy under disintegrations.
\newblock In {\em Geometric science of information}, volume 12829 of {\em Lecture Notes in Comput. Sci.}, pages 340--349. Springer, Cham, 2021.

\end{thebibliography}
